\documentclass[a4paper,10pt]{article}

\usepackage{amsmath,amssymb,amsthm,amscd}
\usepackage[mathscr]{eucal}
\usepackage[all]{xy}
\usepackage{graphicx}

\makeatletter
  
  \@addtoreset{equation}{section}
\makeatother

\def\bn{\bigskip\noindent}

\newcommand{\plim}{\varprojlim}
\newcommand{\mcal}{\mathcal}
\newcommand{\mbf}{\mathbf}

\newcommand{\mfrak}{\mathfrak}
\newcommand{\mbb}{\mathbb}
\newcommand{\mrm}{\mathrm}
\newcommand{\mfS}{\mathfrak{S}}
\newcommand{\mfO}{\mathfrak{O}}
\newcommand{\vphi}{\varphi}
\newcommand{\mfL}{\mathfrak{L}}
\newcommand{\mfM}{\mathfrak{M}}
\newcommand{\mfN}{\mathfrak{N}}

\newcommand{\mft}{\mathfrak{t}}
\newcommand{\mfm}{\mathfrak{m}}
\newcommand{\mfn}{\mathfrak{n}}
\newcommand{\ku}{k[\![u]\!]}
\newcommand{\wh}{\widehat}
\newcommand{\whR}{\widehat{\mathcal{R}}}

\newcommand{\Mod}{\mrm{Mod}}

\newcommand{\cO}{\mathcal{O}}
\newcommand{\E}{\mathcal{E}}
\newcommand{\Eur}{\mathcal{E}^{\mathrm{ur}}}

\newcommand{\cOur}{\mathcal{O}^{\mathrm{ur}}}

\newtheorem{theorem}{Theorem}[section]
\newtheorem{corollary}[theorem]{Corollary}
\newtheorem{lemma}[theorem]{Lemma}
\newtheorem{proposition}[theorem]{Proposition}

\theoremstyle{definition}
\newtheorem{definition}[theorem]{Definition}
\newtheorem{remark}[theorem]{Remark}

\newtheorem{question}[theorem]{Question}

\title{Lattices in crystalline representations and Kisin modules 
associated with 
iterate extensions}

\author{Yoshiyasu Ozeki\footnote{
Kanagawa University, Kanagawa 259-1293, JAPAN.
\endgraf
e-mail: {\tt ozeki@kanagawa-u.ac.jp}}
}

\begin{document}
\maketitle

\begin{abstract}
Cais and Liu extended the theory of Kisin modules and crystalline representations to
allow more general coefficient fields and lifts of Frobenius.
Based on their theory,
we classify lattices in crystalline representations by Kisin modules with additional structures
under a Cais-Liu's setting. Furthermore, we give a geometric interpretation 
of Kisin modules of height one  in terms of  Dieudonn\'e crystals of $p$-divisible groups,
and show a full faithfulness theorem for a restriction functor on torsion crystalline representations.
\end{abstract}

\tableofcontents


\section{Introduction}

Let $K$ be a complete discrete valuation field of mixed characteristics $(0,p)$
with perfect residue field $k$.
Let $\overline{K}$ be an algebraic closure of $K$ and 
$G:=\mrm{Gal}(\overline{K}/K)$ the absolute Galois group of $K$.
Let $e$ be the absolute ramification index of $K$ and $r\ge 0$ an integer.
It is known that to classify $G$-stable lattices in semi-stable or crystalline representations by some linear 
data is one of the powerful tools for studies of various interesting problems such as 
Langlands correspondence.
For this, the theory of {\it Kisin modules}, 
provided in \cite{Kis}, is very useful. 
Based on  Kisin's theory,
Liu \cite{Li2} constructed a theory of {\it $(\vphi,\hat{G})$-modules},
which gives a categorical equivalence between them and 
a category of  $G$-stable lattices in semi-stable representations
with certain Hodge-Tate weights.
One of the advantages of Liu's  theory is that 
there are no  restriction on $e$ and $r$ in his theory.
Throughout Kisin and Liu's theory,
the non-Galois ``Kummer'' extension $K_{\infty}/K$,
obtained by adjoining a given compatible system of $p$-power roots of 
a uniformizer of $K$, plays a central role.
Recently, Cais and Liu \cite{CL} generalized Kisin's theory 
to the setting of many {\it $f$-iterate extension} $K_{\underline{\pi}}/K$.
Here,  the $f$-iterate extension $K_{\underline{\pi}}/K$
that we consider
is defined as follows.
Let $f(u)=u^p+a_{p-1}u^{p-1}+\cdots +a_1u\in \mbb{Z}_p[u]$
such that $f(u)\equiv u^p\ \mrm{mod}\ p\mbb{Z}_p[u]$.
We fix the choice of  a uniformizer $\pi_0=\pi$ of $K$
and $\{\pi_n\}_{n\ge 0}$ such that $f(\pi_{n+1})=\pi_n$.
Then we set $K_{\underline{\pi}}:=\bigcup_{n\ge 0}K(\pi_n)$.
Thus Kisin's theory is the case where $f(u)=u^p$.

The aim of this paper is
to establish the theory of ``crystalline'' $(\vphi,\hat{G})$-modules under 
the Cais-Liu's setting, and apply it to a study of torsion crystalline representations. 
In Section \ref{MainSection}, follwing \cite{Li2},
we define a notion of $(\vphi,\hat{G})$-modules of height $r$.
We show in Theorem \ref{MT} that, 
under some mild assumptions, there exists an anti-equivalence between 
the category of $(\vphi,\hat{G})$-modules of height $r$ (with an additional condition) 
 and the category of  $G$-stable lattices in crystalline
 $\mbb{Q}_p$-representations 
with  Hodge-Tate weights in $[0,r]$.

As a consequence of our arguments,
we can prove a full faithfulness theorem on 
torsion crystalline representations. To give a statement,
we need some more notation.
Let $f(u)=\prod^n_{i=1}f_i(u)$ be the irreducible decomposition
of $f(u)$ in $W(\bar{k})[u]$
with the property that $f_1(u),\dots ,f_m(u)$ are of degree $\le e$
and $f_{m+1}(u),\dots ,f_n(u)$ are of degree $> e$.
We denote by $n_f$ the degree of  $\prod^m_{i=1}f_i(u)$.
For example, we have $n_f=p$ if $f(u)$ is of the form  
$u^p+a_{p-1}u^{p-1}$ with some $a_{p-1}\in p\mbb{Z}_p$.
Let $\mrm{Rep}^{r,\mrm{cris}}_{\mrm{tor}}(G)$
be the category of torsion crystalline representations of $G$
with Hodge-Tate weights in $[0,r]$.
Here, a torsion $\mbb{Z}_p$-representation $T$
is {\it torsion crystalline with Hodge-Tate weights in $[0,r]$}
 if $T$ is a quotient of 
lattices in a crystalline $\mbb{Q}_p$-representation 
with Hodge-Tate weights in $[0,r]$.
It is well-known that the condition that $T$ is torsion crystalline with Hodge-Tate weights in $[0,1]$
is equivalent  to the condition that $T$ is flat in the sense that 
$T$ is of the form $G(\overline{K})$ where $G$ is a finite flat group scheme over $\cO_{\overline{K}}$
killed by some power of $p$.
The theorem below is a torsion analogue of Theorem 1.0.2 of \cite{CL}.

\begin{theorem}[{= Theorems \ref{FFT:r=1} and \ref{FFT}}] 
Under some technical assumptions 
$($see Theorems \ref{FFT:r=1} and \ref{FFT} for details$)$, 
the restriction functor $\mrm{Rep}^{r,\mrm{cris}}_{\mrm{tor}}(G)\to 
\mrm{Rep}_{\mrm{tor}}(G_{\underline{\pi}})$
is fully faithful if $e(r-1)<n_f(p-1)/p$.
\end{theorem}

\noindent
In the case $f(u)=u^p$, this is Theorem 1.2 of \cite{Oz2}.
In this case, previous results have been given by 
some mathematicians.
The theorem was first studied by 
Breuil for  $e=1$ and $r<p-1$ via the Fontaine-Laffaille theory 
(\cite{Br1}, the proof of Th\'eor\`em 5.2).
He also proved the theorem for $p>2$ and $r\le 1$ 
as a consequence of a study of the category of finite flat group schemes (\cite[Theorem 3.4.3]{Br2}).
Later, his result was extended to the case $p=2$ in \cite{Kim}, \cite{La}, \cite{Li3} 
(proved independently). 
Based on studies of  ramification bounds for torsion crystalline representations,
Abrashkin proved the theorem
in the case $[K:\mbb{Q}_p]<\infty$, $e=1$, $p>2$ and $r<p$ (\cite[Section 8.3.3]{Ab}).

On the other hand, our arguments give an affirmative answer to a conjecture 
suggested in \cite[Remark 5.2.3 and Section 6.3]{CL} (in the case where ``$F=\mbb{Q}_p$'').
Let $T$ be a $G$-stable lattice in a crystalline
 $\mbb{Q}_p$-representation
with  Hodge-Tate weights in $[0,r]$. 
Cais-Liu constructed  a  Kisin module $\mfM$
which corresponds to $T|_{G_{\underline{\pi}}}$, 
where $G_{\underline{\pi}}$ is the absolute Galois group of $K_{\underline{\pi}}$.
This Kisin module $\mfM$ depends on the choice of  ($f(u)$, $(\pi_n)_{n\ge 0}$).
If we select another choice of ($f'(u)$, $(\pi'_n)_{n\ge 0}$),
then we obtain a  different Kisin module 
$\mfM'$. 
It seems natural to ask for the relationship between $\mfM$ and $\mfM'$.
For this, we show
\begin{theorem}[= Corollary \ref{comparisoncor}  and Theorem \ref{comparison:r=1}]
Let the notation be as above.
Assume $v_p(a_1)>\mrm{max}\{r,1\}$.
Furthermore,  we assume the condition 
$(P)$ $($cf.\ Section \ref{MainSection}$)$ if $r>1$.
Then the Kisin modules $\mfM$ and $\mfM'$ become isomorphic 
after base change to $W(R)$.
\end{theorem}

\noindent
Now we consider the case $r=1$.
In this case, Cais-Liu showed in \cite[Theorem 5.0.10]{CL} that 
there exists an anti-equivalence of categories between the category 
of Kisin modules of height $1$ and the category of $p$-divisible groups
over the ring of integers $\cO_K$ of $K$.
On the other hand, in the classical Kisin's setting $f(u)=u^p$, 
relationships between Kisin modules of height $1$ and 
Dieudonn\'e crystals are well-studied (cf.\ \cite{Kis}). 
Combining these facts with the above theorem,
we obtain a geometric interpretation of Kisin modules of height $1$ 
for the Cais-Liu's setting.
\begin{corollary}[= Corollary \ref{Dieudonne}]
Assume  $v_p(a_1)>1$.
Let $H$ be a $p$-divisible group over $\cO_K$
and $\mbb{D}(H)$ be the Dieudonn\'e crystal attached to $H$.
Let $\mfM$ be the Kisin module attached to $H$.
Then there is a functorial isomorphism
$A_{\mrm{cris}}\otimes_{\mfS} \vphi^{\ast}\mfM
\simeq \mbb{D}(H)(A_{\mrm{cris}})$.
\end{corollary}

%

\vspace{5mm}
\noindent
{\bf Notation :}
For any topological group $H$,
a free $\mbb{Z}_p$-representation of $H$ 
(resp.\  a $\mbb{Q}_p$-representation of $H$)
is a finitely generated free $\mbb{Z}_p$-module equipped 
with a continuous $\mbb{Z}_p$-linear $H$-action
(resp.\ a finite dimensional $\mbb{Q}_p$-vector space equipped 
with a continuous $\mbb{Q}_p$-linear $H$-action).
We denote by $\mrm{Rep}_{\mbb{Z}_p}(H)$
(resp.\ $\mrm{Rep}_{\mbb{Q}_p}(H)$)
the category of them.

For any ring extension $A\subset B$ and any $A$-linear morphism of $A$-modules
$f\colon M\to N$,
we often abuse notations by writing $f\colon B\otimes_A M\to B\otimes_A N$
for the $B$-linear extension of $f$.


\section{Preliminary}

In this section, we define some basic notation, 
and we recall some results on
iterate extensions given in \cite{CL}.
A lot of arguments in this section are deeply depending on
\cite[Sections 2 and 3]{Li1}. It will be helpful for the reader 
to refer this reference.

\subsection{Basic notation}
Let $p\ge 2$ be a prime number.
Let $K$ be a complete discrete valuation field of mixed characteristics $(0,p)$ 
with perfect residue field $k$.
We denote by $e$ the absolute ramification index of $K$.
Let $\overline{K}$ be an algebraic closure of $K$
and $\cO_{\overline{K}}$ the integer ring of $\overline{K}$.
We denote by $v_p$ the valuation of $\overline{K}$
normalized by $v_p(p)=1$.
We set $G:=\mrm{Gal}(\overline{K}/K)$,
the absolute Galois group of $K$.
We denote by $K_0$ the field $W(k)[1/p]$,
which is the maximal absolutely unramified subfield of $K$.

We fix a uniformizer $\pi$ of $K$ and 
fix the choice of a system $(\pi_n)_{n\ge 0}$, where
$\pi_0=\pi$ and 
$f(\pi_{n+1})=\pi_n$ for any $n\ge 0$.
We also fix a polynomial  $f(u)=\sum^p_{i=1}a_iu^i =u^p+a_{p-1}u^{p-1}+\cdots +a_1u\in \mbb{Z}_p[u]$
which satisfies $f(u)\equiv u^p\ \mrm{mod}\ p$.
By an easy computation of the Newton polygon of $f(u)-\pi_{n-1}$, 
we see that $v_p(\pi_n)=1/(ep^n)$ for any $n\ge 0$. 
We denote by $E(u)$ the minimal polynomial of $\pi$
over $K_0$.

Let $R=\plim \cO_{\overline{K}}/p$, 
where the transition maps are 
given by the $p$-th power map.
This is a complete discrete valuation field with residue field $\overline{k}$.
Let $v_R$ be a valuation of $R$ given by 
$v_R(x):=\lim_{n\to \infty}v_p(\hat{x}_n^{p^n})$
for $x=(x_n)_{n\ge 0}\in R$, where $\hat{x}_n\in \cO_{\overline{K}}$ is any lift of $x_n$.
Let $\mfm_R$ be the maximal ideal of $R$ and set
$\mfm^{\ge c}_R:=\{x\in  R \mid v_R(x)\ge c\}$
for any real number $c\ge 0$.
We set $\underline{\pi}:=(\pi_n\ \mrm{mod}\ p\cO_{\overline{K}})_{n\ge 0}\in R$.
Note that $v_R(\underline{\pi})=1/e$.
By \cite[Lemma 2.2.1]{CL}, there exists a unique set-theoretic section
$\{ \cdot \}_f\colon R\to W(R)$ to the reduction modulo $p$
which satisfies $\vphi(\{x\}_f)=f(\{x\}_f)$ for all $x\in R$.
The embedding 
$W(k)[u]\hookrightarrow W(R)$, given by $u\mapsto \{\underline{\pi}\}_f$,
extends to a unique $W(k)$-algebra embedding 
$\mfS:=W(k)[\![u]\!]\hookrightarrow W(R)$.
By this embedding, we identify $\mfS$ with a $\vphi$-stable $W(k)$-subalgebra
of $W(R)$.
Let $\cO_{\E}$ be the $p$-adic completion of $\mfS[1/u]$.
This is a complete discrete valuation ring with residue field $k(\!(u)\!)$.
Note that $p$ is a uniformizer of $\cO_{\E}$.
Let $\E$ be the fraction field of $\cO_{\E}$.
Then the embedding 
$\mfS\hookrightarrow W(R)$ extends to
$\cO_{\E}\hookrightarrow W(\mrm{Fr}R)$ and 
$\E \hookrightarrow W(\mrm{Fr}R)[1/p]$.
We denote by $\Eur$ the $p$-adic completion of the maximal 
unramified algebraic extension of $\E$, and denote by $\cOur$
the integer ring of $\Eur$. 
We may regard $\Eur$ and  $\cOur$ as $\vphi$-stable 
subrings of $W(\mrm{Fr}R)[1/p]$ and $W(\mrm{Fr}R)$, respectively.
We put $\mfS^{\mrm{ur}}=\cO^{\mrm{ur}}\cap W(R)$.

We set $K_{\underline{\pi}}:=\bigcup_{n\ge 0}K(\pi_n)$
and denote by $G_{\underline{\pi}}$ the absolute Galois group of 
$K_{\underline{\pi}}$.
The extension $K_{\underline{\pi}}/K$ 
is totally wildly ramified. 
Furthermore,  it is shown in \cite[Lemmas 3.1.1 and 3.2.1]{CL} that 
the extension $K_{\underline{\pi}}/K$ is strictly APF in the sense of \cite{Wi},
and the $G_{\underline{\pi}}$-action on $R$ induces an isomorphism
$G_{\underline{\pi}}\simeq G_{k(\!(\underline{\pi})\!)}=G_{{k(\!(u)\!)}}$.
Note that $G_{\underline{\pi}}$-action on $W(\mrm{Fr}R)[1/p]$
preserves  $\Eur$ and  $\cOur$, and 
$G_{\underline{\pi}}$ acts on $\E$ and $\cO_{\E}$ trivial.

Let $\nu \colon W(R)\twoheadrightarrow W(\overline{k})$
be the canonical projection induced by the projection
$R\twoheadrightarrow \overline{k}$.
For any subring $A$ of $B^+_{\mrm{cris}}$,
we set $\mrm{Fil}^iA:=A\cap \mrm{Fil}^iB^+_{\mrm{cris}}$.
We also set 
\begin{align*}
&I_+A:=A\cap \mrm{ker}\ \nu \quad \mrm{and}\\
&I^{[1]}A:=\{x\in A\mid \vphi(x)\in \mrm{Fil}^1A\ \mrm{for\ any}\ n\ge 0 \}.
\end{align*}
Note that we have $I_+A\supset I^{[1]}A$.

\subsection{\'Etale $\vphi$-modules and Kisin modules}
\label{EKmod}

Let $\Mod_{\cO_{\E}}$ (resp.\ $\Mod_{\cO_{\E,\infty}}$)
be the category of finite free $\vphi$-modules $M$ over $\cO_{\E}$
(resp.\ of finite type $\vphi$-modules $M$ over $\cO_{\E}$ killed by a power of $p$)
whose $\cO_{\E}$-linearization 
$1\otimes \vphi\colon \cO_{\E}\otimes_{\vphi,\cO_{\E}}M\to M$
is an isomorphism.
We call objects of these categories {\it \'etale $\vphi$-modules}.

We define  a $\mbb{Z}_p$-representation of $G_{\underline{\pi}}$ 
for any \'etale $\vphi$-module $M$ by
\begin{align*}
T_{\cO_{\E}}(M):=\left\{ 
\begin{array}{cr}
\mrm{Hom}_{\cO_{\E},\vphi}(M,\cOur) & \mrm{if}\ M\in \Mod_{\cO_{\E}}, \cr
\mrm{Hom}_{\cO_{\E},\vphi}(M,\mbb{Q}_p/\mbb{Z}_p\otimes_{\mbb{Z}_p} \cOur) &  \mrm{if}\ M\in \Mod_{\cO_{\E,\infty}}.
\end{array}
\right.
\end{align*}
Here, the  $G_{\underline{\pi}}$-action on  $T_{\cO_{\E}}(M)$
is given by $(g.f)(x):=g(f(x))$ for $f\in T_{\cO_{\E}}(M)$,
$g\in G_{\underline{\pi}}$ and $x\in M$.
Then we have a contravariant functor 
$T_{\cO_{\E}}\colon \Mod_{\cO_{\E}}\to \mrm{Rep}_{\mbb{Z}_p}(G_{\underline{\pi}})$
and $T_{\cO_{\E}}\colon \Mod_{\cO_{\E,\infty}}\to \mrm{Rep}^{\mrm{tor}}_{\mbb{Z}_p}(G_{\underline{\pi}})$.
By \cite[Corollary 3.2.3]{CL}, these two functors give equivalences of categories 
$\Mod_{\cO_{\E}}\simeq \mrm{Rep}_{\mbb{Z}_p}(G_{\underline{\pi}})$ and 
$\Mod_{\cO_{\E,\infty}}\simeq \mrm{Rep}^{\mrm{tor}}_{\mbb{Z}_p}(G_{\underline{\pi}})$.\\

For any integer $r\ge 0$,
we denote by ${}'\Mod^r_{\mfS}$ the category of finite type 
$\vphi$-modules $\mfM$ over $\mfS$
which are {\it of height $r$} in the sense that
the cokernel of the $\mfS$-linearization 
$1\otimes \vphi_{\mfM}\colon \mfS\otimes_{\vphi,\mfS}\mfM\to \mfM$
of $\vphi_{\mfM}$
is killed by $E(u)^r$.
A $\vphi$-modules $\mfM$ is {\it $p'$-torsion free} if, for any non-zero element $x\in \mfM$, 
$\mrm{Ann}_{\mfS}(x)$ is $0$ or $p^n\mfS$ for some $n$.
If $\mfM$ is killed by some power of $p$, then we can check that 
$\mfM$ is $p'$-torsion free if and only if $\mfM$ is $u$-torsion free.
We denote by $\Mod^r_{\mfS}$ the full subcategory of  $'\Mod^r_{\mfS}$
consisting of those objects which are finite and free over $\mfS$.
We also denote by $\Mod^r_{\mfS_{\infty}}$ the full subcategory of  $'\Mod^r_{\mfS}$
consisting of those objects which are $p'$-torsion and killed by a power of $p$.
We call objects of  $\Mod^r_{\mfS}$ or $\Mod^r_{\mfS_{\infty}}$
free Kisin modules or torsion Kisin modules, respectively.
If $\mfM$ is a Kisin module, then one can check that 
$\cO_{\E}\otimes_{\mfS}\mfM$
is an \'etale $\vphi$-module. 

We describe standard linear algebraic properties of  Kisin modules.

\begin{proposition}
Let $0\to \mfM'\to \mfM\to \mfM''\to 0$
be an exact sequence of $\vphi$-modules over $\mfS$.
If $\mfM', \mfM$ and $\mfM''$ are of finite type and $p'$-torsion free 
and $\mfM$ is of height $r$,
then $\mfM'$ and $\mfM$ are of height $r$.
\end{proposition}
\begin{proof}
See Propositions B. 1.3.3 and B. 1.3.5 of \cite{Fo1}.
\end{proof}

\begin{proposition}
\label{BASIC2}
Let $\mfM\in {}'\Mod^r_{\mfS}$ be killed by a power of $p$.
Then the following are equivalent.
\begin{itemize}
\item[$(1)$] $\mfM\in \Mod^r_{\mfS_{\infty}}$,
\item[$(2)$] the natural map $\mfM \to \cO_{\E}\otimes_{\mfS}\mfM$ is injective,
\item[$(3)$] there exists an increasing sequence 
$$
0=\mfM_0\subset \mfM_1\subset \mfM_2\subset \cdots \subset \mfM_n=\mfM
$$
of $\vphi$-modules over $\mfS$ such that, for each $i$,
$\mfM_i/\mfM_{i-1}$ is finite free over $k[\![u]\!]$
and $\mfM_i/\mfM_{i-1}\in {}'\Mod^r_{\mfS}$.
\item[$(4)$] $\mfM$ is a quotient of two finite free $\mfS$-modules 
$\mfN'$ and $\mfN''$ with $\mfN\in \Mod^r_{\mfS}$.
\end{itemize}
Moreover, if this is the case,  $\mfM_i$ and $\mfM_i/\mfM_{i-1}$
are objects of $\Mod^r_{\mfS_{\infty}}$ for each $i$.
\end{proposition}
\begin{proof}
The same proof as \cite[Proposition 2.3.2]{Li1} proceeds.
\end{proof}

\begin{corollary}
\label{exact}
Let $A$ be a $p$-torsion free $\mfS$-algebra
Let $\mfM$ be a Kisin module.
Then we have $\mrm{Tor}^{\mfS}_1(\mfM,A)=0$.
In particular, the functor 
from the category of Kisin modules to the category of $A$-modules
defined by $\mfM\mapsto A\otimes_{\mfS} \mfM$ is  exact. 
\end{corollary}

\begin{proof}
By Proposition \ref{BASIC2} and {\it d\'evissage}  argument,
we can reduce a proof to the case where $\mfM$ is killed by a power of $p$.
In this case, $\mfM$ is a  free $k[\![u]\!]$-module of finite rank.
Thus it suffices to show $\mrm{Tor}^{\mfS}_1(k[\![u]\!],A)=0$.
This equality in fact follows from the assumption that  $A$ is $p$-torsion free. 
\end{proof}
By this proposition, the following corollaries immediately follow:
\begin{corollary}
\label{ring:subset}
Let $\mfM$ be a Kisin module.
Let $A\subset B$ be a ring extension of $p$-torsion free $\mfS$-algebras such that
the natural map $A/pA\to B/pB$ is injective.
Then the natural map $A\otimes_{\mfS} \mfM\to B\otimes_{\mfS} \mfM$ is injective.
\end{corollary}

\begin{corollary}
\label{mod:sunset}
Let $\mfM$ be a Kisin module
and $\mfN$ a $\vphi$-module over $\mfS$ with $\mfM\subset \mfN$.
Let $\mfS\subset A\subset W(\mrm{Fr}R)$ be ring extensions.
Suppose that 
such that the natural map $A/pA\to \mrm{Fr}R$ is injective.

\noindent
$(1)$ The natural map $A\otimes_{\mfS} \mfM\to A\otimes_{\mfS} \mfN$
is injective.

\noindent
$(2)$ If $A$ is $\vphi$-stable, then the natural map 
$A\otimes_{\vphi, \mfS} \mfM\to A\otimes_{\vphi, \mfS} \mfN$
is injective.
\end{corollary}

We define  a $\mbb{Z}_p$-representation of $G_{\underline{\pi}}$ 
for any Kisin module $\mfM$ by
\begin{align*}
T_{\mfS}(\mfM):=\left\{ 
\begin{array}{cr}
\mrm{Hom}_{\mfS,\vphi}(\mfM,\mfS^{\mrm{ur}}) &  \mrm{if}\ \mfM\in \Mod^r_{\mfS},\quad \cr
\mrm{Hom}_{\mfS,\vphi}(\mfM,\mbb{Q}_p/\mbb{Z}_p\otimes_{\mbb{Z}_p} \mfS^{\mrm{ur}}) 
&  
\mrm{if}\ \mfM\in \Mod^r_{\mfS_{\infty}}.
\end{array}
\right.
\end{align*}
Here, the  $G_{\underline{\pi}}$-action on  $T_{\mfS}(\mfM)$
is given by $(g.f)(x):=g(f(x))$ for $f\in T_{\mfS}(\mfM)$,
$g\in G_{\underline{\pi}}$ and $x\in \mfM$.
If $\mfM$ is a Kisin module, then 
$M:=\cO_{\E}\otimes_{\mfS}\mfM$
is an \'etale $\vphi$-module.
Furthermore, 
we have a canonical isomorphism of 
$\mbb{Z}_p[G_{\underline{\pi}}]$-modules 
$T_{\mfS}(\mfM)\simeq T_{\cO_{\E}}(M)$
by \cite[Proposition 3.3.1]{CL}.

\begin{proposition}
\label{KisinFF}
$(1)$ Let $\mfM$ be a Kisin module
and put $M=\cO_{\E}\otimes_{\mfS}\mfM$.
Then, we have a canonical isomorphism 
$T_{\mfS}(\mfM)\simeq T_{\cO_{\E}}(M)$
of 
$\mbb{Z}_p[G_{\underline{\pi}}]$-modules.

\noindent
$(2)$ Let $\mfM$ be a free $($resp.\ torsion$)$ Kisin module.
Then the inclusion $\mfS^{\mrm{ur}}\hookrightarrow W(R)$ induces a natural isomorphism
$T_{\mfS}(\mfM)\simeq \mrm{Hom}_{\mfS,\vphi}(\mfM,W(R))$
$($resp.\ $T_{\mfS}(\mfM)\simeq \mrm{Hom}_{\mfS,\vphi}(\mfM,\mbb{Q}_p/\mbb{Z}_p\otimes_{\mbb{Z}_p}W(R))$$)$
of 
$\mbb{Z}_p[G_{\underline{\pi}}]$-modules.

\noindent
$(3)$ Assume that $\vphi^n(f(u)/u)$ is not a power of $E(u)$
for any $n\ge 0$.
Then the contravariant functor 
$T_{\mfS}\colon \Mod^r_{\mfS}\to \mrm{Rep}_{\mbb{Z}_p}(G)$
is  fully faithful.

\noindent
$(4)$ The contravariant functors 
$T_{\mfS}\colon \Mod^r_{\mfS}\to \mrm{Rep}_{\mbb{Z}_p}(G)$
and 
$T_{\mfS}\colon \Mod^r_{\mfS_{\infty}}\to \mrm{Rep}_{\mrm{tor}}(G)$
are  exact and faithful.
\end{proposition}

\begin{proof}
Assertions (1) and (2) for free Kisin modules are \cite[Proposition 3.3.1]{CL},
and a proof for the torsion case is essentially the same.
For this, the proof of \cite[Corollary 2.2.2]{Li1} is helpful for the readers.
The assertion (3) is \cite[Proposition 3.3.5]{CL}.
To show  (4), it suffices to show that 
$\Mod^r_{\mfS}\to \Mod_{\cO_{\E}}$ and  
$\Mod^r_{\mfS_{\infty}}\to \Mod_{\cO_{\E,\infty}}$
given by $\mfM\to \cO_{\E}\otimes_{\mfS}\mfM$ are
exact and faithful.
The exactness follows from the fact that the inclusion map
$\mfS\to \cO_{\E}$ is flat.
The faithfulness follows from 
Proposition \ref{BASIC2} (2) or Corollary \ref{ring:subset}.
\end{proof}

The following is  the main results of Section 5 of \cite{CL}.

\begin{theorem}[\cite{CL}, Theorem 1.0.3]
\label{pdiv}
Assume $v_p(a_1)>1$.
Then there exists an anti-equivalence of categories between 
the category $\Mod^1_{\mfS}$ of free Kisin modules of height $1$ 
and the category $(p\mrm{-div}{}_{/\cO_K})$ of  $p$-divisible groups
over the ring  of integers $\cO_K$ of $K$.
If $\mfM$ is a free Kisin module of height $1$, then 
the $G_{\underline{\pi}}$-action on $T_{\mfS}(\mfM)$ naturally extends to $G$.
This induces an anti-equivalence of categories between 
$\Mod^1_{\mfS}$ and $\mrm{Rep}^{1,\mrm{cris}}_{\mbb{Z}_p}(G)$.
Moreover, the following diagram is commutative:
\begin{center}
$\displaystyle \xymatrix{  
(p\mrm{-div}{}_{/\cO_K}) \ar^{\simeq}[rr] \ar^{\simeq}[rd] \ar_{T_p}[rd]
& 
&  \Mod^1_{\mfS} \ar^{T_{\mfS}}[dl] \ar_{\simeq}[dl]
\\
& \mrm{Rep}^{1,\mrm{cris}}_{\mbb{Z}_p}(G)
&
}$
\end{center}
\end{theorem}


Assume that $v_p(a_1)>1$.
Let $\mfS(1)$ be the free Kisin module of rank $1$
corresponding to $\mbb{Z}_p(1)$ via Theorem \ref{pdiv}.
Let $\mfrak{e}_{(1)}$ be a generator of $\mfS(1)$.
By Lemma 5.2.1 (2), we have  
$\vphi(\mfrak{e}_{(1)})=\mu_0E(u)\mfrak{e}_{(1)}$
for some $\mu_0\in \mfS^{\times}$.

\bn
{\bf Cartier duality.}
Here we give a Cartier duality theorem for \'etale $\vphi$-modules and 
Kisin modules. 
Since arguments here are completely the same as \cite[Section 3.1]{Li1},
we only give a brief sketch here.
We fix an integer $r\ge 0$. A lot of notion in this subsection depend on 
the choice of $r$ but we omit it from subscripts for an abbreviation.

Assume that $v_p(a_1)>1$.
Let $\mu_0\in \mfS^{\times}$ be as in the previous section. 
Let $\mfS^{\vee}$ be the free Kisin module of rank $1$
such that $\vphi(\mfrak{e})=(\mu_0E(u))^r\mfrak{e}$.
Here $\mfrak{e}$ is a generator of $\mfS^{\vee}$.
(Clearly, we have $\mfS^{\vee}=\mfS(1)$ if $r=1$.)
We see that $\mfS^{\vee}$ is of height $r$.
We set $\cO^{\vee}_{\E}:=\cO_{\E}\otimes_{\mfS}\mfS^{\vee}$, 
which is an \'etale $\vphi$-module. 
Note that we have isomorphisms   
$T_{\cO_{\E}}(\cO^{\vee}_{\E})\simeq T_{\mfS}(\mfS^{\vee})\simeq \mbb{Z}_p(r)$.
For any Kisin module $\mfM$, we define an $\mfS$-module $\mfM^{\vee}$ by
\begin{align*}
\mfM^{\vee}:=\left\{ 
\begin{array}{cr}
\mrm{Hom}_{\mfS}(\mfM,\mfS) & \mrm{if}\ \mfM\in \Mod^r_{\mfS},\quad \cr
\mrm{Hom}_{\mfS}(\mfM,\mfS_{\infty}) &  \mrm{if}\ \mfM\in \Mod^r_{\mfS_{\infty}}.
\end{array}
\right.
\end{align*}
For any \'etale $\vphi$-module $M$, we define an $\cO_{\E}$-module 
$M^{\vee}$ by
\begin{align*}
M^{\vee}:=\left\{ 
\begin{array}{cr}
\mrm{Hom}_{\cO_{\E}}(M,\cO_{\E}) & \mrm{if}\ M\in \Mod_{\cO_{\E}}, \cr
\mrm{Hom}_{\cO_{\E}}(M,\cO_{\E, \infty}) &  \mrm{if}\ M\in \Mod_{\cO_{\E,\infty}}.
\end{array}
\right.
\end{align*}
We then have canonical parings 
\begin{align*}
& \langle \cdot , \cdot \rangle \colon \mfM\times \mfM^{\vee}\to \mfS^{\vee}
& & \mrm{if}\ \mfM\in \Mod^r_{\mfS},\\
& \langle \cdot , \cdot \rangle \colon \mfM\times \mfM^{\vee}\to \mfS_{\infty}^{\vee}
& & \mrm{if}\ \mfM\in \Mod^r_{\mfS_{\infty}}
\end{align*}
and 
\begin{align*}
& \langle \cdot , \cdot \rangle \colon M\times M^{\vee}\to \cO_{\E}^{\vee}\quad 
& & \mrm{if}\ M\in \Mod_{\cO_{\E}},\\
& \langle \cdot , \cdot \rangle \colon M\times M^{\vee}\to \cO_{\E,\infty}^{\vee}
& & \mrm{if}\ M\in \Mod_{\cO_{\E,\infty}}.
\end{align*}

\begin{proposition}
Assume that $v_p(a_1)>1$.

\noindent
$(1)$ There exist a unique $\vphi$-semi-linear map 
$\vphi_{M^{\vee}}\colon M^{\vee}\to M^{\vee}$
which satisfies the following:
\begin{itemize}
\item[$(a)$] $(M^{\vee},\vphi_{M^{\vee}})$ is an \'etale $\vphi$-module,
\item[$(b)$] $\vphi_{M^{\vee}}$ is compatible with 
the pairing $\langle \cdot , \cdot \rangle$ for $M$,
\item[$(c)$] $T_{\cO_{\E}}(M^{\vee})\simeq T_{\cO_{\E}}(M)^{\vee}(r)$.
\end{itemize}

\noindent
$(2)$ Suppose that  $M=\cO_{\E}\otimes_{\mfS}\mfM$.
There exist a unique $\vphi$-semi-linear map 
$\vphi_{\mfM^{\vee}}\colon \mfM^{\vee}\to \mfM^{\vee}$ 
which satisfies the following:
\begin{itemize}
\item[$(a)$] $(\mfM^{\vee},\vphi_{\mfM^{\vee}})$ is a Kisin module of height $r$,
\item[$(b)$] $\vphi_{M^{\vee}}=1\otimes \vphi_{\mfM^{\vee}}$. In particular, 
$\vphi_{M^{\vee}}$ is compatible with the pairing $\langle \cdot , \cdot \rangle$ for $\mfM$,
\item[$(c)$] $T_{\mfS}(\mfM^{\vee})\simeq T_{\mfS}(\mfM)^{\vee}(r)$.
\end{itemize}
\end{proposition}
\begin{proof}
The same proof as \cite[Section 3.1]{Li1} proceeds.
\end{proof}

\bn
{\bf Comparison morphism of Kisin modules.}
We define a comparison morphism between 
Kisin modules and their representations. 
Precise arguments are given in
\cite[Section 3.2]{Li1}.

Let $\mfM$ be a Kisin module.
We define a $W(R)$-linear map 
$\iota_{\mfS}\colon 
W(R)\otimes_{\mfS}\mfM 
\to
W(R)\otimes_{\mbb{Z}_p}T_{\mfS}(\mfM)^{\vee}
$
by the composite
$$
W(R)\otimes_{\mfS}\mfM 
\to
\mrm{Hom}_{\mbb{Z}_p}(T_{\mfS}(\mfM),W(R))\simeq 
W(R)\otimes_{\mbb{Z}_p}T_{\mfS}(\mfM)^{\vee},
$$
where the first map is given by $x\mapsto (f\mapsto f(x))$
and the second is the natural map.
It is not difficult to check that 
$\iota_{\mfS}$ is  $\vphi$-equivalent
and $G_{\underline{\pi}}$-equivalent.

Assume that $v_p(a_1)>1$.
Take any generator $f$ of $T_{\mfS}(\mfS(1))$
and set $\mft:=f(\mfrak{e}_{(1)})\in W(R)$.
Since $f$ is compatible with $\vphi$ and is a generator of $T_{\mfS}(\mfS(1))$,
we see 
$$
\vphi(\mft)=\mu_0 E(u)\mft\quad \mrm{and}
\quad \mft\in W(R)\smallsetminus pW(R).
$$
Such $\mft$ is unique up to multiplication by $\mbb{Z}^{\times}_p$
and is independent of the choice of $f$.

\begin{proposition}
Assume that $v_p(a_1)>1$.
There exist natural 
$W(R)$-linear morphisms
$$
\iota_{\mfS} \colon W(R)\otimes_{\mfS}\mfM 
\to
W(R)\otimes_{\mbb{Z}_p}T_{\mfS}(\mfM)^{\vee}
$$
and 
$$
\iota^{\vee}_{\mfS} \colon 
W(R)^{\vee}\otimes_{\mbb{Z}_p}T_{\mfS}(\mfM)^{\vee}
\to
W(R)(-r)\otimes_{\mfS}\mfM
$$
which satisfy  the following:
\begin{itemize}
\item[$(1)$] $\iota_{\mfS}$ and $\iota^{\vee}_{\mfS}$ are 
$\vphi$-equivalent
and $G_{\underline{\pi}}$-equivalent.
\item[$(2)$] If we identify $W(R)^{\vee}=W(R)(-r)=W(R)$,
then we have
$\iota^{\vee}_{\mfS}\circ \iota_{\mfS}=\mft^r\otimes Id_{\mfM}$
and $\iota_{\mfS}\circ \iota^{\vee}_{\mfS}
=\mft^r\otimes Id_{T_{\mfS}(\mfM)^{\vee}}$.
\end{itemize}
\end{proposition}
\begin{proof}
The proof is completely the same as that of \cite[Theorem 3.2.2]{Li1}.
\end{proof}

\begin{corollary}
\label{comp1}
Assume that $v_p(a_1)>1$.
The maps $\iota_{\mfS}$ and $\iota^{\vee}_{\mfS}$ are injective, 
and we have 
$\mft^r(W(R)\otimes_{\mbb{Z}_p}T_{\mfS}(\mfM)^{\vee})
\subset \mrm{Im}(\iota_{\mfS})$
and 
$\mft^r(W(R)(-r)\otimes_{\mfS}\mfM)\subset 
\mrm{Im}(\iota^{\vee}_{\mfS})$.
\end{corollary}


\section{Lattices in crystalline representations}

In this section, we study Galois actions on Kisin modules 
which corresponds to crystalline representations.
It gives an anti-equivalence 
between a category of Kisin modules with certain Galois actions
and a category of lattices in crystalline representations
with some Hodge-Tate weights.

\subsection{$(\vphi,\hat{G})$-modules}

Let $\widehat{K}_{\underline{\pi}}/K$ be the Galois closure of 
the extension $K_{\underline{\pi}}/K$.
We denote by $\hat{G}$ the absolute Galois group 
$\mrm{Gal}(\overline{K}/\widehat{K}_{\underline{\pi}})$ of $\widehat{K}_{\underline{\pi}}$.
Following \cite{CL}, we set 
$\frak{O}_{\alpha}:=\mfS[\![\frac{E(u)^p}{p}]\!][1/p]\subset B^+_{\mrm{cris}}$.
It is not difficult to check 
$I_+\mfO_{\alpha}=u\mfO_{\alpha}$ and  $\mfO_{\alpha}/I_+\mfO_{\alpha}\simeq K_0$.
We note that we have $\mfS[\![\frac{E(u)^p}{p}]\!]
=\mfS[\![\frac{u^{ep}}{p}]\!]\subset A_{\mrm{cris}}$ and  
$\mfS[\![\frac{E(u)^p}{p}]\!]$ is $p$-adically complete and $\vphi$-stable.

{\it In the rest of this paper, we fix the choice of 
a $K_0$-subalgebra $\mcal{R}_{K_0}$ of  $B^+_{\mrm{cris}}$} 
which satisfies the following properties:
\begin{itemize}
\item $\mfO_{\alpha}\subset \mcal{R}_{K_0}$,
\item $\mcal{R}_{K_0}\subset B^+_{\mrm{cris}}$ is stable under $\vphi$ and $G$-actions, and
\item the $G$-action on $\mcal{R}_{K_0}$ factors through $\hat{G}$.
\end{itemize}

\begin{remark}
(1) Such  $\mcal{R}_{K_0}$ exists.
In fact, the $K_0$-subalgebra  of $B^+_{\mrm{cris}}$  
generated by $\{gx \mid g\in G, x\in  \mfO_{\alpha} \}$ 
satisfies all the desired properties.

\noindent
(2) In the classical setting $f(u)=u^p$,
an explicitly described $\mcal{R}_{K_0}$ has been considered. For this, see \cite{Li2}.
\end{remark}

\noindent
We set $\whR:=\mcal{R}_{K_0}\cap W(R)$.
By definition, we see that 
$\whR\subset W(R)$ is stable under $\vphi$ and $G$-actions, 
the $G$-action on $\whR$ factors through $\hat{G}$, and 
the map $\nu$ induces isomorphisms
$\mcal{R}_{K_0}/I_+\mcal{R}_{K_0}\simeq K_0$
and 
$\whR/I_+\whR\simeq W(k)$.

\begin{definition}
A {\it $(\vphi,\hat{G})$-module $($of height $r$$)$} is a triple 
$\hat{\mfM}=(\mfM,\vphi,\hat{G})$ where
\begin{itemize}
\item[(1)] $(\mfM,\vphi)$ is a free Kisin module $\mfM$ of height $r$,
\item[(2)] $\hat{G}$ is an $\whR$-semi-linear continuous\footnote{
This means that the $G$-action on 
$W(R)\otimes_{\whR} (\whR \otimes_{\vphi, \mfS}\mfM)
=W(R) \otimes_{\vphi, \mfS}\mfM$
induced by the $\hat{G}$-action on $\whR \otimes_{\vphi, \mfS}\mfM$ 
 is continuous with respect to  the weak topology of $W(R)$.} $\hat{G}$-action 
on $\whR \otimes_{\vphi, \mfS}\mfM$,
\item[(3)] the $\hat{G}$-action on $\whR \otimes_{\vphi, \mfS}\mfM$
commutes  with $\vphi_{\whR}\otimes \vphi_{\mfM}$, and 
\item[(4)] $\vphi^{\ast}\mfM \subset 
(\whR \otimes_{\vphi, \mfS}\mfM)^{G_{\underline{\pi}}}$.
\end{itemize}
We denote by $\Mod^{r,\hat{G}}_{\mfS}$ the category of 
 $(\vphi,\hat{G})$-modules of height $r$.
\end{definition}

We define  a $\mbb{Z}_p$-representation $\hat{T}(\hat{\mfM})$ 
of $G$ for any $(\vphi,\hat{G})$-module $\hat{\mfM}$ by
$$
\hat{T}(\hat{\mfM}):=
\mrm{Hom}_{\whR,\vphi}(\whR\otimes_{\vphi,\mfS}\mfM,W(R)). 
$$
Here, the  $G$-action on  $\hat{T}(\hat{\mfM})$
is given by $(g.f)(x):=g(f(g^{-1}(x)))$ for 
$f\in \hat{T}(\hat{\mfM})$,
$g\in G$ and $x\in \whR\otimes_{\vphi,\mfS}\mfM$.
Note that 
we have a natural isomorphism of 
$\mbb{Z}_p[G_{\underline{\pi}}]$-modules
$$
\theta\colon T_{\mfS}(\mfM)\overset{\sim}{\longrightarrow}
\hat{T}(\hat{\mfM})
$$
given by $\theta(f)(a\otimes x):=a\vphi(f(x))$
for $f\in T_{\mfS}(\hat{\mfM})$, $a\in \whR$ and $x\in \mfM$.
In particular, 
$\hat{T}(\hat{\mfM})$ is a free $\mbb{Z}_p$-module of rank $d$,
where $d:=\mrm{rank}_{\mfS}\mfM$.
Hence we obtain a contravariant functor 
$$
\hat{T}\colon \Mod^{r,\hat{G}}_{\mfS}\to \mrm{Rep}_{\mbb{Z}_p}(G).
$$
Note also that we have a canonical isomorphism
$\hat{T}(\hat{\mfM})\simeq
\mrm{Hom}_{W(R),\vphi}(W(R)\otimes_{\vphi,\mfS}\mfM,W(R))$.

\begin{definition}
\label{mod}
(1) We denote by ${}'\Mod^{r,\hat{G},\mrm{cris}}_{\mfS}$ the full subcategory 
of $\Mod^{r,\hat{G}}_{\mfS}$ consisting of objects $\hat{\mfM}$
which satisfy the following condition:
For any $g\in \hat{G}$, there exists $\alpha_g\in B^+_{\mrm{cris}}$
such that 
\begin{itemize}
\item[(a)] $g(1\otimes x)-(1\otimes x)\in 
\alpha_g(B^+_{\mrm{cris}}\otimes_{\vphi,\mfS} \mfM)$ for any $x\in \mfM$, and 
\item[(b)] $\vphi^n(\alpha_g)/p^{nr}$ converges to $0$
$p$-adically in $B^+_{\mrm{cris}}$.
\end{itemize}

\noindent
(2) We denote by $\Mod^{r,\hat{G},\mrm{cris}}_{\mfS}$ the full subcategory 
of $\Mod^{r,\hat{G}}_{\mfS}$ consisting of objects $\hat{\mfM}$
which satisfy the following condition:
For any $g\in \hat{G}$ and $x\in \mfM$, we have 
$$
g(1\otimes x)-(1\otimes x)\in 
\vphi(gu-u)B^+_{\mrm{cris}}\otimes_{\vphi,\mfS} \mfM.
$$
(Note that, if this is the case, $g(1\otimes x)-(1\otimes x)$
is in fact contained in  
$\vphi(gu-u)B^+_{\mrm{cris}}\otimes_{\vphi,\mfS} \mfM
\cap I^{[1]}W(R)\otimes_{\vphi,\mfS}\mfM$
since $\vphi(gu-u)\in I^{[1]}W(R)$.)
\end{definition}

\begin{remark}
To understand  $(\vphi,\hat{G})$-module,
it is very important 
to study the structure of the Galois group $\hat{G}$
and to find a ``good choice'' of $\mcal{R}_{K_0}$.
In the classical Kisin's setting $f(u)=u^p$,
these are well studied. For this, see \cite{Li2}.

We should remark that in this classical setting,
we may consider $(\vphi,\hat{G})$-modules
as ``linear data'' like $(\vphi,\Gamma)$-modules.
In fact,  $\hat{G}$ is topologically 
generated  by $\mrm{Gal}(\widehat{K}_{\underline{\pi}}/K_{\underline{\pi}})$
and a (fixed) generator $\tau$ of 
$\mrm{Gal}(\widehat{K}_{\underline{\pi}}/K(\mu_{p^{\infty}}))$.
Here, $\mu_{p^{\infty}}$ is the set of $p$-power roots of unity. 
Hence the $\hat{G}$-action on  a 
$(\vphi,\hat{G})$-module  
is essentially determined by the $\tau$-action only.
\end{remark}

\begin{remark}
To understand objects of the category $\Mod^{r,\hat{G},\mrm{cris}}_{\mfS}$,
studying the ideal $I_g:=\vphi(gu-u)B^+_{\mrm{cris}}\cap I^{[1]}W(R)$
of $W(R)$  must be important.
However, it is not so easy (at least for the author).
Later, we give a partial result on this ideal in 
Proposition \ref{lem1}.
Here we describe some known facts about $I_g$ and give some remarks.

\noindent
(1) Suppose $v_p(a_1)>1$.
Then we can check $I_g \subset I^{[1+]}W(R)$ as follows:
Let $\mft$ be as in the previous section, which is a generator of $I^{[1]}W(R)$.
Take $x=\vphi(gu-u)y=\vphi(\mft)z$ with $y\in B^+_{\mrm{cris}}$ 
and $z\in W(R)$.
It suffices to show $z\in I_+W(R)$.
By  \cite[Lemma 2.3.2]{CL} (see also Proposition \ref{Gact}),
we have $gu-u\in \vphi(\mft)I_+W(R)$.
This implies  
$\vphi(gu-u)\in \vphi^2(\mft)I_+W(R)=\vphi(E(u))\vphi(\mft)I_+W(R)
\subset \vphi(\mft)I_+W(R)$,
and thus we obtain $z=\vphi(gu-u)y/\vphi(\mft) \in W(R)\cap I_+B^+_{\mrm{cris}}=
I_+W(R)$ as desired.

\noindent
(2) (Kisin's setting)
If $f(u)=u^p$,
then we can show that
$I_g \subset u^pI^{[1]}W(R)$ as follows:
Since $gu-u\in uW(R)$ in this case.
it suffices to show $u^pB^+_{\mrm{cris}}\cap I^{[1]}W(R)\subset u^pI^{[1]}W(R)$.
Take any $x=u^py\in u^pB^+_{\mrm{cris}}\cap I^{[1]}W(R)$.
By \cite[Lemma 3.2.2]{Li3}, $u^py\in W(R)$ shows  $y\in W(R)$.
On the other hand, $u^py\in I^{[1]}W(R)$ and $\vphi^n(u^p)\notin \mrm{Fil}^1B_{\mrm{dR}}$
for any $n\ge 0$
implies that $y\in I^{[1]}B^+_{\mrm{cris}}$.
Hence we have $y\in I^{[1]}W(R)$, which induces $x\in u^pI^{[1]}W(R)$ as desired. 

The ideal $u^pI^{[1]}W(R)$ of $W(R)$
plays an important role for studies of $(\vphi,\hat{G})$-modules
(cf. \cite{Li2})
which correspond to lattices in crystalline representations.
It allows us to study reductions of crystalline representations
and also gives interesting  applications such as the weight part of Serre's conjecture
(cf.\ \cite{Ga},\cite{GLS1},\cite{GLS2}).
\end{remark}

\begin{remark}
\label{fullsub}
By Corollary \ref{polycor} later, 
we obtain the fact that the category 
$\Mod^{r,\hat{G},\mrm{cris}}_{\mfS}$
is a full subcategory of 
${}'\Mod^{r,\hat{G},\mrm{cris}}_{\mfS}$
if $v_p(a_1)>r$.
\end{remark}

\bn
{\bf Comparison morphism of $(\vphi,\hat{G})$-modules.}
Let $\hat{\mfM}$ be a $(\vphi,\hat{G})$-module.
We define a $W(R)$-linear map 
$\hat{\iota}\colon 
W(R)\otimes_{\vphi, \mfS}\mfM 
\to
W(R)\otimes_{\mbb{Z}_p}\hat{T}(\hat{\mfM})^{\vee}
$
by the composite
$$
W(R)\otimes_{\vphi, \mfS}\mfM 
\to
\mrm{Hom}_{\mbb{Z}_p}(\hat{T}(\hat{\mfM}),W(R))\simeq 
W(R)\otimes_{\mbb{Z}_p}\hat{T}(\hat{\mfM})^{\vee},
$$
where the first map is given by $x\mapsto (f\mapsto f(x))$
and the second is the natural map.
It is not difficult to check that 
$\hat{\iota}$ is  $\vphi$-equivalent
and $G$-equivalent.
By the same argument as that in the proof of 
\cite[Proposition (2),(3)]{Li2}, we can check the following.

\begin{proposition}
\label{comp2}
$(1)$ We have $\hat{\iota}\simeq W(R)\otimes_{\vphi,W(R)}\iota_{\mfS}$, that is, the 
following diagram is commutative.
\begin{center}
$\displaystyle \xymatrix{
W(R)\otimes_{\vphi, \mfS}\mfM \ar^{\hat{\iota}}[rr] 
\ar@{=}[d] 
& &
W(R)\otimes_{\mbb{Z}_p}\hat{T}(\hat{\mfM})^{\vee} \ar_{W(R)\otimes \theta^{\vee}}[d] \ar^{\wr}[d]\\
W(R)\otimes_{\vphi, \mfS}\mfM \ar^{\vphi^{\ast}\iota_{\mfS}}[rr] 
& &
W(R)\otimes_{\mbb{Z}_p} T_{\mfS}(\mfM)^{\vee}. 
}$
\end{center}
Here, $\vphi^{\ast}\iota_{\mfS}:=W(R)\otimes_{\vphi,W(R)}\iota_{\mfS}$.

\noindent
$(2)$ Assume that $v_p(a_1)>1$. Then  the map $\hat{\iota}$ is injective and we have
 $\mft_0^r(W(R)\otimes_{\mbb{Z}_p}\hat{T}(\hat{\mfM})^{\vee})
\subset \mrm{Im}(\hat{\iota})$.
Here, $\mft_0$ is any generator of $I^{[1]}W(R)$
(e.g., $\mft_0=\vphi(\mft)$ (cf., \cite[Proposition 5.1.3]{Fo2})).
\end{proposition}

\subsection{Main Results}
\label{MainSection}

We often use the following conditions.\\

\noindent
{\bf Condition (P):} $\vphi^n(f(u)/u)$
is not a power of $E(u)$ for any $n\ge 0$.\\

\noindent
{\bf Condition:} $v_p(a_1)>\max \{r,1\}$.\\

\noindent
Note that these conditions are satisfied if $a_1=0$.
We denote by $\mrm{Rep}^{r,\mrm{cris}}_{\mbb{Z}_p}(G)$
the category of $G$-stable $\mbb{Z}_p$-lattices in 
crystalline $\mbb{Q}_p$-representations of $G$
with Hodge-Tate weights in $[0,r]$.
Now we state our main theorem of this paper.

\begin{theorem}
\label{MT}
Assume the conditions $(P)$ and $v_p(a_1)>\max \{r,1\}$.

\noindent
$(1)$ $\Mod^{r,\hat{G},\mrm{cris}}_{\mfS}={}'\Mod^{r,\hat{G},\mrm{cris}}_{\mfS}$.

\noindent
$(2)$ The contravariant functor $\hat{T}$ 
induces an anti-equivalence of categories
between $\Mod^{r,\hat{G},\mrm{cris}}_{\mfS}$ and $\mrm{Rep}^{r,\mrm{cris}}_{\mbb{Z}_p}(G)$.
\end{theorem}

Summary, 
we have 
$$
\Mod^{r,\hat{G},\mrm{cris}}_{\mfS}={}'\Mod^{r,\hat{G},\mrm{cris}}_{\mfS}
\overset{\sim}{\longrightarrow} \mrm{Rep}^{r,\mrm{cris}}_{\mbb{Z}_p}(G).
$$
under the conditions (P) and $v_p(a_1)>\max \{r,1\}$.
The theorem is an easy consequence of the following result
and Remark \ref{fullsub},
which we show in the rest of this section.

\begin{theorem}
\label{MT'}
$(1)$ Assume the conditions $(P)$ and $v_p(a_1)>1$.
Then the contravariant functor 
$\hat{T}\colon \Mod^{r,\hat{G}}_{\mfS}\to \mrm{Rep}_{\mbb{Z}_p}(G)$
is fully faithful.

\noindent
$(2)$ Assume the condition $v_p(a_1)>\max \{r,1\}$.
Then the contravariant functor 
$\hat{T}\colon {}'\Mod^{r,\hat{G},\mrm{cris}}_{\mfS}\to \mrm{Rep}_{\mbb{Z}_p}(G)$
has values in $\mrm{Rep}^{\infty,\mrm{cris}}_{\mbb{Z}_p}(G)$.
If we furthermore assume the condition $(P)$, 
then it has values in $\mrm{Rep}^{r,\mrm{cris}}_{\mbb{Z}_p}(G)$.

\noindent
$(3)$ Assume the conditions $(P)$ and $v_p(a_1)>\max \{r,1\}$.
Then the contravariant functor 
$\hat{T}\colon \Mod^{r,\hat{G},\mrm{cris}}_{\mfS}\to \mrm{Rep}^{r,\mrm{cris}}_{\mbb{Z}_p}(G)$
is essentially surjective.
\end{theorem}

The contravariant functor 
$T_{\mfS}\colon \Mod^r_{\mfS}\to \mrm{Rep}_{\mbb{Z}_p}(G_{\underline{\pi}})$
is fully faithful under the condition (P).
By the condition $v_p(a_1)>1$, we know the injectivity of comparison morphisms
(cf. Corollary \ref{comp1} and Proposition \ref{comp2} (2)).
Thus Theorem (1) follows by completely the same way 
as the last paragraph of \cite[Section 3.1]{Li2} 
and so we leave the proof of (1)  for the readers.

In the rest of this section,
we show Theorem (2) and (3).

\subsection{Some notations and Properties}
\label{technical}

Before a proof of Theorem \ref{MT'} (2) and (3), 
we give some notations and their properties.

\bn
{\bf The map $\xi_{\alpha}$.}
Let $\mfM\in \Mod^r_{\mfS}$ be a 
Kisin module of rank $d$
and set $M:=\vphi^{\ast}\mfM/u\vphi^{\ast}\mfM$.

\begin{lemma}[\cite{CL}, Lemma 4.5.6.]
\label{xi1}
Assume that $v_p(a_1)>r$. 
Then  there exists a unique $\vphi$-equivalent 
$\mfrak{O}_{\alpha}$-linear isomorphism 
$$
\xi_{\alpha}\colon \mfrak{O}_{\alpha}\otimes_{W(k)} M
\overset{\sim}{\longrightarrow}
\mfrak{O}_{\alpha}\otimes_{\mfS} \vphi^{\ast}\mfM
$$
whose reduction modulo $u$ is the identity map on $M$.
\end{lemma}

We recall how to define $\xi_{\alpha}$.
Let $\mfrak{e}_1,\dots \mfrak{e}_d$ be a basis of $\mfM$
and let $A\in M_d(\mfS)$ be a matrix such that
$\vphi(\mfrak{e}_1,\dots ,\mfrak{e}_d)
=(\mfrak{e}_1,\dots ,\mfrak{e}_d)A$.
Put $e_i=1\otimes \mfrak{e}_i\in \vphi^{\ast}\mfM$ 
for each $i$.
Then $e_1,\dots ,e_d$ is a basis of $\vphi^{\ast}\mfM$
and $\vphi(e_1,\dots, e_d)
=(e_1,\dots ,e_d)\vphi(A)$.
Put $\bar{e}_i:=e_i\ \mrm{mod}\ u\vphi^{\ast}\mfM$ 
for each $i$.
Then $\bar{e}_1,\dots ,\bar{e}_d$ is a basis of $M$
and $\vphi(\bar{e}_1,\dots ,\bar{e}_d)
=(\bar{e}_1,\dots ,\bar{e}_d)\vphi(A_0)$
where $A_0=A\ \mrm{mod}\ u\mfS\in M_d(W(k))$.

It was shown in the proof of \cite[Lemma 4.5.6]{CL}
that the matrix
$$
\vphi(A)\cdots \vphi^n(A)\vphi^n(A_0^{-1})\cdots \vphi(A_0^{-1})
$$
converges to an element of $GL_d(\mfrak{O}_{\alpha})$.
Putting 
$$
Y:=\lim_{n\to \infty} \vphi(A)\cdots 
\vphi^n(A)\vphi^n(A_0^{-1})\cdots \vphi(A_0^{-1}),
$$
we define 
$\xi_{\alpha}\colon \mfrak{O}_{\alpha}\otimes_{W(k)} M
\overset{\sim}{\longrightarrow}
\mfrak{O}_{\alpha}\otimes_{\mfS} \vphi^{\ast}\mfM$
by
$
\xi_{\alpha}(\bar{e}_1,\dots ,\bar{e}_d)=(e_1,\dots ,e_d)Y.
$

\bn
{\bf The map $\xi'_{\alpha}$.}
Let  $T$ be an object of $\mrm{Rep}^{r,\mrm{cris}}_{\mbb{Z}_p}(G)$
and put $V=T[1/p]$.
Let $D=D_{\mrm{cris}}(V):=(B_{\mrm{cris}}\otimes_{\mbb{Q}_p} V^{\vee})^G$
 be the filtered $\vphi$-module
corresponding to $V$. 
Let $\mfO$ be the subring of  $K_0(\!(u)\!)$ 
consisting of those elements which converge for all $x\in \overline{K}$
with $v_p(x)\ge 0$.
We equip $\mfO$ with a $K_0$-semi-linear Frobenius $\vphi\colon \mfO\to \mfO$
such that $\vphi(u)=f(u)$. 
We see that $\mfO$ is a $\vphi$-stable subring of $\mfO_{\alpha}$.
By \cite[Section 4.2]{CL}, there exists a $\vphi$-module 
$\mcal{M}=\mcal{M}(D)$ over $\mfO$ such that 
\begin{itemize}
\item $\mcal{D}_0 \subset \mcal{M} \subset \lambda^{-r}\mcal{D}_0$ where
$\mcal{D}_0:=\mfO\otimes_{K_0} D$ and 
$\lambda:=\prod^{\infty}_{n=0}\vphi^n(E(u)/E(0))\in \mfO$.
\item $\mcal{M}$ is of height $r$ in the sense that  
the cokernel of the $\mfO$-linearization 
$1\otimes \vphi_{\mcal{M}}\colon \mfO\otimes_{\vphi,\mfO}\mcal{M}\to \mcal{M}$
of $\vphi_{\mcal{M}}$ is killed by $E(u)^r$.
\item $\mcal{M}$ is \'etale in the sense of \cite[Section 4.4]{CL}.
\end{itemize}

By Theorem 4.4.1 of {\it loc.\ cit.},
there exists a Kisin module $\mfM\subset \mcal{M}$ of height $r$
such that $\mfO\otimes_{\mfS} \mfM=\mcal{M}$.
Now we define an isomorphism $\xi'_{\alpha}\colon \mfO_{\alpha}\otimes_{K_0} D\overset{\sim}{\rightarrow} 
\mfO_{\alpha}\otimes_{\mfS} \vphi^{\ast}\mfM$
as follows:
The isomorphism $1\otimes \vphi\colon \vphi^{\ast}D\overset{\sim}{\rightarrow}D$
induces an isomorphism 
$1\otimes \vphi\colon \vphi^{\ast}\mcal{D}_0\overset{\sim}{\rightarrow}\mcal{D}_0$.
Thus we obtain an injection
$\xi'\colon \mcal{D}_0\overset{\sim}{\rightarrow} 
\vphi^{\ast}\mcal{D}_0\hookrightarrow \vphi^{\ast}\mcal{M}\simeq 
\mfO\otimes_{\mfS} \vphi^{\ast}\mfM$.
Then we define $\xi'_{\alpha}=\mfO_{\alpha}\otimes_{\mfO} \xi'$.
It is shown in Lemma 4.2.2  of {\it loc.\ cit.} that 
$\xi'_{\alpha}$ is an isomorphism.

\bn
{\bf The map $\iota_0$.}
Following  Proposition 4.5.1 of {\it loc.\ cit.}, 
we define
a  $G_{\underline{\pi}}$-equivalent injection 
$\iota_0\colon T_{\mfS}(\mfM)\hookrightarrow V$
by the composite 
\begin{align*}
T_{\mfS}(\mfM)=\mrm{Hom}_{\mfS,\vphi}(\mfM,W(R)) 
& \hookrightarrow \mrm{Hom}_{\mfO,\vphi,\mrm{Fil}}
(\vphi^{\ast}\mcal{M}, B^+_{\mrm{cris}})\\
& \overset{\sim}{\rightarrow} \mrm{Hom}_{\mfO_{\alpha},\vphi,\mrm{Fil}}
(\mfO_{\alpha}\otimes_{\mfO}\vphi^{\ast}\mcal{M}, B^+_{\mrm{cris}})\\
& \overset{\sim}{\rightarrow} \mrm{Hom}_{\mfO_{\alpha},\vphi,\mrm{Fil}}
(\mfO_{\alpha}\otimes_{K_0} D, B^+_{\mrm{cris}})\\
& \simeq V_{\mrm{cris}}(D)\simeq V,
\end{align*}
where the first arrow is given by 
$f\mapsto (a\otimes x\in \mfO\otimes_{\vphi,\mfS} \mfM=\vphi^{\ast}\mcal{M} \mapsto a\vphi(f(x)))$,
the second and the fourth arrows are natural isomorphisms,
and the third arrow is given by $(f\mapsto f\circ \xi'_{\alpha})$.
We omit definitions of filtrations of various modules appeared above
since precise informations  of them  are not so important here.
We only note that definitions of filtrations  are  given in \cite[Section 4]{CL}.

\bn
{\bf The $G$-action on $u$.}
We consider a difference between $gu$ and $u$ for $g\in G$.
We recall that $f(u)=\sum^p_{i=1}a_iu^i = u^p+a_{p-1}u^{p-1}+\cdots +a_1u\in \mbb{Z}_p[u]$
with the property $f(u)\equiv u^p\ \mrm{mod}\ p$.

At first, here is a Cais-Liu's observation.
\begin{proposition}[\cite{CL}, Lemma 2.3.2]
\label{Gact}
Let $g\in G$ be arbitrary.

\noindent
$(1)$ We have $gu-u\in I^{[1]}W(R)$.

\noindent
$(2)$ If $v_p(a_1)>1$, then we have $gu-u\in I^{[1+]}W(R)$.

\end{proposition}

We use the following proposition in the final section. 

%

\begin{proposition}
\label{val}
Let $j_0$ be the minimum integer $1\le j\le p$ such that 
$v_p(ja_j)=1$.
Let $g\in G\smallsetminus G_{\underline{\pi}}$ and $N\ge 1$ the integer 
such that $g\pi_{N-1}=\pi_{N-1}$ and $g\pi_N\not= \pi_N$.
We denote by $\bar{u}$ the image of $u$ for the projection 
$W(R)\twoheadrightarrow R$.
Then we have $v_R(g\bar{u}-\bar{u})=p^N/(p-1)+(j_0-1)/(e(p-1))$.
\end{proposition}

\begin{proof}
Since $v_R(g\bar{u}-\bar{u})=\lim_{n\to \infty} p^nv_p(g\pi_n-\pi_n)$,
it suffices to show that 
$v_p(g\pi_n-\pi_n)=c_n$ for $n\ge N$, 
where  $c_n:=p^N/((p-1)p^n)+(j_0-1)/(ep^n(p-1))$.

We note that we have an equation $\sum^p_{i=1}a_i(g\pi_n^i-\pi_n^i)=g\pi_{n-1}-\pi_{n-1}$. 
Putting $b^{(\ell)}_n=
\sum^{p-\ell}_{j=0}a_{\ell+j}
\begin{pmatrix}\ell+j\\j \end{pmatrix}
\pi_n^j
$, we have 
$$
  \sum^p_{i=1}a_i(g\pi_n^i-\pi_n^i)
= \sum^p_{i=1}\sum^{i-1}_{j=0} a_i
\begin{pmatrix}i\\j \end{pmatrix}
(g\pi_n-\pi_n)^{i-j}\pi_n^j
=\sum^p_{\ell=1} 
b^{(\ell)}_n
(g\pi_n-\pi_n)^{\ell}.
$$
Hence we obtain that $g\pi_n-\pi_n$ is a solution of the equation
$$
\sum^p_{\ell=1} 
b^{(\ell)}_nX^{\ell}-(g\pi_{n-1}-\pi_{n-1})=0.
$$
We note that we have $b_n^{(p)}=1$ and
\begin{align*}
v_p(a_{\ell+j}
\begin{pmatrix}\ell+j\\j \end{pmatrix}
\pi_n^j)
=
\left\{ 
\begin{array}{cr}
v_p(a_{\ell+j})+\frac{j}{ep^n} &  \mrm{for}\ 0\le j<p-\ell, \cr
1+\frac{p-\ell}{ep^n} & \mrm{for}\ j=p-\ell \qquad 
\end{array}
\right.
\end{align*}
if  $1\le \ell \le p-1$.\\

\noindent
{\bf The case $j_0=p$}: By the assumption $v_p(a_1),\dots ,v_p(a_{p-1})>1$,
we have 
\begin{equation}
\label{val1}
v_p(b_n^{(\ell)})=1+\frac{p-{\ell}}{ep^n}
\end{equation}
for $1\le \ell \le p-1$.
Now we show $v_p(g\pi_n-\pi_n)=c_n$ by induction on $n\ge N$.

Suppose $n=N$.
Then $g\pi_N-\pi_N$ is a solution of the equation
$$
\sum^{p-1}_{\ell=0} 
b^{(\ell+1)}_NX^{\ell}=0.
$$
Hence it is enough to show that the Newton polygon of the polynomial
$\sum^{p-1}_{\ell=0} 
b^{(\ell+1)}_NX^{\ell}\in \mbb{Z}_p[X]$
is the line segment, denoted by $l_N$, connecting $(0,(p-1)c_N)$ to $(p-1,0)$.
This follows immediately by \eqref{val1}.

We suppose that the assertion holds for $n$ and consider the case $n+1$.
We recall that $g\pi_{n+1}-\pi_{n+1}$ is a solution of the equation
$$
\sum^p_{\ell=1} 
b^{(\ell)}_{n+1}X^{\ell}-(g\pi_n-\pi_n)=0.
$$
Thus it is enough to show that the Newton polygon of the polynomial
$\sum^p_{\ell=1} 
b^{(\ell)}_{n+1}X^{\ell}-(g\pi_n-\pi_n)$
is the line segment, denoted by $l_{n+1}$, connecting 
$(0,c_n)$ to $(p,0)$.
This follows immediately by \eqref{val1} again.

\noindent
{\bf The case $j_0<p$}:
Let $s$ be the number of integers $j$ such that 
$1\le j\le p-1$ and $v_p(a_j)=1$.
By assumption we have $s>0$.
Let $j_{-1},j_0,j_1,\dots ,j_{s-1}$ be integers such that 
$j_{-1}=0<  j_0< j_1< \cdots < j_{s-1}\le p-1$ and $v_p(a_{j_0})=\cdots =v_p(a_{j_{s-1}})=1$.
Then we see 
\begin{align}
\label{val2}
v_p(b_n^{(\ell)})
=
\left\{ 
\begin{array}{cr}
1+\frac{j_k-\ell}{ep^n} &  \mrm{if}\ j_{k-1}< \ell \le j_k\ \mrm{for\ some}\ 0\le k\le s-1, \cr
1+\frac{p-\ell}{ep^n} & \mrm{if}\ j_{s-1}<\ell \le p-1. \qquad \qquad \qquad \quad \qquad
\end{array}
\right.
\end{align}
By a similar strategy to the proof of (1),
we can show $v_p(g\pi_n-\pi_n)=c_n$ by induction on $n\ge N$.
We leave a proof to the readers.
\end{proof}



We consider a convergence property of $\vphi^n(gu-u)$. For this, we need 

\begin{lemma}
\label{poly}
Assume that $v_p(a_1)>r$.
Then, for any $n\ge 1$ and $0\le i\le n$, there exists a polynomial $h^{(n)}_{2^{n-i}}(X,Y)\in \mbb{Z}_p[X,Y]$
such that 
\begin{itemize}
\item $\mrm{deg}\ h^{(n)}_{2^{n-i}}(X,Y)\ge 2^{n-i}$,
\item $h^{(n)}_{2^{n-i}}(X,Y)\in X\mbb{Z}_p[X,Y]$ and 
\item $\vphi^n(gu-u)=\sum^n_{i=0} h^{(n)}_{2^{n-i}}(gu-u,u)p^{(r+1)i}$ for any $g\in G$.
\end{itemize}
\end{lemma}
\begin{proof}
To simplify notation, we put $u_g=gu-u$.
The idea of the proof here is similar to \cite[Lemma 2.2.2]{CL}.
We proceed a proof by induction on $m=n$.
We consider the case $m=1$.
We have $$
\vphi(u_g)=\sum^p_{i=1}a_i(gu^i-u^i)
=\sum^p_{i=1}\sum^{i-1}_{j=0}a_i
\begin{pmatrix}
i\\
j
\end{pmatrix}
u_g^{i-j}u^j
=\sum^p_{i=2}\sum^{i-1}_{j=0}a_i
\begin{pmatrix}
i\\
j
\end{pmatrix}
u_g^{i-j}u^j + a_1u_g.
$$
Putting $h^{(1)}_{2}(X,Y)=\sum^p_{i=2}\sum^{i-1}_{j=0}a_i
\begin{pmatrix}
i\\
j
\end{pmatrix}
X^{i-j}Y^j$ and $h^{(1)}_{1}(X,Y)=p^{-r-1}a_1X$,
we have the desired result. We remark that $p^{-r-1}a_1\in \mbb{Z}_p$
by the assumption $v_p(a_1)>r$.

We suppose that the assertion holds for $m=n$. Thus we have
$$
\vphi^{n+1}(u_g)
= \vphi(\sum^n_{i=0} h^{(n)}_{2^{n-i}}(u_g,u)p^{(r+1)i})
= \sum^n_{i=0} h^{(n)}_{2^{n-i}}(\vphi(u_g),\vphi(u))p^{(r+1)i}.
$$
Write  $h^{(n)}_{2^{n-i}}(X,Y)=\sum^{2^{n-i}}_{j=0}c^{(i)}_j(X,Y)X^{2^{n-i}-j}Y^j$
with some $c^{(i)}_j(X,Y)\in \mbb{Z}_p[X,Y]$ 
for $0\le j < 2^{n-i}$ and $c^{(i)}_{2^{n-i}}(X,Y)\in X\mbb{Z}_p[X,Y]$.
We also write $f(u)=u^2h(u)+a_1u$ with $h(u)\in \mbb{Z}_p[u]$.
Then, we have 
\begin{align*}
 h^{(n)}_{2^{n-i}}(\vphi(u_g),\vphi(u))
& = \sum^{2^{n-i}}_{j=0}c^{(i)}_j(\vphi(u_g),\vphi(u))\vphi(u_g)^{2^{n-i}-j}\vphi(u)^j\\
& = \sum^{2^{n-i}}_{j=0}c^{(i)}_j(\vphi(u_g),\vphi(u))
(h^{(1)}_2(u_g,u)+a_1u_g)^{2^{n-i}-j}(u^2h(u)+a_1u)^j\\
& = \sum^{2^{n-i}}_{j=0} \sum^{2^{n-i}-j}_{\ell=0}\sum^j_{\ell'=0} 
c^{(i)}_j(\vphi(u_g),\vphi(u))
\begin{pmatrix}
2^{n-i}-j\\
\ell
\end{pmatrix}
\begin{pmatrix}
j\\
\ell'
\end{pmatrix}\\
& \qquad \qquad \qquad \cdot (h^{(1)}_2(u_g,u))^{2^{n-i}-j-\ell} (a_1u_g)^{\ell} (u^2h(u))^{j-\ell'} (a_1u)^{\ell'}\\
& = \sum^{2^{n-i}}_{j=0} \sum^{2^{n-i}-j}_{\ell=0}\sum^j_{\ell'=0} 
\left\{c^{(i)}_j(\vphi(u_g),\vphi(u))
\begin{pmatrix}
2^{n-i}-j\\
\ell
\end{pmatrix}
\begin{pmatrix}
j\\
\ell'
\end{pmatrix}
\left(\frac{a_1}{p^{r+1}}\right)^{\ell+\ell'} h(u)^{j-\ell'}  \right\}\\
& \qquad \qquad \qquad 
\cdot \left\{(h^{(1)}_2(u_g,u))^{2^{n-i}-j-\ell} u_g^{\ell} u^{2j-\ell'}\right\}\cdot p^{(r+1)(\ell+\ell')}.
\end{align*}
We define polynomials 
$h^0_{j,\ell,\ell'}(X,Y), h_{j,\ell,\ell'}(X,Y)\in \mbb{Z}_p[X,Y]$
by
$$
h^0_{j,\ell,\ell'}(X,Y)=
(h^{(1)}_2(X,Y))^{2^{n-i}-j-\ell} X^{\ell} Y^{2j-\ell'}
$$
and 
$$
h_{j,\ell,\ell'}(X,Y):=
\left\{c^{(i)}_j(h^{(1)}_2(X,Y)+a_1X,f(Y))
\begin{pmatrix}
2^{n-i}-j\\
\ell
\end{pmatrix}
\begin{pmatrix}
j\\
\ell'
\end{pmatrix}
\left(\frac{a_1}{p^{r+1}}\right)^{\ell+\ell'} h(Y)^{j-\ell'} \right\}
h^0_{j,\ell,\ell'}(X,Y).
$$
Then we have inequalities 
\begin{equation}
\label{polyineq}
\mrm{deg}\ h_{j,\ell,\ell'}(X,Y)\ge 
\mrm{deg}\ h^0_{j,\ell,\ell'}(X,Y)\ge 2(2^{n-i}-j-\ell)+\ell+(2j-\ell')
=2^{n+1-i}-(\ell+\ell').
\end{equation}
We claim $h_{j,\ell,\ell'}(X,Y)\in X\mbb{Z}_p[X,Y]$.
If $\ell\not=0$, then it follows from $h^0_{j,\ell,\ell'}(X,Y)\in X\mbb{Z}_p[X,Y]$.
Suppose $\ell=0$.
If $j<2^{n-i}$.
Then $(h^{(1)}_2(X,Y))^{2^{n-i}-j-\ell}$ is contained in $ X\mbb{Z}_p[X,Y]$,
and thus we have  $h^0_{j,\ell,\ell'}(X,Y)\in X\mbb{Z}_p[X,Y]$.
Suppose $j=2^{n-i}$. 
Since $c^{(i)}_{2^{n-i}}(X,Y)\in X\mbb{Z}_p[X,Y]$,
we have .
$
c^{(i)}_j(h^{(1)}_2(X,Y)+a_1X,f(Y))\in (h^{(1)}_2(X,Y)+a_1X)\mbb{Z}_p[X,Y]
\subset X\mbb{Z}_p[X,Y].
$
Hence the claim follows.

On the other hand, we have
\begin{align*}
\vphi^{n+1}(u_g) 
& =
\sum^n_{i=0} h^{(n)}_{2^{n-i}}(\vphi(u_g),\vphi(u))p^{(r+1)i}\\
& = \sum^n_{i=0} \sum^{2^{n-i}}_{j=0} \sum^{2^{n-i}-j}_{\ell=0}\sum^j_{\ell'=0}   
h_{j,\ell,\ell'}(u_g,u) p^{(r+1)(i+\ell+\ell')}\\
& = \sum_{(i,j,\ell,\ell')\in S_1} h_{j,\ell,\ell'}(u_g,u)p^{(r+1)(i+\ell+\ell'-(n+1))}
p^{(r+1)(n+1)}\\
& \qquad +
\sum_{(i,j,\ell,\ell')\in S_2} h_{j,\ell,\ell'}(u_g,u)
p^{(r+1)(i+\ell+\ell')}.
\end{align*}
Here, $S_1$ (resp.\ $S_2$) is the set of $(i,j,\ell,\ell')$ such that 
$0\le i\le n$, $0\le j\le 2^{n-i}$, $0\le \ell\le 2^{n-i}-j$,
$0\le \ell' \le j$ and $i+\ell+\ell'\ge n+1$ 
(resp.\ $i+\ell+\ell'< n+1$).
Now we set 
$$
h^{(n+1)}_1(X,Y):=\sum_{(i,j,\ell,\ell')\in S_1} h_{j,\ell,\ell'}(X,Y)p^{(r+1)(i+\ell+\ell'-(n+1))}
\in X\mbb{Z}_p[X,Y]
$$
and 
$$
h^{(n+1)}_{2^{n+1-m}}(X,Y):=
\sum_{(i,j,\ell,\ell')\in S_2 \atop i+\ell+\ell'=m} h_{j,\ell,\ell'}(X,Y)
\in X\mbb{Z}_p[X,Y]
$$
for  $0\le m<n+1$.
We clearly have 
 $\vphi^{n+1}(u_g)=\sum^{n+1}_{m=0} h^{(n+1)}_{2^{n+1-m}}(u_g,u)p^{(r+1)m}$.
Hence we finish a proof if  we show 
$\mrm{deg}\ h_{j,\ell,\ell'}(X,Y)\ge 2^{n+1-m}$
if $0\le m<n+1$. To show this, 
it suffices to show  
$\mrm{deg}\ h_{j,\ell,\ell'}(X,Y)\ge 2^{n+1-m}$
if $(i,j,\ell,\ell')\in S_2$ and $m=i+\ell+\ell'<n+1$.
By \eqref{polyineq}, it is enough to show 
$2^{n+1-i}-(\ell+\ell')\ge 2^{n+1-m}$ for such $(i,j,\ell,\ell')$.
Since we have $2^{i_0}-\ell_0\ge 2^{i_0-\ell_0}$ 
for integers $\ell_0 < i_0$,
the desired inequality follows by setting 
$i_0:=n+1-i$ and $\ell_0:=\ell+\ell'$. 
\end{proof}

\begin{corollary}
\label{polycor}
Assume that $v_p(a_1)>r$.  Put $u_g=gu-u$
for any $g\in G$.
Then, $\vphi^n(u_g)/(u_gp^{nr})$ converges to zero $p$-adically in $B^+_{\mrm{cris}}$.
\end{corollary}
\begin{proof}
Let $h^{(n)}_{2^{n-i}}(X,Y)\in X\mbb{Z}_p[X,Y]$ be as in Lemma \ref{poly}
and write 
$$
h^{(n)}_{2^{n-i}}(X,Y)=X\sum^{2^{n-i}-1}_{m=0}h^{n,i}_m(X,Y)X^{2^{n-i}-1-m}Y^m
$$
with some $h^{n,i}_m(X,Y)\in \mbb{Z}_p[X,Y]$.
Then we have 
\begin{align*}
\frac{\vphi^n(u_g)}{p^{nr}u_g}
& = \frac{1}{p^{nr}}\sum^n_{i=0} \sum^{2^{n-i}-1}_{m=0}
h^{n,i}_m(u_g,u)u_g^{2^{n-i}-1-m}u^mp^{(r+1)i}\\
& = \sum^n_{i=0} \sum^{2^{n-i}-1}_{m=0} \sum^{2^{n-i}-1-m}_{m'=0}
(-1)^{m'}h^{n,i}_m(u_g,u)\frac{gu^{2^{n-i}-1-m-m'}u^{m+m'}p^{(r+1)i}}{p^{nr}}.
\end{align*}
Hence it suffices to show that there exists a constant $C(n)>0$ 
which satisfies the following properties:
\begin{itemize}
\item $C(n)\to \infty$ as $n\to \infty$ and 
\item $\displaystyle \frac{gu^{2^{n-i}-1-m-m'}u^{m+m'}p^{(r+1)i}}{p^{nr}}\in p^{C(n)}A_{\mrm{cris}}$
for any $0\le i\le n$, $0\le m\le 2^{n-i}-1$ and $0\le m'\le 2^{n-i}-1-m$.
\end{itemize}
Let $i,m$ and $m'$ be as above.
Let $q,q',r''$ and $r'$ be non-negative integers such that 
$2^{n-i}-1-m-m'=epq+r''$, $m+m'=epq'+r'$, $0\le r'',r'< ep$.
Then  we have 
$$
\frac{gu^{2^{n-i}-1-m-m'}u^{m+m'}p^{(r+1)i}}{p^{nr}}
=g\left(\frac{u^{ep}}{p}\right)^q\left(\frac{u^{ep}}{p}\right)^{q'}\cdot gu^{r''}\cdot u^{r'}
\cdot p^{q+q'+(r+1)i-nr}.
$$
Hence it is enough to find $C(n)$ such that 
$C(n)\to \infty$ as $n\to \infty$ and $q+q'+(r+1)i-nr>C(n)$.
Since we  have an inequality 
$2^{n-i}-1=ep(q+q')+ r+r' < ep(q+q') + 2ep$,
we obtain $q+q'>(ep)^{-1}(2^{n-i}-1)-2>(ep)^{-1}2^{n-i}-3$.

Suppose $r=0$. Take any real number $C$ such that 
$(ep)^{-1}2^x-x>C$ for any $x\in \mbb{R}$.
Then we see
$$
q+q'+(r+1)i-nr=q+q'+i> (ep)^{-1}2^{n-i}+i-3 > C+n-3.
$$
Therefore, if we set $C(n):=C+n-3$, then $C(n)$ satisfies the desired property.

Suppose $r>0$.
Take any real number $C'$ such that 
$(ep)^{-1}2^x-2xr>C'$ for any $x\in \mbb{R}$.
Then we see
\begin{align*}
q+q'+(r+1)i-nr 
& > (ep)^{-1}2^{n-i}-3 -(n-i)r +i \\
& = (ep)^{-1}2^{n-i}-2(n-i)r +(n-i)r+i-3\\
& > C' +(n-i)r+i-3\\
& \ge C'+n-3.
\end{align*}
Therefore, if we set $C(n):=C'+n-3$, then $C(n)$ satisfies the desired property.
\end{proof}

\subsection{Essential image of $\hat{T}$}

The goal of this subsection is to show Theorem \ref{MT'} (2). 
We continue to use the same notation as in previous section.

\begin{lemma}
\label{fix}
For any $\mfM\in {}'\Mod^{r,\hat{G},\mrm{cris}}_{\mfS}$,
we have
$\xi_{\alpha}(M)\subset (B^+_{\mrm{cris}}\otimes_{\vphi, \mfS}\mfM)^G$.
\end{lemma}

\begin{proof}
It suffices to show $g((e_1,\dots ,e_d)Y)=(e_1,\dots ,e_d)Y$
for any $g\in G$.
We define $X_g\in GL_d(W(R))$ by
$$
g(e_1,\dots ,e_d)=(e_1,\dots ,e_d)X_g.
$$
It is enough to show 
$X_gg(Y)=Y$.
Taking $\alpha_g\in B^+_{\mrm{cris}}$
as in Definition \ref{mod},
we know
$
X_g-I_d\in M_d(\alpha_gB^+_{\mrm{cris}}).
$
Hence we have $X_g=I_d+\alpha_gY_g$ for some 
$Y_g\in M_d(B^+_{\mrm{cris}})$.
Furthermore, we have $X_gg\vphi(A)=\vphi(A)\vphi(X_g)$
since $\vphi$ commutes with the $G$-action.
Thus we have 
\begin{align*}
  & X_gg(\vphi(A)\cdots 
\vphi^n(A)\vphi^n(A_0^{-1})\cdots \vphi(A_0^{-1}))\\
= & X_gg\vphi(A)g\vphi^2(A)\cdots 
g\vphi^n(A)\vphi^n(A_0^{-1})\cdots \vphi(A_0^{-1})\\
= & \vphi(A)\vphi(X_g)g\vphi^2(A)\cdots 
g\vphi^n(A)\vphi^n(A_0^{-1})\cdots \vphi(A_0^{-1})\\
= & \vphi(A)\vphi^2(A)\vphi^2(X_g)g\vphi^3(A)\cdots 
g\vphi^n(A)\vphi^n(A_0^{-1})\cdots \vphi(A_0^{-1})\\
= & \cdots \\
= & \vphi(A)\cdots \vphi^n(A)\vphi^n(X_g) 
\vphi^n(A_0^{-1})\cdots \vphi(A_0^{-1})\\
= & \vphi(A)\cdots \vphi^n(A)
\vphi^n(A_0^{-1})\cdots \vphi(A_0^{-1})\\
& + \vphi^n(\alpha_g)\vphi(A)\cdots \vphi^n(A)\vphi^n(Y_g) 
\vphi^n(A_0^{-1})\cdots \vphi(A_0^{-1}).
\end{align*}
Hence the proof completes if we show  that
$Z_n:=\vphi^n(\alpha_g)\vphi(A)\cdots \vphi^n(A)\vphi^n(Y_g) 
\vphi^n(A_0^{-1})\cdots \vphi(A_0^{-1})$
converges to zero $p$-adically in $B^+_{\mrm{cris}}$.
Let $\lambda>0$ be an integer such that 
$p^{\lambda}Y_g\in M_d(A_{\mrm{cris}})$.
Since $\mfM$ is of height $r$,
we see that $Z_n$
is contained in 
$\vphi^n(\alpha_g)/p^{nr} \cdot p^{-\lambda}
M_d(A_{\mrm{cris}})$. Since  $\vphi^n(\alpha_g)/p^{nr}$
converges to zero, we obtain the desired result.
\end{proof}

\begin{proof}[Proof of Theorem \ref{MT'} (2)]
We continue to use the same notation. 
First we assume the condition $v_p(a_1)>\max \{r,1\}$.
Proposition \ref{comp2} (2) and  Lemma \ref{fix},
we have injections
$$
M\overset{\xi_{\alpha}}{\hookrightarrow}
(B^+_{\mrm{cris}}\otimes_{\vphi, \mfS}\mfM)^G
\overset{\hat{\iota}}{\hookrightarrow}
(B^+_{\mrm{cris}}\otimes_{\mbb{Z}_p}\hat{T}(\hat{\mfM})^{\vee})^G.
$$
Hence we the equality $\mrm{dim}_{\mbb{Q}_p}(B^+_{\mrm{cris}}
\otimes_{\mbb{Z}_p}\hat{T}(\hat{\mfM})^{\vee})^G
=\mrm{dim}_{\mbb{Q}_p} \hat{T}(\hat{\mfM})^{\vee}[1/p]$.
This implies that $\hat{T}(\hat{\mfM})[1/p]$
is a crystalline $\mbb{Q}_p$-representation with non-negative 
Hodge-Tate weights.
In the rest of this proof, we show that the Hodge-Tate
weights of $\hat{T}(\hat{\mfM})[1/p]$ are at most $r$.

From now on, we assume the condition (P).
Under this assumption, we know that 
$T_{\mfS}$ is fully faithful (cf.\ Proposition \ref{KisinFF}). 
Put $V=\hat{T}(\hat{\mfM})[1/p]$ and $D=D_{\mrm{cris}}(V)$.
Take an integer $r'>0$  such that Hodge-Tate weights 
of $V$ are at most $r'$.
Let $\mcal{M}=\mcal{M}(D)$ be the $\vphi$-module over $\mfO$
corresponding to $D$ and take any 
free Kisin module $\mfM'\subset \mcal{M}$ of height $r'$
such that $\mfO\otimes_{\mfS} \mfM'=\mcal{M}$.

We claim that $\mfM'$ is of height $r$.
Note that  $T_{\mfS}(\mfM)$ and $T_{\mfS}(\mfM')$
are lattices of $V$.
By replacing $\mfM'$ with some $p^{\ell}\mfM$, 
we may assume that we have 
$T_{\mfS}(\mfM)\subset T_{\mfS}(\mfM')$.
Let $c>0$ be an integer such that 
$T_{\mfS}(\mfM')\subset p^{-c}T_{\mfS}(\mfM)$.
We consider the following commutative diagram.
\begin{center}
$\displaystyle \xymatrix{
T_{\mfS}(\mfM)  
\ar@{^{(}->}[r] \ar@{^{(}->}[rrd] 
& T_{\mfS}(\mfM') \ar@{^{(}->}[r] 
& p^{-c}T_{\mfS}(\mfM) \ar^{\simeq}[d]
\\
& &  
T_{\mfS}(p^c\mfM)
}$
\end{center}
Here, $p^{-c}T_{\mfS}(\mfM)\overset{\simeq}{\rightarrow} 
T_{\mfS}(p^c\mfM)$ in the diagram 
is the map given by $f\mapsto f|_{p^c\mfM}$,
and the other arrows are natural injections.
Since $T_{\mfS}$ is fully faithful,
we obtain maps 
$\eta'\colon \mfM'\to \mfM$ and 
$\eta\colon p^c\mfM\to \mfM'$ such that 
$\eta'\circ \eta$ is the inclusion map
$p^c\mfM\hookrightarrow \mfM$.
We see that $\eta$ and $\eta'$ are injective and 
$p^c\mfM\subset \eta'(\mfM')$.
We regard $\mfM'$ as a $\vphi$-stable submodule 
of $\mfM$ by $\eta'$.
Since $\mfM/\mfM'$ is killed by a power of $p$,
Proposition \ref{BASIC2} shows  that the natural map
$\mfM/\mfM'\to \cO_{\E}\otimes_{\mfS} \mfM/\mfM'$
is injective.
Thus we obtain the fact that $\mfM/\mfM'$  is $p'$-torsion free
in the sense that 
$\mrm{Ann}_{\mfS}(\mfM/\mfM')$ is zero or of the form $p^{\ell}\mfS$.
It follows from \cite[Proposition B.1.3.5]{Fo2}
that $\mfM'$ is of height $r$. 
In particular, $\mcal{M}$ is of height $r$.

Note that $\xi'_{\alpha}$ induces an isomorphism 
$\vphi^{\ast}\mcal{M}/E(u)\vphi^{\ast}\mcal{M}
\simeq K\otimes_{K_0}D=:D_K$.
If we define a decreasing filtration 
$\mrm{Fil}^i\mcal{M}$ of $\mcal{M}$ by
$$
\mrm{Fil}^i\mcal{M}=\{x\in \vphi^{\ast} \mcal{M}
\mid (1\otimes \vphi)(x)\in E(u)^i\mcal{M}\},
$$ 
then the natural projection 
$$
\vphi^{\ast}\mcal{M}\twoheadrightarrow 
\vphi^{\ast}\mcal{M}/E(u)\vphi^{\ast}\mcal{M}
\simeq D_K
$$
is strict compatible with filtrations (cf.\ \cite[Corollary 4.2.4]{CL}).
Since $\mcal{M}$ is of height $r$,
we have
$\mrm{Fil}^{r+1}\vphi^{\ast}\mcal{M}\subset E(u)\vphi^{\ast}\mcal{M}$,
which induces the fact $\mrm{Fil}^{r+1}D_K=0$
as desired.
\end{proof}

\subsection{Essential surjectiveness of $\hat{T}$}

We show Theorem \ref{MT'} (3). 
Let  $T$ be an object of $\mrm{Rep}^{r,\mrm{cris}}_{\mbb{Z}_p}(G)$
and put $V=T[1/p]$.
Let $D=D_{\mrm{cris}}(V)$ be the filtered $\vphi$-module
corresponding to $V$.
Throughout this subsection,
we identify $V$ with 
$V_{\mrm{cris}}(D)
=\mrm{Hom}_{K_0,\vphi}(D,B^+_{\mrm{cris}})
\cap \mrm{Hom}_{K,\mrm{Fil}}(D_K, B^+_{\mrm{dR}})
(\subset \mrm{Hom}_{K_0}(D,B^+_{\mrm{cris}})).$
Let $\mcal{M}=\mcal{M}(D)$ be the $\vphi$-module over $\mfO$
corresponding to $D$.
By Theorem 4.4.1 of {\it loc.\ cit.},
there exists a Kisin module $\mfM\subset \mcal{M}$ of height $r$
such that $\mfO\otimes_{\mfS} \mfM=\mcal{M}$.
In Section \ref{technical},
we defined a $G_{\underline{\pi}}$-equivalent injection
$\iota_0\colon T_{\mfS}(\mfM)\hookrightarrow V$. 
The image of $\iota_0$
might  not   coincide with $T$. However, we have

\begin{lemma}
\label{=T}
Assume the condition $(P)$. Then
we can choose $\mfM$ so that 
$\iota_0(T_{\mfS}(\mfM))=T$.
\end{lemma}
\begin{proof}
We identify $V$ with $\mrm{Hom}_{\mfO,\vphi,\mrm{Fil}}
(\vphi^{\ast}\mcal{M}, B^+_{\mrm{cris}})$ by 
isomorphisms
$\mrm{Hom}_{\mfO,\vphi,\mrm{Fil}}
(\vphi^{\ast}\mcal{M}, B^+_{\mrm{cris}})
\simeq 
\mrm{Hom}_{\mfO_{\alpha},\vphi,\mrm{Fil}}
(\mfO_{\alpha}\otimes_{\mfO}\vphi^{\ast}\mcal{M}, B^+_{\mrm{cris}})
\simeq 
\mrm{Hom}_{\mfO_{\alpha},\vphi,\mrm{Fil}}
(\mfO_{\alpha}\otimes_{K_0} D, B^+_{\mrm{cris}})
\simeq V_{\mrm{cris}}(D)=V$
(see the definition of $\iota_0$).
Under this identification,
$\iota_0$ is the injection
$$
\iota_0\colon T_{\mfS}(\mfM)=\mrm{Hom}_{\mfS,W(R)}(\mfM,W(R)) \to 
\mrm{Hom}_{\mfO,\vphi,\mrm{Fil}}
(\vphi^{\ast}\mcal{M}, B^+_{\mrm{cris}})=V
$$
given by 
$\iota_0(f)(a\otimes x)=a\vphi(f(x))$
for $f\in T_{\mfS}(\mfM)$, $a\in \cO$, $x\in \mfM$
(here we identify $\vphi^{\ast}\mcal{M}$ with 
$\cO\otimes_{\vphi, \mfS}\mfM$).
Put $L=\iota_0(T_{\mfS}(\mfM))$.
For any integer $\ell\ge 0$, we have 
natural injections  
$T_{\mfS}(p^{-\ell}\mfM)
\hookrightarrow T_{\mfS}(\mfM)
\hookrightarrow T_{\mfS}(p^{\ell}\mfM)$ 
induced by embeddings $p^{\ell}\mfM\subset \mfM\subset p^{-\ell}\mfM$.
It is not difficult to check the equality
$\iota_0(T_{\mfS}(p^{\pm \ell}\mfM))=p^{\mp \ell}L$.
Thus by replacing $\mfM$ with $p^{-\ell}\mfM$ for $\ell$ large enough,
we may assume that $L$ is a submodule of $T$.
Let $N\to M$ be the morphism of free \'etale $\vphi$-modules 
which corresponds to the natural injection $L\hookrightarrow T$.
This implies that we have the following commutative diagram:
\begin{center}
$\displaystyle \xymatrix{
T 
& T_{\cO_{\E}}(N) \ar_{\sim \quad }[l] 
\\
L 
\ar@{^{(}->}[u]
& T_{\cO_{\E}}(M) 
\ar@{^{(}->}[u] \ar_{\sim \quad }[l] 
}$
\end{center}
We denote by $\eta$ the  isomorphism $T_{\cO_{\E}}(M)\simeq L$
in the diagram.
We see that  $N\to M$ is injective and $M/N$ is a torsion \'etale $\vphi$-module
killed by $p^c$. Here, $c$ is any integer $c>0$ such that $p^c$ kills $T/L$.
Let $\mfrak{g}\colon \cO_{\E}\otimes_{\mfS}\mfM \overset{\sim}{\rightarrow} M$ 
be the morphism of 
\'etale $\vphi$-modules which corresponds to  the composition 
$T_{\cO_{\E}}(M)\underset{\eta}{\rightarrow} L\underset{\iota_0^{-1}}{\rightarrow}
T_{\mfS}(\mfM)\overset{\sim}{\rightarrow} T_{\cO_{\E}}(\cO_{\E}\otimes_{\mfS} \mfM)$.
We have the following commutative diagram:
\begin{center}
$\displaystyle \xymatrix{
T_{\cO_{\E}}(M) \ar[r] \ar^{\quad \sim}[r] \ar_{\quad \eta}[r] \ar@{=}[d]
& L \ar^{\sim\quad }[r] \ar_{\iota_0^{-1}\quad }[r]
& T_{\mfS}(\mfM) \ar^{\sim \qquad }[r] 
& T_{\cO_{\E}}(\cO_{\E}\otimes_{\mfS} \mfM)
\ar@{=}[d]
\\
T_{\cO_{\E}}(M) \ar@{^{(}->}[rrr]  \ar^{T_{\cO_{\E}}(\mfrak{g})}[rrr]
& 
&
& T_{\cO_{\E}}(\cO_{\E}\otimes_{\mfS} \mfM)
}$
\end{center}
Let $\mrm{pr}\colon M\to M/N$ be the natural projection.
Then $\mfN':=\mrm{ker}(\mrm{pr}\circ \mfrak{g})\subset \mfM$
is a $\vphi$-module of height $r$ by \cite[Proposition B.1.3.5]{Fo2}.
Put $\mfN=\mfN'[1/p]\cap (\cO_{\E}\otimes_{\mfS} \mfN')$.
It follows from  \cite[Lemma 3.3.4]{CL} that 
$\mfN$ is a free Kisin module of height $r$.
By the condition (P) and \cite[Proposition 3.3.5]{CL},
the embedding 
$\cO_{\E}\otimes_{\mfS} \mfN = \cO_{\E}\otimes_{\mfS} \mfN'
\hookrightarrow \cO_{\E}\otimes_{\mfS} \mfM$
induces an embedding 
$\mfN\hookrightarrow \mfM$.
We see that we have an isomorphism 
$\mfrak{g}'\colon \cO_{\E}\otimes_{\mfS}\mfN \overset{\sim}{\rightarrow} N$
which makes the diagram 
\begin{center}
$\displaystyle \xymatrix{
N \ar@{_{(}->}[d] 
&  \cO_{\E}\otimes_{\mfS} \mfN \ar@{_{(}->}[d]
\ar_{\sim}[l] \ar^{\mfrak{g}'}[l]
\\
M 
& \cO_{\E}\otimes_{\mfS} \mfM
 \ar_{\sim}[l] \ar^{\mfrak{g}}[l]
}$
\end{center}
commutative.
Here we consider the following commutative diagram.
\begin{center}
$\displaystyle \xymatrix{
T \ar[r] \ar^{\sim \quad }[r] 
& T_{\cO_{\E}}(N)  \ar^{\sim\quad }[r] \ar_{T_{\cO_{\E}}(\mfrak{g}')\quad }[r]
& T_{\cO_{\E}}(\cO_{\E}\otimes_{\mfS} \mfN) 
& T_{\mfS}(\mfN) \ar_{\qquad \sim}[l] 
\\
L \ar[r] \ar^{\sim \quad }[r] \ar_{\eta \quad }[r] 
\ar@{^{(}->}[u]
& T_{\cO_{\E}}(M)  \ar^{\sim\quad }[r] \ar_{T_{\cO_{\E}}(\mfrak{g}) \quad }[r]
\ar@{^{(}->}[u]
& T_{\cO_{\E}}(\cO_{\E}\otimes_{\mfS} \mfM) 
\ar@{^{(}->}[u]
& T_{\mfS}(\mfM) \ar@{^{(}->}[u] \ar_{\qquad \sim}[l] 
}$
\end{center}
The  composite map
$L\overset{\sim}{\to} T_{\cO_{\E}}(M)\overset{\sim}{\to} 
T_{\cO_{\E}}(\cO_{\E}\otimes_{\mfS} \mfM)\overset{\sim}{\to}  T_{\mfS}(\mfM) $
in the diagram is just $\iota^{-1}_0$.
It suffices to show that
the inverse $\iota_0'$ of the composite map
$T\overset{\sim}{\to} T_{\cO_{\E}}(N)\overset{\sim}{\to} 
T_{\cO_{\E}}(\cO_{\E}\otimes_{\mfS} \mfN)\overset{\sim}{\to}  T_{\mfS}(\mfN) $
is just $\iota_0\colon T_{\mfS}(\mfN) \hookrightarrow V$.

Since $\mfM/\mfN'\subset M/N$ is killed by $p^c$,
we have $p^c\mfM\subset \mfN' \subset \mfN \subset \mfM$.
Consider the following diagram:
\begin{center}
$\displaystyle \xymatrix{
T_{\mfS}(\mfM) \ar^{\iota_0}[d] \ar_{\simeq}[d]  \ar@{^{(}->}[r]
& T_{\mfS}(\mfN)  \ar^{\iota_0'}[d] \ar_{\simeq}[d]  \ar@{^{(}->}[r]
& T_{\mfS}(p^c\mfM) \ar^{\iota_0}[d] \ar_{\simeq}[d]
&
\\
L  \ar@{^{(}->}[r]
& T   \ar@{^{(}->}[r]
& p^{-c}L  \ar@{^{(}->}[r]
& V=\mrm{Hom}_{\mfO,\vphi,\mrm{Fil}}
(\vphi^{\ast}\mcal{M}, B^+_{\mrm{cris}})
}$
\end{center}
The biggest square in the diagram clearly commutes.
The left square in the  diagram also commutes by definition of $\iota_0'$.
Thus we see that the right square commutes.
This implies that 
$\iota_0'$ is the map $\iota_0\colon T_{\mfS}(p^c\mfM)\hookrightarrow V$
restricted to $T_{\mfS}(\mfN)$, which must coincide with 
$\iota_0\colon T_{\mfS}(\mfN) \hookrightarrow V$.
\end{proof}

{\it In the rest of this subsection,
we always assume the condition (P) and $v_p(a_1)>\max \{r,1\}$.}
Let $\mfM$ be as in Lemma \ref{=T}.
Then $\iota_0\colon T_{\mfS}(\mfM) \hookrightarrow V$
 induces an isomorphism $T_{\mfS}(\mfM)\simeq  T$.
By this isomorphism, we equip $T_{\mfS}(\mfM)$
with a $G$-action. 
Here, we consider the following diagram:
\begin{center}
\begin{equation}
\label{diag}
\displaystyle \xymatrix{
B^+_{\mrm{cris}}\otimes_{K_0} D  \ar@{^{(}->}[r]  \ar^{\xi'_{\alpha}}_{\simeq}[d] 
& \mrm{Hom}_{\mbb{Q}_p}(V_{\mrm{cris}}(D), B^+_{\mrm{cris}})  \ar^{\sim}[r]  
& B^+_{\mrm{cris}}\otimes_{\mbb{Q}_p} V_{\mrm{cris}}(D)^{\vee} \ar@{=}[r]
& B^+_{\mrm{cris}}\otimes_{\mbb{Z}_p} T^{\vee}\ar@{=}[d]\\
B^+_{\mrm{cris}}\otimes_{\mfS} \vphi^{\ast}\mfM  
\ar^{\vphi^{\ast}\iota_{\mfS}}[rr] \ar@{^{(}->}[rr] 
&  
& B^+_{\mrm{cris}}\otimes_{\mbb{Z}_p} T_{\mfS}(\mfM)^{\vee} 
& B^+_{\mrm{cris}}\otimes_{\mbb{Z}_p} T^{\vee} 
\ar^{\qquad \iota^{\vee}_0}[l] \ar_{\qquad \sim}[l]\\
W(R)\otimes_{\vphi, \mfS} \mfM  \ar^{\vphi^{\ast}\iota_{\mfS}}[rr] \ar@{^{(}->}[rr]  \ar@{^{(}->}[u]
& 
& W(R)\otimes_{\mbb{Z}_p} T_{\mfS}(\mfM)^{\vee} 
& W(R)\otimes_{\mbb{Z}_p} T^{\vee} 
\ar^{\qquad \iota^{\vee}_0}[l] \ar_{\qquad \sim}[l] \ar@{^{(}->}[u]
}
\end{equation}
\end{center}
The square 
\begin{center}
$\displaystyle \xymatrix{
B^+_{\mrm{cris}}\otimes_{\mfS} \vphi^{\ast}\mfM  \ar@{^{(}->}[r]  
& B^+_{\mrm{cris}}\otimes_{\mbb{Z}_p} T^{\vee} \\
W(R)\otimes_{\vphi, \mfS} \mfM  \ar@{^{(}->}[r]  \ar@{^{(}->}[u]
& W(R)\otimes_{\mbb{Z}_p} T^{\vee}  \ar@{^{(}->}[u]
}$
\end{center}
in the above diagram is clearly commutative. Furthermore, by direct computations, 
we can check that  the square
\begin{center}
$\displaystyle \xymatrix{
B^+_{\mrm{cris}}\otimes_{K_0} D  \ar@{^{(}->}[r]  \ar^{\xi'_{\alpha}}_{\simeq}[d] 
& B^+_{\mrm{cris}}\otimes_{\mbb{Z}_p} T^{\vee}\ar@{=}[d]\\
B^+_{\mrm{cris}}\otimes_{\mfS} \vphi^{\ast}\mfM  \ar@{^{(}->}[r]  
& B^+_{\mrm{cris}}\otimes_{\mbb{Z}_p} T^{\vee} \\
}$
\end{center}
in the  diagram 
is also commutative (here we note that 
$\xi'_{\alpha}$ appears in the definition of $\iota_0$).
Hence, seeing the biggest square in  the diagram \eqref{diag}, 
we obtain a commutative diagram
\begin{center}
$\displaystyle \xymatrix{
B^+_{\mrm{cris}}\otimes_{K_0} D  \ar@{^{(}->}[r]  
& B^+_{\mrm{cris}}\otimes_{\mbb{Z}_p} T^{\vee} \\
W(R)\otimes_{\vphi, \mfS} \mfM  \ar@{^{(}->}[r]  \ar@{^{(}->}[u]
& W(R)\otimes_{\mbb{Z}_p} T^{\vee}  \ar@{^{(}->}[u]
}$
\end{center}
By this diagram,
we regard $B^+_{\mrm{cris}}\otimes_{K_0} D,W(R)\otimes_{\mbb{Z}_p} T^{\vee}$
and $W(R)\otimes_{\vphi, \mfS} \mfM$
as $\vphi$-stable submodules of $B^+_{\mrm{cris}}\otimes_{\mbb{Z}_p} T^{\vee}$.
Note that  $B^+_{\mrm{cris}}\otimes_{K_0} D$ and 
$W(R)\otimes_{\mbb{Z}_p} T^{\vee}$ are 
$G$-stable submodules of $B^+_{\mrm{cris}}\otimes_{\mbb{Z}_p} T^{\vee}$.


\begin{lemma}
\label{stable}
Let the notation be as above.

\noindent
$(1)$ $G_{\underline{\pi}}$ acts on $\vphi^{\ast}\mfM$ trivial.

\noindent
$(2)$ The $G$-action on 
$W(R)\otimes_{\mbb{Z}_p} T^{\vee}$ preserves $W(R)\otimes_{\vphi, \mfS} \mfM$.

\noindent
$(3)$ The $G$-action on 
$W(R)\otimes_{\vphi, \mfS} \mfM$  commutes with $\vphi$.

\noindent
$(4)$ $G(\vphi^{\ast}\mfM)\subset \whR\otimes_{\vphi,\mfS} \mfM$.
\end{lemma}

\begin{proof}
(1) is trivial. If  we admit (2), the statement 
(3) follows from the fact that
$W(R)\otimes_{\vphi, \mfS} \mfM$ 
 is a $\vphi$-stable submodule of 
$W(R)\otimes_{\mbb{Z}_p} T^{\vee}$
and the $G$-action on $W(R)\otimes_{\mbb{Z}_p} T^{\vee}$
commutes with $\vphi$.
Hence it suffices to show (2) and (4).

We show (2). 
Take any $g\in G$.
Let $e_1,\dots ,e_d$ be a basis of $\vphi^{\ast}\mfM$.
Note that this
is also a basis of $B^+_{\mrm{cris}}\otimes_{K_0} D$.
Hence we have $g(e_1,\dots,e_d )=(e_1,\dots ,e_d)X_g$ 
for some $X_g\in GL_d(B^+_{\mrm{cris}})$.
By Proposition \ref{comp2} (2),
$\vphi(\mft)^rg(e_1,\dots,e_d )=(e_1,\dots,e_d )A_g$
for some $A_g\in M_d(W(R))$.
Hence we have $\vphi(\mft)^rX_g=A_g\in M_d(W(R))$.
Note that $\vphi(\mft)$ is a generator of $I^{[1]}W(R)$
by \cite[Proposition 5.1.3]{Fo1}.
Hence we have $X_g\in M_d(W(R))$ by \cite[Lemma 3.1.3]{Li4}.

Finally we show (4). By (2),
it suffices to show that $X_g$ has coefficients in $\mcal{R}_{K_0}$.
Put  $M:=\vphi^{\ast}\mfM/u\vphi^{\ast}\mfM$.
Let $\xi_{\alpha}\colon \mfO_{\alpha}\otimes_{W(k)} M\overset{\sim}{\longrightarrow} 
\mfO_{\alpha}\otimes_{\mfS} \vphi^{\ast}\mfM$,
$\xi'_{\alpha}\colon \mfO_{\alpha}\otimes_{K_0} D\overset{\sim}{\longrightarrow} 
\mfO_{\alpha}\otimes_{\mfS} \vphi^{\ast}\mfM$ and $Y$ be as in 
Section \ref{technical}.
By \cite[Corollary 4.5.7]{CL},
we have an equality 
$\xi_{\alpha}(M[1/p])=\xi'_{\alpha}(D)$.
By definition of the $G$-action on $\mfM$,
we know that $B^+_{\mrm{cris}}\otimes_{\mfO_{\alpha}} \xi'_{\alpha}$ is $G$-equivalent 
and thus $G$ acts on $\xi_{\alpha}(M)$ trivial. 
This implies $g((e_1,\dots,e_d )Y)=(e_1,\dots ,e_d)Y$.
Thus we have $X_g=Yg(Y)^{-1}$, which is an element of $GL_d(\mcal{R}_{K_0})$.
\end{proof}

By the above lemma,
we have a natural $\whR$-semi-linear $G$-action on $\whR\otimes_{\vphi,\mfS} \mfM$,
which commutes with $\vphi$.
Since $\mrm{Gal}(\overline{K}/\wh{K}_{\underline{\pi}})$ acts on $\whR$ and $\vphi^{\ast}\mfM$
trivial, the $G$-action on $\whR\otimes_{\vphi,\mfS} \mfM$
factors through $\hat{G}$.
Hence
$\mfM$ has a structure of an object of $\Mod^{r,\hat{G}}_{\mfS}$,
which we denote by  $\hat{\mfM}$.

\begin{lemma}
Let the notation be as above. 
Then we have a natural isomorphism $\hat{T}(\hat{\mfM})\simeq T$
of $\mbb{Z}_p[G]$-modules.
\end{lemma}
\begin{proof}
We follow the method of \cite[Section 3.2]{Li2}.
First we recall that we defined a $G$-action on $T_{\mfS}(\mfM)$
by the isomorphism $\iota_0\colon T_{\mfS}(\mfM)\simeq T$,
and also recall that the injection 
$W(R)\otimes_{\vphi, \mfS} \mfM\hookrightarrow W(R)\otimes_{\mbb{Z}_p} T_{\mfS}(\mfM)^{\vee}$
is $G$-equivalent by definition of the $G$-action on $W(R)\otimes_{\vphi, \mfS} \mfM$. 
We consider the following commutative diagram:
\begin{center}
$\displaystyle \xymatrix{
W(R)\otimes_{\vphi, \mfS} \mfM  \ar@{=}[d] \ar^{\vphi^{\ast}\iota_{\mfS}\ }[r] \ar@{^{(}->}[r]  
& W(R)\otimes_{\mbb{Z}_p} T_{\mfS}(\mfM)^{\vee}  
& W(R)\otimes_{\mbb{Z}_p} T^{\vee} \ar^{\quad \iota^{\vee}_0}[l] \ar_{\quad \sim}[l]
\\
W(R)\otimes_{\vphi, \mfS} \mfM  \ar^{\hat{\iota}\ }[r] \ar@{^{(}->}[r] 
& W(R)\otimes_{\mbb{Z}_p} \hat{T}(\hat{\mfM})^{\vee}  \ar^{\simeq}[u] \ar_{\eta}[u]  &
}$
\end{center}
Here, $\eta:=W(R)\otimes \theta^{\vee}$.
It suffices to show that $\eta$ is $G$-equivalent.
Note that all arrows in the diagram except $\eta$ are known to be 
$G$-equivalent and $\vphi(\mft)W(R)=I^{[1]}W(R)$
is stable under the $G$-action on $W(R)$. 
By Corollary \ref{comp1} and Proposition \ref{comp2} (2),
we can regard $\vphi(\mft)W(R)\otimes_{\mbb{Z}_p} \hat{T}(\hat{\mfM})^{\vee}$
and $\vphi(\mft)W(R)\otimes_{\mbb{Z}_p} T_{\mfS}(\mfM)^{\vee}$
as $G$-stable submodules of  $W(R)\otimes_{\vphi, \mfS} \mfM$,
and thus $\eta$ restricted to 
$\vphi(\mft)W(R)\otimes_{\mbb{Z}_p} \hat{T}(\hat{\mfM})^{\vee}$
induces an $G$-equivalent isomorphism
$\vphi(\mft)W(R)\otimes_{\mbb{Z}_p} \hat{T}(\hat{\mfM})^{\vee}
\simeq \vphi(\mft)W(R)\otimes_{\mbb{Z}_p} T_{\mfS}(\mfM)^{\vee}$.
It follows from this that $\eta$ is $G$-equivalent.
\end{proof}

Finally, we show the following, which completes a proof of  Theorem \ref{MT'} (3).
\begin{lemma}
\label{range}
Let the notation be as above. 
Then $\hat{\mfM}$ is an object of $\Mod^{r,\hat{G},\mrm{cris}}_{\mfS}$.
\end{lemma}

\begin{proof}
Let $\mfrak{e}_1,\dots \mfrak{e}_d$ be a basis of $\mfM$
and let $A\in M_d(\mfS)$ be a matrix such that
$\vphi(\mfrak{e}_1,\dots ,\mfrak{e}_d)
=(\mfrak{e}_1,\dots ,\mfrak{e}_d)A$.
Put $e_i=1\otimes \mfrak{e}_i\in \vphi^{\ast}\mfM$ 
for each $i$.
Then $e_1,\dots ,e_d$ is a basis of $\vphi^{\ast}\mfM$
and $\vphi(e_1,\dots, e_d)
=(e_1,\dots ,e_d)\vphi(A)$.
Put  $M:=\vphi^{\ast}\mfM/u\vphi^{\ast}\mfM$
and $\bar{e}_i=e_i\ \mrm{mod}\ u\vphi^{\ast}\mfM$ 
for each $i$.
Then $\bar{e}_1,\dots ,\bar{e}_d$ is a basis of $M$
and $\vphi(\bar{e}_1,\dots ,\bar{e}_d)
=(\bar{e}_1,\dots ,\bar{e}_d)\vphi(A_0)$
where $A_0=A\ \mrm{mod}\ u\mfS\in M_d(W(k))$.
Take any $g\in G$. 
Let $X_g\in GL_d(\whR)$ be a matrix given by 
$$
g(e_1,\dots ,e_d)=(e_1,\dots ,e_d)X_g.
$$ 
We show  $X_g-I_d\in \vphi(u_g)M_d(B^+_{\mrm{cris}})$ below.

Let $Y$ be as in 
Section \ref{technical}.
We have $X_g=Yg(Y)^{-1}$ (see the proof of Lemma \ref{stable} (4)).
We recall that 
$$
\vphi(A)\cdots \vphi^n(A)\vphi^n(A_0^{-1})\cdots \vphi(A_0^{-1})
$$
converges to $Y$.
Hence the matrix 
$$
\vphi(A_0)\cdots \vphi^n(A_0)g(\vphi^n(A)^{-1}\cdots \vphi(A)^{-1})\in 
GL_d(B^+_{\mrm{cris}})
$$
converges to $g(Y)^{-1}$
and the matrix 
$$
X_{n,g}:=\vphi(A)\cdots \vphi^n(A)g(\vphi^n(A)^{-1}\cdots \vphi(A)^{-1})\in 
GL_d(B^+_{\mrm{cris}})
$$
converges to $X_g$. 

Since $\mfM$ is of height $r$,
there exists a matrix $B\in M_d(\mfS)$ such that 
$AB=E(u)^rI_d$.
It is not difficult to check the following:
\begin{itemize}
\item $\vphi^n(E(u))/p\in \mfS[\![\frac{u^{ep}}{p}]\!]^{\times}\subset A_{\mrm{cris}}^{\times}$
for any $n\ge 1$,
\item $\vphi^n(A),\vphi^n(B)\in GL_d(\mfO_{\alpha})\subset 
GL_d(B^+_{\mrm{cris}})$ for any $n\ge 1$.
\end{itemize}
Setting $u_g:=gu-u\in W(R)$, we claim the following:
\begin{itemize}
\item[(I)] $X_{1,g}-I_d\in \vphi(u_g)M_d(B^+_{\mrm{cris}})$,
\item[(II)] $X_{n,g}-X_{n-1,g}\in \vphi(u_g)M_d(B^+_{\mrm{cris}})$ for $n\ge 2$, and
\item[(III)] $\vphi(u_g)^{-1}(X_{n,g}-X_{n-1,g})\to O
\ (n\to \infty)$ $p$-adically.
\end{itemize}

First we show (I). We have
\begin{align*}
g\vphi(E(u))^r\vphi(B)(X_{1,g}-I_d) 
& = g\vphi(E(u))^r\vphi(B)(\vphi(A)g\vphi(A)^{-1}-I_d)\\
& = g\vphi(E(u))^r \vphi(BA) g\vphi(A)^{-1}
    -g\vphi(E(u))^r\vphi(B)\\
& = \vphi(E(u))^rg\vphi(B)-g\vphi(E(u))^r\vphi(B)\\
& = \vphi(E(u))^r(g\vphi(B)-\vphi(B)) - (g\vphi(E(u))^r - \vphi(E(u))^r)\vphi(B).
\end{align*}
Since we have $g\vphi(B)-\vphi(B)\in \vphi(u_g)M_d(B^+_{\mrm{cris}})$ and 
$g\vphi(E(u))^r - \vphi(E(u))^r\in \vphi(u_g)B^+_{\mrm{cris}}$,
we obtain $g\vphi(E(u))^r\vphi(B)(X_{1,g}-I_d) \in \vphi(u_g)M_d(B^+_{\mrm{cris}})$.
This shows  $X_{1,g}-I_d \in \vphi(u_g)M_d(B^+_{\mrm{cris}})$ as desired. 

Next we show (II) and (III). Suppose $n\ge 2$.
Writing $X_{1,g}=I_d+\vphi(u_g)C_g$ with some $C_g\in M_d(B^+_{\mrm{cris}})$,
we have 
\begin{align*}
X_{n,g}-X_{n-1,g}
& = \vphi(A)\cdots \vphi^{n-1}(A)
(\vphi^n(A)g\vphi^n(A)^{-1}-I_d)
g\vphi^{n-1}(A)^{-1}\cdots g\vphi(A)^{-1}\\
& = \vphi(A)\cdots \vphi^{n-1}(A)
(\vphi^{n-1}(X_{1,g})-I_d)
g\vphi^{n-1}(A)^{-1}\cdots g\vphi(A)^{-1}\\
& = \vphi^n(u_g)\vphi(A)\cdots \vphi^{n-1}(A)
\vphi^{n-1}(C_g)
g\vphi^{n-1}(A)^{-1}\cdots g\vphi(A)^{-1}\\
& = \frac{\vphi^n(u_g)}{p^{(n-1)r}} \cdot \vphi^{n-1}(p^rE(u)^{-r})\cdots \vphi(p^rE(u)^{-r})\cdot 
\vphi(A)\cdots \vphi^{n-1}(A)\\
& \quad \cdot \vphi^{n-1}(C_g)\cdot 
g\vphi^{n-1}(B)\cdots g\vphi(B)\\
& \in \vphi(u_g)\cdot \vphi\left(\frac{\vphi^{n-1}(u_g)}{u_g p^{(n-1)r}}\right)
\cdot p^{-\lambda_g}M_d(A_{\mrm{cris}}).
\end{align*}
Here, $\lambda_g>0$ is an integer such that $p^{\lambda_g} C_g\in M_d(A_{\mrm{cris}})$.
By Corollary \ref{polycor},
we know that $\vphi\left(\vphi^{n-1}(u_g)/(u_g p^{(n-1)r}) \right)$
converges to $0$. This shows (II) and (III).

We set $Y_{n,g}:=\vphi(u_g)^{-1}(X_{n,g}-X_{n-1,g})$ and 
$Y_g:=\sum^{\infty}_{n=2} Y_{n,g}$.These are elements of 
$M_d(B^+_{\mrm{cris}})$ by the claim.
Since we have $X_{n,g}=\sum^n_{k=2}(X_{k,g}-X_{k-1,g})+X_{1,g}
=\vphi(u_g)(\sum^n_{k=2}Y_{k,g}+C_g)+I_d$,
by taking a limit, 
we obtain 
$$
X_g=\vphi(u_g)(Y_g+C_g)+I_d.
$$
Therefore,
we obtain $X_g-I_d\in \vphi(u_g)M_d(B^+_{\mrm{cris}})$.\\

Now we are ready to finish a proof of Lemma \ref{range}.
Take any $g\in G$ and $x\in \mfM$. We want to show 
$g(1\otimes x)-(1\otimes x)\in \vphi(u_g)B^+_{\mrm{cris}}\otimes_{\vphi,\mfS}\mfM$.
Let $\mbf{x}\in M_{d,1}(\vphi(\mfS))$ be a matrix such that  
$1\otimes x=(e_1,\dots e_d)\mbf{x}$. 
Then we have 
$g(1\otimes x)-(1\otimes x)=(e_1,\dots e_d)(X_gg\mbf{x}-\mbf{x})$.
Since we can write $X_g=I_d+\vphi(u_g)X'_g$ by some matrix $X'_g\in M_d(B^+_{\mrm{cris}})$,
we have 
$X_gg\mbf{x}-\mbf{x}=\vphi(u_g)X'_g+(g\mbf{x}-\mbf{x})$.
Thus it suffices to show $g\mbf{x}-\mbf{x}\in \vphi(u_g)M_{d,1}(W(R))$
but this is an easy exercise.
\end{proof}


\subsection{Compatibility of different uniformizers, and Dieudonn\'e crystals}

Suppose the conditions (P) and $v_p(a_1)>\max \{r,1\}$.
Let $T$ be an object of $\mrm{Rep}^{r,\mrm{cris}}_{\mbb{Z}_p}(G)$.
Then there exists a $(\vphi,\hat{G})$-module $\hat{\mfM}$ 
such that $\hat{T}(\hat{\mfM})\simeq T$.
Note that our arguments depends on the choice of a uniformizer $\pi$ of $K$,
a polynomial $f(u)$ and  a system $(\pi_n)_{n\ge 0}$.

If we select a different choice of 
a uniformizer $\pi'$ of $K$,
a polynomial $f'(u)$ and a system  $(\pi'_n)_{n\ge 0}$,
then we get another $(\vphi,\hat{G}')$-module $\hat{\mfM}'$.

\begin{question}
What is the relationship between $\hat{\mfM}$ and $\hat{\mfM}'$ ?
\end{question}

We denote by $\mfS_{\underline{\pi}}$ 
(resp.\ $\mfS_{\underline{\pi}'}$) 
the image of the injection $W(k)[\![u]\!]\to W(R)$ given by 
$u\mapsto \{\underline{\pi}\}_f$
(resp.\ $u\mapsto \{\underline{\pi}'\}_{f'}$).
We may regard $\mfM$ (resp.\ $\mfM'$)
as a $\vphi$-module over $\mfS_{\underline{\pi}}$
(resp.\ $\mfS_{\underline{\pi}'}$). 
Write $\mfS:=\mfS_{\underline{\pi}}$
(resp.\ $\mfS':=\mfS_{\underline{\pi}'}$).
We have comparison morphisms 
$$
\hat{\iota}\colon W(R)\otimes_{\vphi,\mfS}
\mfM\hookrightarrow W(R)\otimes_{\mbb{Z}_p}\hat{T}(\hat{\mfM})^{\vee}
\simeq W(R)\otimes_{\mbb{Z}_p}T^{\vee}
$$
and  
$$
\hat{\iota}'\colon W(R)\otimes_{\vphi,\mfS'} 
\mfM'\hookrightarrow W(R)\otimes_{\mbb{Z}_p}\hat{T}(\hat{\mfM}')^{\vee}
\simeq W(R)\otimes_{\mbb{Z}_p}T^{\vee}
$$

\begin{theorem}
\label{comparison}
Assume  the conditions $(P)$ and $v_p(a_1)>\max \{r,1\}$.
Let the notation be as above. Then we have
$\hat{\iota}(W(R)\otimes_{\vphi,\mfS} \mfM)
=\hat{\iota}(W(R)\otimes_{\vphi,\mfS'} \mfM')$.
In particular, we have a functorial isomorphism 
$W(R)\otimes_{\vphi,\mfS} \mfM
\simeq W(R)\otimes_{\vphi,\mfS'} \mfM'$
which commutes with $\vphi$ and $G$-actions.
\end{theorem}

\begin{proof}
Let $d$ be the $\mbb{Z}_p$-rank of $T$.
Put $M=\vphi^{\ast}\mfM/u\vphi^{\ast}\mfM$.
We have $G$-equivalent injections
$B^+_{\mrm{cris}}\otimes_{W(k)} M
\overset{\xi_{\alpha}}{\simeq} B^+_{\mrm{cris}}\otimes_{\mfS} \vphi^{\ast}\mfM
\overset{\hat{\iota}}{\hookrightarrow} 
B^+_{\mrm{cris}}\otimes_{\mbb{Z}_p} T^{\vee}$.
By Lemma \ref{fix},
we have 
$\xi_{\alpha}(M)\subset (B^+_{\mrm{cris}}\otimes_{\mfS} \vphi^{\ast}\mfM)^G
\overset{\hat{\iota}}{\hookrightarrow}  
(B^+_{\mrm{cris}}\otimes_{\mbb{Z}_p} T^{\vee})^G
\subset D_{\mrm{cris}}(V)$.
Since the $W(k)$-rank of  $M$ is $d$,
we have isomorphisms 
\begin{equation}
\label{is}
M[1/p]
\overset{\xi_{\alpha}}{\simeq}
(B^+_{\mrm{cris}}\otimes_{\mfS} \vphi^{\ast}\mfM)^G
\overset{\hat{\iota}}{\simeq}  
(B^+_{\mrm{cris}}\otimes_{\mbb{Z}_p} T^{\vee})^G
=D_{\mrm{cris}}(V).
\end{equation}
Therefore, we obtain the following diagram:
\begin{center}
$\displaystyle \xymatrix{  
& B^+_{\mrm{cris}}\otimes_{\mfS} \vphi^{\ast}\mfM 
\ar^{\hat{\iota}}[r] \ar@{^{(}->}[r] 
& B^+_{\mrm{cris}}\otimes_{\mbb{Z}_p} T^{\vee} 
\\
B^+_{\mrm{cris}}\otimes_{W(k)} M  
\ar_{\hat{\iota}}[ur] \ar^{\simeq}[ur] 
\ar_{\xi_{\alpha}\qquad}[r] \ar^{\sim \qquad}[r] 
& B^+_{\mrm{cris}}\otimes_{K_0}(B^+_{\mrm{cris}}\otimes_{\mfS} \vphi^{\ast}\mfM) 
\ar[u] 
\ar_{\qquad  \hat{\iota}}[r] 
\ar^{\qquad  \sim}[r]
&
B^+_{\mrm{cris}}\otimes_{K_0} D_{\mrm{cris}}(V)
\ar@{^{(}->}[u] 
}$
\end{center}
Here, two vertical arrows in the diagram are natural maps.
We see that the left vertical arrow is isomorphism by the commutativity
of the diagram.

Let $e_1,\dots ,e_d$ be a basis of $\vphi^{\ast}\mfM$
and $e'_1,\dots ,e'_d$ be a basis of $\vphi^{\ast}\mfM'$.
Seeing the above diagram,
we obtain the fact that 
$\hat{\iota}(e_1),\dots ,\hat{\iota}(e_d)$ is a basis
of $B^+_{\mrm{cris}}\otimes_{K_0} D_{\mrm{cris}}(V)$.
Similarly, $\hat{\iota}'(e'_1),\dots ,\hat{\iota}'(e'_d)$
is also.
Hence there exist a matrix $X\in GL_d(B^+_{\mrm{cris}})$
such that 
$\hat{\iota}(e_1,\dots ,e_d)=\hat{\iota}'(e'_1,\dots ,e'_d)X$.
On the other hand,
if we take any generator $\mft_0$ of $I^{[1]}W(R)$,
we have 
$\mft_0^r\hat{\iota}'(W(R)\otimes_{\mfS} \vphi^{\ast}\mfM')
\subset \mft_0^r(W(R)\otimes_{\mbb{Z}_p} T^{\vee})
\subset \hat{\iota}(W(R)\otimes_{\vphi, \mfS} \vphi^{\ast}\mfM)$.
Thus we obtain  $\mft_0^rX\in M_d(W(R))$.
By \cite[Lemma 3.1.3]{Li4}, 
$X\in M_d(W(R))$. By the similar manner 
we can check $X^{-1}\in M_d(W(R))$.  
This finishes a proof.
(The assertion for the functoriality follows immediately by construction.)
\end{proof}

The following statements gives an affirmative answer of 
\cite[Section 6.3]{CL}.
\begin{corollary}
\label{comparisoncor}
Assume  the conditions $(P)$ and $v_p(a_1)>\max \{r,1\}$.
Let $T$ be an object of $\mrm{Rep}^{r,\mrm{cris}}_{\mbb{Z}_p}(G)$.
Let $\mfM$ $($resp.\ $\mfM'$$)$ be a Kisin module with respect to the choice of 
$(f(u),(\pi_n)_{n\ge 0})$ $($resp.\ 
$(f'(u),(\pi'_n)_{n\ge 0})$$)$
such that $T_{\mfS}(\mfM)\simeq T$ $($resp.\ $T_{\mfS'}(\mfM')\simeq T$$)$.
Then we have a functorial isomorphism 
$W(R)\otimes_{\mfS} \mfM\simeq W(R)\otimes_{\mfS'} \mfM'$
of $\vphi$-modules over $W(R)$.
\end{corollary}

\begin{proof}
Let $\hat{\mfN}$ (resp.\ $\hat{\mfN}'$) be a 
$(\vphi,\hat{G})$-module with respect to the choice of  
$(f(u),(\pi_n)_{n\ge 0})$ (resp.\ 
$(f'(u),(\pi'_n)_{n\ge 0})$)
corresponding to $T$.
By Theorem \ref{comparison},
we have an isomorphism 
$W(R)\otimes_{\vphi,\mfS} \mfN\simeq W(R)\otimes_{\vphi,\mfS'} \mfN'$.
Taking $W(R)\otimes_{\vphi^{-1},W(R)}$,
we obtain an isomorphism
$W(R)\otimes_{\mfS} \mfN\simeq W(R)\otimes_{\mfS'} \mfN'$.
On the other hand, 
we have isomorphisms 
$T_{\mfS}(\mfM)\simeq T|_{G_{\underline{\pi}}}
\simeq \hat{T}(\hat{\mfN})|_{G_{\underline{\pi}}}\simeq T_{\mfS}(\mfN)$.
Similarly, we also have $T_{\mfS'}(\mfM')\simeq T_{\mfS'}(\mfN')$.
By the condition (P) and Proposition \ref{KisinFF},
we have isomorphisms $\mfM\simeq \mfN$ and $\mfM'\simeq \mfN'$.
Thus the result follows.
\end{proof}

\bn
{\bf The case $r\le 1$.}
In the case $r\le 1$,
we can omit the assumption (P) from Theorem \ref{comparison}
and Corollary \ref{comparisoncor}.

\begin{theorem}
\label{comparison:r=1}
Assume  $v_p(a_1)>1$.
Let $T$ be an object of $\mrm{Rep}^{1,\mrm{cris}}_{\mbb{Z}_p}(G)$.
Let $\mfM$ $($resp.\ $\mfM'$$)$ be the Kisin module with respect to the choice of 
$(f(u),(\pi_n)_{n\ge 0})$ $($resp.\ 
$(f'(u),(\pi'_n)_{n\ge 0})$$)$
corresponding to $T$ via Theorem \ref{pdiv}.
Then we have a functorial isomorphism 
$W(R)\otimes_{\mfS} \mfM\simeq W(R)\otimes_{\mfS'} \mfM'$
of $\vphi$-modules over $W(R)$.
\end{theorem}

\begin{proof}
At first, in the proof of Theorem \ref{comparison}, we used the assumption (P) 
to apply Lemma \ref{fix} and to obtain \eqref{is}. 
Following the arguments of \cite[Section 5]{CL},
we can obtain the same result without (P) in the case $r=1$, as follows.

By the arguments of \cite[Section 5]{CL},
we can equip $W(R)\otimes_{\mfS} \mfM$ with a (unique) $G$-action 
which satisfies the following:
\begin{itemize}
\item $G_{\underline{\pi}}$ acts on $\mfM$ trivial, and
\item $g(1\otimes x)-1\otimes x\in \mft M_d(I_+W(R))$ for any $g\in G$ and $x\in \mfM$.
\end{itemize}
(Note that their arguments do not work for $r>1$.)
Moreover, if we equip $T_{\mfS}(\mfM)$ with a $G$-action 
by the isomorphism $T_{\mfS}(\mfM)\simeq \mrm{Hom}_{W(R),\vphi}(W(R)\otimes_{\mfS} \mfM,W(R))$,
then we have  an isomorphism 
$T_{\mfS}(\mfM)\simeq T$
of $\mbb{Z}_p[G]$-modules.
Now we recall how to define a $G$-action on $W(R)\otimes_{\mfS} \mfM$.
Let $\mfrak{e}_1,\dots ,\mfrak{e}_d$ be a basis of $\mfM$ and 
let $A\in M_d(\mfS)$ be the matrix given by 
$\vphi(\mfrak{e}_1,\dots ,\mfrak{e}_d)=(\mfrak{e}_1,\dots ,\mfrak{e}_d)A$.
Set $X'_g:=\lim_{n\to \infty} A\vphi(A)\cdots \vphi^n(A)g\vphi^n(A)^{-1}\cdots g\vphi(A)^{-1}gA^{-1}$,
which is an element of $GL_d(W(R))$.
We put $X_g=\vphi(X'_g)$.
Then we have $X_g=Yg(Y)^{-1}$ where $Y$ is  the matrix defined in Section \ref{technical}.
Hence we see that the composite 
$B^+_{\mrm{cris}}\otimes_{W(k)} M
\overset{\xi_{\alpha}}{\simeq} B^+_{\mrm{cris}}\otimes_{\mfS} \vphi^{\ast}\mfM
\overset{\hat{\iota}}{\hookrightarrow} 
B^+_{\mrm{cris}}\otimes_{\mbb{Z}_p} T^{\vee}$
induces 
$\xi_{\alpha}(M)\subset 
(B^+_{\mrm{cris}}\otimes_{\mfS} \vphi^{\ast}\mfM)^G
\overset{\hat{\iota}}{\hookrightarrow} 
(B^+_{\mrm{cris}}\otimes_{\mbb{Z}_p} T^{\vee})^G$,
which gives 
$M[1/p]
\overset{\xi_{\alpha}}{\simeq}
(B^+_{\mrm{cris}}\otimes_{\mfS} \vphi^{\ast}\mfM)^G
\overset{\hat{\iota}}{\simeq}  
(B^+_{\mrm{cris}}\otimes_{\mbb{Z}_p} T^{\vee})^G
=D_{\mrm{cris}}(V)$
as \eqref{is}.
Then the same arguments as Theorem \ref{comparison} proceeds. 
\end{proof}

\bn
{\bf Comparison with Dieudonn\'e crystals.}
In this section, we give a geometric interpretation of Kisin modules 
in terms of  Dieudonn\'e crystals of $p$-divisible groups under our 
$K_{\underline{\pi}}/K$-setting, which is well-known 
in the Kisin's setting $f(u)=u^p$. 
We recall that (cf.\ Theorem \ref{pdiv}), under the assumption $v_p(a_1)>1$,
there exists an anti-equivalence of categories between 
the category $\Mod^1_{\mfS}$ of free Kisin modules of height $1$ 
and the category of  $p$-divisible groups
over the ring  of integers $\cO_K$ of $K$.

\begin{remark}
\label{pdiv:Kisin}
Consider the Kisin's setting $f(u)=u^p$. In this case Theorem \ref{pdiv}
is well-studied.
Let $S$ be the $p$-adic completion of the  divided power envelope of
the surjection $W[\![u]\!]\twoheadrightarrow \cO_K$ given by $u\mapsto \pi$. 
Let $H$ be a $p$-divisible group over $\cO_K$ and 
$\mfM$ the free Kisin module attached to $H$.
Then it is known that 
we have a functorial isomorphism 
$S\otimes_{\mfS} \vphi^{\ast}\mfM\simeq \mbb{D}(H)(S)$. 
For this, see \cite[Theorem 2.2.7 and Proposition A.6]{Kis} for $p>2$
and \cite[Proposition 4.2]{Kim} for $p=2$.
\end{remark}

Combining Theorems \ref{pdiv}, \ref{comparison:r=1}  and  Remark \ref{pdiv:Kisin},
the result below follows immediately.

\begin{theorem}
\label{Dieudonne}
Assume  $v_p(a_1)>1$.
Let $H$ be a $p$-divisible group over $\cO_K$
and $\mbb{D}(H)$ be the Dieudonn\'e crystal attached to $H$.
Let $\mfM$ be the Kisin module attached to $H$.
Then there exists a functorial isomorphism
$A_{\mrm{cris}}\otimes_{\mfS} \vphi^{\ast}\mfM
\simeq \mbb{D}(H)(A_{\mrm{cris}})$.
\end{theorem}

\section{Torsion representations and full faithfulness theorem}

In this section, we study torsion  Kisin modules 
and show a full faithfulness theorem 
for a restriction functor on a category of torsion crystalline representations. 

\subsection{Statements of full faithfulness theorems}
We state main results of this section.
Let $\mrm{Rep}^{r,\mrm{cris}}_{\mrm{tor}}(G)$
be the category of torsion crystalline representations of $G$
with Hodge-Tate weights in $[0,r]$.
Here, a torsion $\mbb{Z}_p$-representation $T$ of $G$
is {\it torsion crystalline with Hodge-Tate weights in $[0,r]$}
 if $T$ is a quotient of 
lattices in a crystalline $\mbb{Q}_p$-representation of $G$
with Hodge-Tate weights in $[0,r]$.
For example, it is well-known that the category
$\mrm{Rep}^{1,\mrm{cris}}_{\mrm{tor}}(G)$ 
coincides with the category of flat representations of $G$.
Here, a torsion $\mbb{Z}_p$-representation $T$ of $G$
is {\it flat} if it is of the form $H(\overline{K})$ with some
finite flat group scheme $H$ over the integer ring of $K$ 
killed by a power of $p$.

In the case where $r=1$,  we have 
\begin{theorem} 
\label{FFT:r=1}
Assume the condition $(P)$ and  $v_p(a_i)>1$  for any $1\le i\le p-1$.
Then 
the restriction functor $\mrm{Rep}^{1,\mrm{cris}}_{\mrm{tor}}(G)\to 
\mrm{Rep}_{\mrm{tor}}(G_{\underline{\pi}})$
is fully faithful.
\end{theorem}

\noindent
We recall that the condition (P)
is  that 
$\vphi^n(f(u)/u)$
is not a power of $E(u)$ for any $n\ge 0$.

For general $r$, we need some more technical assumptions.
Let $f(u)=\prod^n_{i=1}f_i(u)$ be an irreducible decomposition
of $f(u)$ in $W(\bar{k})[u]$
with the property that $f_1(u),\dots ,f_m(u)$ are of degree $\le e$
and $f_{m+1}(u),\dots ,f_n(u)$ are of degree $> e$.
We put $u_f=\prod^m_{i=1}f_i(u)$ and denote by $n_f$ the degree of $u_f$.
By definition, $u_f$ is divided by $u^{i_0}$
if $f(u)=\sum^p_{i=i_0}a_iu^i$ with $a_{i_0}\not=0$.
For example, we have $u_f=f(u)$ and $n_f=p$ 
if $f(u)$ is of the form $u^p+a_{p-1}u^{p-1}$.
For any integer $m\ge 0$,
we denote by $f^{(m)}(u)$ the $m$-th composite $(f\circ f\circ \cdots \circ f)(u)$
of $f(u)$.

\begin{theorem} 
\label{FFT}
Assume the following conditions.
\begin{itemize}
\item[${\rm (i)}$] $gu\in uW(R)$ for any $g\in G$.
\item[${\rm (ii)}$] $f^{(n)}(\pi)\not=0$ for any $n\ge 1$. 
\item[${\rm (iii)}$] $v_p(a_1)>r$ and $v_p(a_i)>1$  for any $1\le i\le p-1$.
\end{itemize}
Then 
the restriction functor $\mrm{Rep}^{r,\mrm{cris}}_{\mrm{tor}}(G)\to 
\mrm{Rep}_{\mrm{tor}}(G_{\underline{\pi}})$
is fully faithful if $e(r-1)<n_f(p-1)/p$.
\end{theorem}


\subsection{Some remarks}
We give some remarks about the conditions (i) and (ii) in Theorem \ref{FFT}.  

We suspect that Theorem \ref{FFT} is still valid 
if we remove the condition (i) 
(moreover, maybe (i) is always satisfied). 
Here are some examples.

\begin{itemize}
\item[--] If $f(u)=u^p$, it is clear that  the conditions (i), (ii) and (iii) above are satisfied.

\item[--]  If $p$ is odd, $K$ is a finite extension of $\mbb{Q}_p$ and 
$K_{\underline{\pi}}/K$ is Galois (in this case this is abelian (cf.\ Remark 7.16 of \cite{CD})),
then the condition (i)  is satisfied.
In fact, the $G$-action on $W(R)$ preserves 
$\mfS$ if $K_{\underline{\pi}}/K$ is Galois and hence 
we have $gu\in I_+W(R)\cap \mfS= u\mfS\subset uW(R)$.
\end{itemize}

We give two remarks for the condition (ii).
First, it is not difficult to check that 
the condition (ii) implies the condition (P).
Next, for a fixed $f(u)$, 
the condition (ii) is satisfied except only finitely many choice of 
uniformizers $\pi$ of $K$.
Moreover, we have the following.
(We recall that $i_0$ is the integer defined by $f(u)=\sum^p_{i=i_0}a_iu^i$ with $a_{i_0}\not=0$.)

\begin{proposition}
Put
\begin{align*}
n_0=
\left\{ 
\begin{array}{cr}
ev_p(a_1) 
 & \mrm{if}\ i_0=1, \cr
\mrm{max}\{n\in \mbb{Z}\mid i_0^n\le e(i_0-1)v_p(a_{i_0})+1\} 
 &  \mrm{if}\ i_0\not=1.
\end{array}
\right.
\end{align*}
Then the following are equivalent.
\begin{itemize}
\item[${\rm (ii)}$] $f^{(n)}(\pi)\not=0$ for any $n\ge 1$.
\item[${\rm (ii')}$] $f^{(n)}(\pi)\not=0$ for any $1\le n\le n_0$.
\end{itemize} 
\end{proposition}

\begin{proof}
Assume that there exists an integer $n\ge 1$
such that $f^{(i)}(\pi)\not=0$ for any $0\le i\le n-1$ and $f^{(n)}(\pi)=0$.
(In particular, we have $f(u)\not= u^p$.)
It suffices to show $n\le n_0$.
Put $c_i=v_p(f^{(i)}(\pi))$ for $0\le i\le n-1$. 
We have $c_0=1/e$ by definition.
Note that $f^{(n-1)}(\pi)$ is a root of 
$X^{p-i_0}+\sum^{p-1}_{i=i_0}a_iX^{i-i_0}$.
Seeing the Newton polygon of this polynomial, 
it is not difficult to check that the inequality $c_{n-1}\le v_p(a_{i_0})$ holds. 
On the other hand, we claim that the inequality
\begin{equation}
\label{ii}
c_j\ge \frac{1}{e}\sum^j_{k=0}i_0^k
\end{equation}
holds for any $0\le j\le n-1$.
We show this claim by induction on $j$.
The case $j=0$ is clear. Assume that \eqref{ii} 
holds for $j= m-1$ and consider the case where $j=m$.
It follows from the equation 
$f^{(m)}(\pi)=f^{(m-1)}(\pi)^p+\sum^{p-1}_{i=i_0}a_if^{(m-1)}(\pi)^i$
that we have
\begin{align*}
c_m 
& 
\ge \mrm{min}\{ pc_{m-1}, v_p(a_i) + ic_{m-1} \mid i=i_0,\dots , p-1  \}
\ge \mrm{min}\{ pc_{m-1}, 1 + i_0c_{m-1} \}\\
& 
\ge \mrm{min}\left\{ \frac{p}{e}\sum^{m-1}_{k=0}i_0^k, 
\frac{1}{e} + \frac{i_0}{e}\sum^{m-1}_{k=0}i_0^k \right\} 
=\mrm{min}\left\{ \frac{p}{e}\sum^{m-1}_{k=0}i_0^k, 
\frac{1}{e}\sum^m_{k=0}i_0^k \right\}. 
\end{align*}
Since we have 
$p\sum^{m-1}_{k=0}i_0^k-\sum^m_{k=0}i_0^k
\ge  (1+i_0)\sum^{m-1}_{k=0}i_0^k-\sum^m_{k=0}i_0^k
=  \sum^{m-1}_{k=0}i_0^k-1
\ge 0$,
we obtain $c_m\ge e^{-1}\sum^m_{k=0}i_0^k$ as desired.
Therefore, we obtain 
$$
\frac{1}{e}\sum^{n-1}_{k=0}i_0^k\le c_{n-1}\le v_p(a_{i_0}).
$$
The desired result immediately follows from this.

\end{proof}

\subsection{Maximal objects}

We recall that the contravariant functor
$T_{\mfS}\colon \Mod^r_{\mfS_{\infty}}\to \mrm{Rep}_{\mrm{tor}}(G_{\underline{\pi}})$
is exact and faithful (cf.\ Proposition \ref{KisinFF}).
However, this is not full in general.
In this section, following \cite{CL1}, we first define a notion of 
{\it maximal} Kisin modules\footnote{We can also study 
the theory of {\it minimal} Kisin modules by similar arguments to \cite{CL1}.
However, we do not consider it in this paper since we do not need it 
for our purpose.}\footnote{As well as \cite{CL1}, 
results in this section can be applied 
also for the case ``$r=\infty$'' with suitable (minor) modifications.}.
Almost the arguments given in \cite{CL1} carry over to the present situation.
In particular, we can check that a category of maximal Kisin modules is abelian 
and the functor $T_{\mfS}$ restricted to a category of 
maximal Kisin modules is fully faithful.
These play an important role in the proof of Theorems \ref{FFT:r=1} and \ref{FFT}.

Let $M$ be an \'etale $\vphi$-module over $\cO_{\E}$ which is killed by a power of $p$.
Let $F^r_{\mfS}(M)$ be the set of torsion Kisin modules $\mfM$ over $\mfS$ of height $r$
such that $\mfM\subset M$ and $\mfM[1/u]=M$.
The set $F^r_{\mfS}(M)$ is an partially ordered set by inclusion.
\begin{lemma}
\label{supinf}
If $\mfM, \mfM'\in F^r_{\mfS}(M)$, 
then we have $\mfM+\mfM',\mfM\cap\mfM'\in F^r_{\mfS}(M)$.
\end{lemma}
\begin{proof}
See the proof of Proposition 3.2.3 of \cite{CL1}.
\end{proof}

\begin{lemma}
\label{length}
Let $\mfM$ be a torsion Kisin module $\mfM$ over $\mfS$ of height $r$ and
put $M=\mfM[1/u]$.
If $\mfM'\in  F^r_{\mfS}(M)$ and $\mfM \subset \mfM'$,
then we have 
$$
\mrm{length}_{\mfS} (\mfM'/\mfM)\le \left[\frac{er}{p-1}\right]\cdot \mrm{length}_{\cO_{\E}}M.
$$
Here, $[x]$ denotes the integer part of  $x$.
\end{lemma}
\begin{proof}
See the proof of Lemma 3.2.4 of \cite{CL1}.
\end{proof}

By the above lemmas, we immediately obtain
\begin{corollary}
\label{Great}
Let $M\in \Mod_{\cO_{\E,\infty}}$ and suppose that 
 $F^r_{\mfS}(M)\not= \emptyset$.

\noindent
$(1)$ The set $F^r_{\mfS}(M)$ has a greatest element and a smallest element.

\noindent
$(2)$ If $er<p-1$, then $F^r_{\mfS}(M)$ contains only one element.
\end{corollary}

\begin{definition}
Let $\mfM$ be a torsion Kisin module over $\mfS$ of height $r$.
We denote by $\mrm{Max}^r(\mfM)$ the greatest element 
of $F^r_{\mfS}(\mfM[1/u])$.
We say that $\mfM$ is {\it maximal (of height $r$)} if $\mfM=\mrm{Max}^r(\mfM)$.
\end{definition}

We denote by $\mrm{Max}^r_{\mfS_{\infty}}$ the full subcategory of 
$\Mod^r_{\mfS_{\infty}}$ consisting of maximal Kisin modules.
By Corollary \ref{Great},
we have $\Mod^r_{\mfS_{\infty}}=\mrm{Max}^r_{\mfS_{\infty}}$
if $er<p-1$.

We can check that all the properties given in Section 3.3 
in \cite{CL1} holds  also for the present situation by the same arguments given in 
{\it loc.\ cit}. 
Here we describe only a part of properties on maximal Kisin modules
that we need later.

\begin{theorem}
\label{MAX}
$(1)$ The implication $\mfM \mapsto \mrm{Max}^r(\mfM)$ defines a covariant functor
$\mrm{Max}^r\colon \Mod^r_{\mfS_{\infty}}\to \Mod^r_{\mfS_{\infty}}$.
Furthermore, this is left exact and $\mrm{Max}^r\circ \mrm{Max}^r=\mrm{Max}^r$.

\noindent
$(2)$ The category $\mrm{Max}^r_{\mfS_{\infty}}$ is abelian.
Moreover, for any morphism $f\colon \mfM\to \mfM'$ in
$\mrm{Max}^r_{\mfS_{\infty}}$, we have the following.
\begin{itemize}
\item[${\rm (i)}$] The  kernel $\mrm{ker}(f)$ of $f$ in the usual sense is an object of  $\mrm{Max}^r_{\mfS_{\infty}}$.
Furthermore, it is the  kernel of $f$ in the abelian category
$\mrm{Max}^r_{\mfS_{\infty}}$.
\item[${\rm (ii)}$] The cokernel $\mrm{coker}(f)$ in the usual sense is of height $r$
and $\mrm{coker}(f)/(u\mathchar`-\mrm{tors})$ is a Kisin module of height $r$.
Moreover, 
$\mrm{Max}^r\left(\mrm{coker}(f)/(u\mathchar`-\mrm{tors}) \right)$
is the cokernel  of  $f$ in the abelian category
$\mrm{Max}^r_{\mfS_{\infty}}$.
If $f$ is injective, then $\mrm{coker}(f)$ is $u$-torsion free.
\item[${\rm (iii)}$] The image $\mrm{im}(f)$ $($resp.\ the coimage $\mrm{coim}(f)$$)$ 
of $f$ in the usual sense is a Kisin module of height $r$.
Moreover, $\mrm{Max}^r(\mrm{im}(f))$ $($resp.\ $\mrm{Max}^r(\mrm{coim}(f))$$)$ is
the image $($resp.\ the coimage$)$ of  $f$ in the abelian category
$\mrm{Max}^r_{\mfS_{\infty}}$.
\end{itemize}

\noindent
$(3)$ Let $0\to \mfM'\overset{\alpha}{\to} \mfM\overset{\beta}{\to} \mfM''\to 0$ 
be a sequence in $\mrm{Max}^r_{\mfS_{\infty}}$ such that $\beta\circ \alpha=0$.
Then this sequence is exact in the abelian category $\mrm{Max}^r_{\mfS_{\infty}}$
if and only if 
$0\to \mfM'[1/u]\overset{\alpha[1/u]}{\to} \mfM[1/u]
\overset{\beta[1/u]}{\to} \mfM''[1/u]\to 0$ is exact as $\cO_{\E}$-modules.

\noindent
$(4)$ The functor  $\mrm{Max}^r_{\mfS_{\infty}}\to \Mod_{\cO_{\E,\infty}}$  given by 
$\mfM\mapsto \cO_{\E}\otimes_{\mfS}\mfM$ is exact and fully faithful.

\noindent
$(5)$ The functor $T_{\mfS}\colon \mrm{Max}^r_{\mfS_{\infty}}
\to \mrm{Rep}_{\mrm{tor}}(G_{\underline{\pi}})$
is exact and fully faithful.
\end{theorem}

\begin{proof}
(1) :  See the proof of Propositions 3.3.2 to 3.3.4 of \cite{CL1}.

\noindent
(2) : See the proof of Theorem 3.3.8 \cite{CL1}.

\noindent
(3) and (4) :  See the proof of Lemma 3.3.9 of \cite{CL1}.

\noindent
(5) : This follows from (4) immediately.

\end{proof}

Let us consider simple objects in the abelian category $\mrm{Max}^r_{\mfS_{\infty}}$.
Let $\mcal{S}$ be the set of sequences $\mfn=(n_i)_{i\in \mbb{Z}/d\mbb{Z}}$
of integers $0\le n_i\le er$ with smallest period  $d$ for some integer $d>0$.

\begin{definition}
\label{simple1}
Let $\mfn=(n_i)_{i\in \mbb{Z}/d\mbb{Z}}\in \mcal{S}$
be a sequence
with smallest period $d$.
We define a torsion Kisin module $\mfM(\mfn)$ of height $r$, killed by $p$, as follows:
\begin{itemize}
\item as a $\ku$-module, $\mfM(\mfn)=\bigoplus_{i\in \mbb{Z}/d\mbb{Z}} \ku e_i$;
\item for all $i\in \mbb{Z}/d\mbb{Z}$, $\vphi(e_i)=u^{n_i}e_{i+1}$.
\end{itemize}
\end{definition}

We denote by $\mcal{S}^r_{\mrm{max}}$ the set of sequences $\mfn=(n_i)_{i\in \mbb{Z}/d\mbb{Z}}$
of integers $0\le n_i\le \mrm{min}\{er, p-1\}$ with smallest period $d$ for some integer $d$
except the constant sequence with value $p-1$ (if necessary).  

\begin{proposition}
\label{simple2}
Assume that $k$ is algebraically closed.
Then all simple objects in the abelian category $\mrm{Max}^r_{\mfS_{\infty}}$
 are of the form $\mfM(\mfn)$
with  some $\mfn\in \mcal{S}^r_{\mrm{max}}$.
\end{proposition}

\begin{proof}
This is a part of Propositions 3.6.8 and 3.6.12 in \cite{CL1}.
\end{proof}

\subsection{$(\vphi,G)$-modules}

\begin{definition}
A {\it free $($resp.\ torsion$)$ $(\vphi,G)$-module $($of height $r$$)$} is a triple 
$\hat{\mfM}=(\mfM,\vphi,G)$ where 
\begin{itemize}
\item[(1)] $(\mfM,\vphi)$ is a free (resp.\ torsion) 
Kisin module $\mfM$ of height $r$,
\item[(2)] $G$ is an $W(R)$-semi-linear continuous
$G$-action  on $W(R) \otimes_{\vphi, \mfS}\mfM$,
\item[(3)] the $G$-action on $W(R) \otimes_{\vphi, \mfS}\mfM$
commutes  with $\vphi_{W(R)}\otimes \vphi_{\mfM}$, and 
\item[(4)] $\vphi^{\ast}\mfM \subset 
(W(R) \otimes_{\vphi, \mfS}\mfM)^{G_{\underline{\pi}}}$.
\end{itemize}
We denote by $\Mod^{r,G}_{\mfS}$ (resp.\ $\Mod^{r,G}_{\mfS_{\infty}}$) 
the category of free (resp.\ torsion) $(\vphi,G)$-modules of height $r$.
\end{definition}

We define  a $\mbb{Z}_p$-representation $\hat{T}(\hat{\mfM})$ 
of $G$ for any $(\vphi,G)$-module $\hat{\mfM}$ by
\begin{align*}
\hat{T}(\hat{\mfM}):=
\left\{ 
\begin{array}{cr}
\mrm{Hom}_{W(R),\vphi}(W(R)\otimes_{\vphi,\mfS}\mfM,W(R)). 
 & \mrm{if}\ \hat{\mfM}\in \Mod^{r,G}_{\mfS}, \cr
\mrm{Hom}_{W(R),\vphi}(W(R)\otimes_{\vphi,\mfS}\mfM,W(R)\otimes_{\mbb{Z}_p} 
\mbb{Q}_p/\mbb{Z}_p). 
 &  \mrm{if}\ \hat{\mfM}\in \Mod^{r,G}_{\mfS_{\infty}}.
\end{array}
\right.
\end{align*}
Here, the  $G$-action on  $\hat{T}(\hat{\mfM})$
is given by $(g.f)(x):=g(f(g^{-1}(x)))$ for 
$f\in \hat{T}(\hat{\mfM})$,
$g\in G$ and $x\in W(R)\otimes_{\vphi,\mfS}\mfM$.
Note that 
we have a natural isomorphism of 
$\mbb{Z}_p[G_{\underline{\pi}}]$-modules
$$
\theta\colon T_{\mfS}(\mfM)\overset{\sim}{\longrightarrow}
\hat{T}(\hat{\mfM})
$$
given by $\theta(f)(a\otimes x):=a\vphi(f(x))$
for $f\in T_{\mfS}(\hat{\mfM})$, $a\in W(R)$ and $x\in \mfM$.
In particular, if $\hat{\mfM}$ is free, then
$\hat{T}(\hat{\mfM})$ is a free $\mbb{Z}_p$-module of rank $d$,
where $d:=\mrm{rank}_{\mfS}\mfM$.
Hence we obtain a contravariant functor 
$$
\hat{T}\colon \Mod^{r,G}_{\mfS}\to \mrm{Rep}_{\mbb{Z}_p}(G)\quad \mrm{and}\quad
\hat{T}\colon \Mod^{r,G}_{\mfS_{\infty}}\to \mrm{Rep}^{\mrm{tor}}_{\mbb{Z}_p}(G).
$$
For a $(\vphi,\hat{G})$-module $\hat{\mfM}$, by extending the $G$-action on
$\whR\otimes_{\vphi,\mfS} \mfM$ (naturally obtained by the $\hat{G}$-action on this module)
to $W(R)\otimes_{\vphi,\mfS} \mfM$ by $W(R)$-semi-linearity,
we obtain a $(\vphi,G)$-module; we abuse notation by writing $\hat{\mfM}$ for it.

\begin{definition}
Let $\alpha\in W(R)\smallsetminus pW(R)$.
We define a full subcategory 
$\Mod^{r,G}_{\mfS}(\alpha)$ (resp.\ $\Mod^{r,G}_{\mfS_{\infty}}(\alpha)$)
of 
$\Mod^{r,G}_{\mfS}$ (resp.\ $\Mod^{r,G}_{\mfS_{\infty}}$) 
consisting of objects $\hat{\mfM}$ with the condition that
$$
g(1\otimes x)-(1\otimes x) \in \alpha I^{[1]}W(R)\otimes_{\vphi, \mfS}\mfM 
$$ 
for any $g\in G$ and $x\in \mfM$.
We put $\bar{\alpha}=\alpha \ \mrm{mod}\ pW(R)\in R$.
\end{definition}

\begin{theorem}
\label{FFT2}
Let $r,r'\ge 0$, 
$\hat{\mfM}\in \Mod^{r,G}_{\mfS_{\infty}}(\alpha)$
and 
$\hat{\mfN}\in \Mod^{r',G}_{\mfS_{\infty}}(\alpha)$.
Then we have
$\mrm{Hom}(\hat{\mfM},\hat{\mfN})=
\mrm{Hom}(\mfM,\mfN)$
if $v_R(\bar{\alpha})>p(r-1)/(p-1)$.

In particular, 
the forgetful functor 
 $\Mod^{r,G}_{\mfS_{\infty}}(\alpha)\to \Mod^r_{\mfS_{\infty}}$
is fully faithful
if $v_R(\bar{\alpha})>p(r-1)/(p-1)$.
\end{theorem}

\begin{proof}
The proof is the same as that of Proposition 4.2 in \cite{Oz2}.
Here we only explain why we need 
the condition $v_R(\bar{\alpha})>p(r-1)/(p-1)$.
Assume $p\mfN=0$ for simplicity. Let $g\in G$ and $f\colon \mfM \to \mfN$ 
be a morphism of Kisin modules. 
We also denote by 
$f\colon W(R)\otimes_{\vphi,\mfS}\mfM\to  W(R)\otimes_{\vphi,\mfS}\mfN$
the $W(R)$-linear extension of $f$.
Then it follows from the argument of the proof of  Proposition 4.2 of {\it loc.\ cit.} 
that  we have 
$f\circ g(x)-g\circ f(x)
\in \mfm_R^{\ge c(s)}\otimes_{\vphi,\mfS} \mfN$
for any $s\ge 0$ and $x\in W(R)\otimes_{\vphi,\mfS}\mfM$.
Here, $c(s)$ is defined by $c(0)=v_R(\bar{\alpha})+p/(p-1)$
and $c(s+1)=pc(s) -pr$, that is, $c(s)=(v_R(\bar{\alpha})-p(r-1)/(p-1))p^s+pr/(p-1)$.
By the assumption $v_R(\bar{\alpha})>p(r-1)/(p-1)$,
we have  $\lim_{s\to \infty} c(s)=\infty$ and hence $f$ commutes with $g$.
\end{proof}

\subsection{A $G$-action on $\mfM(\mfn)$}
\label{simple:act}

In this section, we equip a $(\vphi,G)$-module structure on 
$\mfM(\mfn)$. 
In the classical setting $f(u)=u^p$,
this has been already studied in Section 4.3 of \cite{Oz2} by using the fact that
the $G$-action on $u$ is explicitly calculated.
In the present setting, the $G$-action on $u$ is not  so easy to understand, and so
we need more delicate arguments. 
\begin{theorem}
\label{simple:str}
Assume that $v_p(a_i)>1$ for any $1\le i\le p-1$.
Let $\mfn=(n_i)_{i\in \mbb{Z}/d\mbb{Z}}\in \mcal{S}$
be a sequence
with smallest period $d$.
Let $\mfM(\mfn)$ be the Kisin module of height $r$ defined in Definition \ref{simple1}. 
Then there exists a $W(R)$-semi-linear $G$-action on 
$W(R)\otimes_{\vphi,\mfS} \mfM(\mfn)$ which satisfies the following properties:

\noindent
$(1)$ The $G$-action on $W(R)\otimes_{\vphi,\mfS} \mfM(\mfn)$
commutes with $\vphi_{W(R)}\otimes \vphi_{\mfM(\mfn)}$.

\noindent
$(2)$ $G_{\underline{\pi}}$ acts on $\mfS\otimes_{\vphi,\mfS} \mfM(\mfn)$ trivial.

\noindent
$(3)$ For any $g\in G$ and $x\in \mfM$, we have 
$g(1\otimes x)-(1\otimes x)\in  \mfm^{\ge p^2/(p-1)}_{R}\otimes_{\vphi,\mfS} \mfM(\mfn)$. 

\noindent
Moreover, such a $G$-action is uniquely determined if $0\le n_i\le \mrm{min}\{ er,p-1 \}$ for any $i$.
\end{theorem}

\begin{remark}
For the uniqueness assertion above, we do not need  (2).
\end{remark}

\begin{proof}[Proof of Theorem \ref{simple:str}]
Take any $(p^d-1)$-st root $\pi_{(0)}$ of $\pi=\pi_0$.
We define $\pi_{(n)}$ inductively by the formula
 $\pi_{(n)}=\pi^{p^{d-1}}_{(n-1)}\pi_n^{-1}$ for $n\ge 1$.
We see $v_p(\pi_{(n)})=1/(ep^n(p^d-1))$ and thus 
we have $\pi_{(n)}\in \cO_{\overline{K}}$.
Now we claim the following.
\begin{equation}
\label{eq}
\quad \pi^p_{(n)}\equiv \pi_{(n-1)}\ \mrm{mod}\ p\pi_{(n-1)}\cO_{\overline{K}}
\quad \mrm{and} \quad 
\pi^{p^d-1}_{(n)}\equiv \pi_n\ \mrm{mod}\ p\pi_n\cO_{\overline{K}}.
\end{equation}
We proceed a proof of this claim by induction on $n$.

Consider the case $n=1$. We have $\pi^p_{(1)}=\pi^{p^d}_{(0)}\pi_1^{-p}
=\pi_{(0)}\pi^{p^d-1}_{(0)}\pi_1^{-p}=\pi_{(0)}\cdot \pi\pi_1^{-p}$
and $\pi=\pi^p_1+\sum^{p-1}_{i=1}a_i\pi_1^i$. 
Hence we obtain  
$\pi^p_{(1)}=\pi_{(0)}+\pi_{(0)}\sum^{p-1}_{i=1}a_i\pi_1^{-(p-i)}$.
By the assumption $v_p(a_i)>1$ for any $i$,
we obtain $\pi^p_{(1)}\equiv \pi_{(0)}\ \mrm{mod}\ p\pi_{(0)}\cO_{\overline{K}}$.
On the other hand,
we have 
$\pi^{p^d-1}_{(1)}=\pi^{p^{d-1}(p^d-1)}_{(0)}\pi_1^{-(p^d-1)}
=\pi_1(\pi\pi^{-p}_1)^{p^{d-1}}
=\pi_1(1+\sum^{p-1}_{i=1}a_i\pi_1^{-(p-i)})^{p^{d-1}}
$.
By the assumption $v_p(a_i)>1$ for any $i$,
we have 
$(1+\sum^{p-1}_{i=1}a_i\pi_1^{-(p-i)})^{p^{d-1}}\in 1+p\cO_{\overline{K}}$.
Hence we have $\pi^{p^d-1}_{(1)}\equiv \pi_1\ \mrm{mod}\ p\pi_1\cO_{\overline{K}}$
as desired.

Next we assume that \eqref{eq} holds for $n=m-1$ and consider the case $n=m$.
By induction hypothesis, we have $\pi^{p^d-1}_{(m-1)}=\pi_{m-1}+p\pi_{m-1}x$
for some $x\in \cO_{\overline{K}}$. Thus 
we have 
$\pi^p_{(m)}=\pi^{p^d}_{(m-1)}\pi_m^{-p}
=\pi_{(m-1)}\pi^{p^d-1}_{(m-1)}\pi^{-p}_m
=\pi_{(m-1)}(\pi_{m-1}+p\pi_{m-1}x)\pi_m^{-p}
=\pi_{(m-1)}(1+\sum^{p-1}_{i=1}a_i\pi^{-(p-i)}_{m}
+p\pi_{m-1}\pi^{-p}_m x)$.
By the assumption $v_p(a_i)>1$ for any $i$, 
we obtain $\pi^p_{(m)}\equiv \pi_{(m-1)}\ \mrm{mod}\ p\pi_{(m-1)}\cO_{\overline{K}}$.
On the other hand,
we have $\pi^{p^d-1}_{(m)}
=(\pi^{p^{d-1}}_{(m-1)}\pi_m^{-1})^{p^d-1}
=(\pi_{m-1}+p\pi_{m-1}x)^{p^{d-1}}\pi^{-p^d}_m\pi_m
=\pi_m((\pi_{m-1}+p\pi_{m-1}x)\pi^{-p}_m)^{p^{d-1}}
=\pi_m(1+\sum^{p-1}_{i=1}a_i\pi^{-(p-i)}_{m}
+p\pi_{m-1}\pi^{-p}_m x)^{p^{d-1}}$.
By the assumption $v_p(a_i)>1$ for any $i$, we have 
$(1+\sum^{p-1}_{i=1}a_i\pi^{-(p-i)}_{m}
+p\pi_{m-1}\pi^{-p}_m x)^{p^{d-1}}
\in 1+p\cO_{\overline{K}}$.
Therefore, we obtain 
$\pi^{p^d-1}_{(m)}\equiv \pi_m\ \mrm{mod}\ p\pi_m\cO_{\overline{K}}$. 
This finishes a proof of \eqref{eq}.

By \eqref{eq}, we can define an element $\underline{\pi}_d$ of $R$
by $\underline{\pi}_d:=(\pi_{(n)}\ \mrm{mod}\ p\cO_{\overline{K}})_{n\ge 0}$.
By definition we have $\underline{\pi}_d^{p^d-1}=\underline{\pi}$.
On the other hand, for any $g\in G$, there exists a unique 
$a_g\in \mbb{F}^{\times}_{p^d}$ such that 
$g\pi_{(0)}\pi_{(0)}^{-1}=[a_g]$. 
Here, $[\cdot]$ stands for the Teichm\"uller lift.
We note that 
we have a cocycle condition
$a_{gh}=a_g \cdot ga_h$ for any $g,h\in G$.
Put $x_g=a_g^{-1}g\underline{\pi}_d\underline{\pi}_d^{-1}\in R^{\times}$.
By the cocycle condition above, we can define an $R$-semi-linear 
$G$-action on $R\otimes_{\vphi,\mfS} \mfM(\mfn)=
\oplus_{i\in \mbb{Z}/d\mbb{Z}}R(1\otimes e_i)$ by 
$$
g(1\otimes e_i)=x^{m_i}_g(1\otimes e_i)
$$
for any $g\in G$ and $i\in \mbb{Z}/d\mbb{Z}$.
Here, $m_i=\sum^d_{j=1} p^j n_{i-j}$.
In the rest of this proof, we show that this $G$-action 
satisfies the assertions (1), (2) and (3) in the statement of this lemma. 
The assertion (1) can be checked by a direct computation  without difficulty.
We check (2) and (3) below.

We show (2). Let $g\in G_{\underline{\pi}}$. It suffices to show that 
$g\underline{\pi}_d\underline{\pi}_d^{-1}$ coincides with $a_g$.
The case $(p,d)=(2,1)$ is clear.
Thus we may assume $(p,d) \not=(2,1)$. 
Put $b_g:=g\underline{\pi}_d\underline{\pi}_d^{-1}$, 
which is an element of $\mbb{F}^{\times}_{p^d}$.
Seeing the $0$-th components of both sides of 
$g\underline{\pi}_d=b_g\underline{\pi}_d$,
we have $g\pi_{(0)}\equiv [b_g]\pi_{(0)}\ \mrm{mod}\ p\cO_{\overline{K}}$.
Thus we have $[a_g]\pi_{(0)} \equiv [b_g]\pi_{(0)}\ \mrm{mod}\ p\cO_{\overline{K}}$,
and this induces $[a_g]-[b_g]\in p\pi^{-1}_{(0)}\cO_{\overline{K}}$.
By the assumption $(p,d) \not=(2,1)$,
we have $v_p(p\pi^{-1}_{(0)})=1-1/(e(p^d-1))>0$,
and hence we obtain $[a_g]-[b_g]\in pW(\overline{k})$.
Therefore, we have $a_g=b_g$. This shows (2).

We show (3). We may assume $g\notin G_{\underline{\pi}}$.
At first we show 
\begin{equation}
\label{eq2}
g(1\otimes e_i)-(1\otimes e_i)\in  \mfm^{\ge \frac{p^2}{p-1}}_{R}\otimes_{\vphi,\mfS} \mfM(\mfn)
\end{equation}
for any $g\in G$ and $i\in \mbb{Z}/d\mbb{Z}$. 
Since $m_i$ is divided by $p$, it suffices to show 
$x_g-1\in \mfm_R^{\ge p/(p-1)}\otimes_{\vphi,\mfS} \mfM(\mfn)$.
Note that the $n$-th component of 
$a_g^{-1}g\underline{\pi}_d-\underline{\pi}_d$
is $[a^{-p^{-n}}_g]g\pi_{(n)}-\pi_{(n)}\ \mrm{mod}\ p\cO_{\overline{K}}$.
Hence we have 
\begin{align*}
v_R(x_g-1) & 
=v_R(a_g^{-1}g\underline{\pi}_d-\underline{\pi}_d)-v_R(\underline{\pi}_d)\\
& =\lim_{n\to \infty} 
p^nv_p([a^{-p^{-n}}_g]g\pi_{(n)}-\pi_{(n)})
- \frac{1}{e(p^d-1)}\\
& =\lim_{n\to \infty} 
p^n(v_p([a^{-p^{-n}}_g]g\pi_{(n)}-\pi_{(n)})
-v_p(\pi_{(n)}))\\
& =\lim_{n\to \infty} 
p^nv_p\left([a^{-p^{-n}}_g]\frac{g\pi_{(n)}}{\pi_{(n)}}-1\right).
\end{align*}
Hence it is enough to show that 
$v_p\left([a^{-p^{-n}}_g]g\pi_{(n)}\pi_{(n)}^{-1}-1\right)\ge p/(p^n(p-1))$
for $n$ large enough. More precisely, we claim the following: Let $N\ge 1$ be the integer 
such that $g\pi_{N-1}=\pi_{N-1}$ and $g\pi_N\not= \pi_N$.
(Such $N$ exists  by the assumption  $g\notin G_{\underline{\pi}}$.)
Then we have
\begin{equation}
\label{eq3}
v_p\left([a^{-p^{-n}}_g]\frac{g\pi_{(n)}}{\pi_{(n)}}-1\right)\ge \frac{p^N}{p^n(p-1)}
\end{equation}
for $n$ large enough.
We show this inequality. 
Put $c_n=[a^{-p^{-n}}_g]g\pi_{(n)}\pi_{(n)}^{-1}$ for $n\ge 0$.
Since $a_g^{p^d}=a_g$, we have 
\begin{align*}
c_n-1&=[a^{-p^{-n}}_g]^{p^d}\frac{g\pi^{p^{d-1}}_{(n-1)}g\pi_n^{-1}}{\pi^{p^{d-1}}_{(n-1)}\pi_n^{-1}}-1
= \left([a^{-p^{-(n-1)}}_g]\frac{g\pi_{(n-1)}}{\pi_{(n-1)}}\right)^{p^{d-1}}
\left( \frac{g\pi_n}{\pi_n}\right)^{-1}-1 \\
& = c^{p^{d-1}}_{n-1}\left(  \frac{g\pi_n}{\pi_n} \right)^{-1}-1
=  (c^{p^{d-1}}_{n-1}-1)\left(  \frac{g\pi_n}{\pi_n} \right)^{-1}
+\left(  \frac{g\pi_n}{\pi_n} \right)^{-1}-1.
\end{align*}
In particular, we have $v_p(c_n-1)\ge \mrm{min}\{v_p(c_{n-1}-1), v_p(g\pi_n\pi_n^{-1}-1)\}$.
Repeating this argument, we obtain
$v_p(c_n-1)\ge \mrm{min}\{v_p(c_{0}-1),v_p(g\pi_1\pi_1^{-1}-1),\dots ,v_p(g\pi_n\pi_n^{-1}-1)\}$.
Since $c_{0}-1=0$, we have 
$v_p(c_n-1)\ge \mrm{min}\{v_p(g\pi_1\pi_1^{-1}-1),\dots ,v_p(g\pi_n\pi_n^{-1}-1) \}$.
On the other hand, 
we know $v_R(g\pi_n\pi_n^{-1}-1)=
p^N/(p^n(p-1))$ for any $n\ge N$ by the proof of Proposition \ref{val}.
Hence, to show  \eqref{eq3}, it suffices to show  
$v_p(g\pi_n\pi_n^{-1}-1)>0$ for any $n\ge 1$. More precisely, we show 
\begin{equation}
\label{eq4}
v_p\left(\frac{g\pi_n}{\pi_n}-1\right)>\frac{p}{p^n(p-1)}
\end{equation}
for any $n\ge 1$.
We note that 
$x_n:=g\pi_n\pi_n^{-1}-1$ is a root of $\sum^p_{i=1}a_i\pi^{-(p-i)}_n(X+1)^i-g\pi_{n-1}\pi_n^{-p}$.
Put $b_j=\sum^p_{i=j} \binom{i}{j} a_i\pi^{-(p-i)}_n \in p\cO_{\overline{K}}$ 
for any $1\le j\le p-1$.
Then we see the equality 
$\sum^p_{i=1} a_i \pi^{-(p-i)}_n(X+1)^i=X^p+\sum^{p-1}_{j=1}b_jX^j+\pi_{n-1}\pi^{-p}_n$.
Hence $x_n$ for $n\ge 2$ (resp.\ $n=1$) is a root of 
$X^p+\sum^{p-1}_{j=1}b_jX^j+(g\pi_{n-1}-\pi_{n-1})\pi^{-p}_n$
(resp.\ $X^{p-1}+\sum^{p-1}_{j=1}b_jX^{j-1}$).
Now  \eqref{eq4})follows by induction on $n$ and arguments of Newton polygons.  
Consequently we finish a proof of  \eqref{eq2}.

To finish a proof of (3), we need to show
\begin{equation}
\label{eq5}
g(1\otimes x)-(1\otimes x)\in  \mfm^{\ge \frac{p^2}{p-1}}_{R}\otimes_{\vphi,\mfS} \mfM(\mfn)
\end{equation}
for any  $x\in \mfM(\mfn)$. 
Writing $x=\sum^d_{i=1}a_ie_i$ with some $a_i\in \ku$,
we have 
$g(1\otimes x)-(1\otimes x)=
\sum^d_{i=1}(g(1\otimes a_ie_i)-(1\otimes a_ie_i))
=\sum^d_{i=1}((ga_i-a_i)^pg(1\otimes e_i))+a^p_i(g(1\otimes e_i)-(1\otimes e_i)))$.
By \eqref{eq2},
it suffices to show $ga_i-a_i\in \mfm_R^{\ge p/(p-1)}$
but this immediately follows from Proposition \ref{val}. 
Consequently, we obtained a proof of (3).

Finally, we show that an $R$-semi-linear $G$-action on $R\otimes_{\vphi,\mfS} \mfM(\mfn)$
satisfying (1) and (3) is uniquely determined when $0\le n_i\le \mrm{min}\{er, p-1\}$ for any $i$.
Assume that two  $G$-actions 
$\rho_1,\rho_2\colon G\to \mrm{End}_{R}(R\otimes_{\vphi,\mfS} \mfM(\mfn))$
on $R\otimes_{\vphi,\mfS} \mfM(\mfn)$
satisfy (1) and (3), 
and put $g_{\ast}(x)=\rho_1(g)(x)$ and $g_{\sharp}(x)=\rho_2(g)(x)$ 
for any $g\in G$ and $x\in R\otimes_{\vphi,\mfS} \mfM(\mfn)$.
By (3),  we have 
$g_{\ast}(1\otimes e_i)-g_{\sharp}(1\otimes e_i)\in 
\mfm^{\ge c(0)}_{R}\otimes_{\vphi,\mfS} \mfM(\mfn)$
where $c(0)=p^2/(p-1)$.
Thus, by (1), we obtain
$$
gu^{pn_i}(g_{\ast}(1\otimes e_{i+1})-g_{\sharp}(1\otimes e_{i+1}))
= \vphi(g_{\ast}(1\otimes e_i)-g_{\sharp}(1\otimes e_i))
\in \mfm^{\ge pc(0)}_{R}\otimes_{\vphi,\mfS} \mfM(\mfn).
$$
Furthermore, we have $pc(0)-pn_i/e\ge pc(0)-p(p-1)$
by the assumption $0\le n_i\le \mrm{min}\{er, p-1\}$.
Hence we obtain 
$g_{\ast}(1\otimes e_{i+1})
-g_{\sharp}(1\otimes e_{i+1})
\in \mfm^{\ge c(1)}_{R}\otimes_{\vphi,\mfS} \mfM(\mfn)$
where $c(1)=pc(0)-p(p-1)$.
Repeating this argument,
we obtain 
$g_{\ast}(1\otimes e_{i+s})-g_{\sharp}(1\otimes e_{i+s})
\in \mfm^{\ge c(s)}_{R}\otimes_{\vphi,\mfS} \mfM(\mfn)$
for any $s\ge 0$
where $c(s)=pc(s-1)-p(p-1)=p^{s+1}/(p-1)+p$.
Since $\lim_{s\to \infty} c(s)=\infty$,
we obtain $g_{\ast}(1\otimes e_i)=g_{\sharp}(1\otimes e_i)$ for any $i$ as desired.
\end{proof}

\subsection{Proofs of Theorems \ref{FFT:r=1} and \ref{FFT} }

In this section we prove  Theorems \ref{FFT:r=1} and \ref{FFT}.
We put $\bar{a}=a\ \mrm{mod}\ pW(R)$ for any $a\in W(R)$.
It is known (cf.\ Example 3.3.2 of \cite{CL}) 
that there exists $\mft'\in W(R)\smallsetminus pW(R)$ such that
$\vphi(\mft')=E(u)\mft'$. By Lemma 2.3.1 of {\it loc.\ cit.},
$\vphi(\mft')$ is a generator of $I^{[1]}W(R)$.

\begin{remark}
Under the condition $v_p(a_1)>1$, 
we defined $\mft\in W(R)\smallsetminus pW(R)$ in Section \ref{EKmod}
such that $\vphi(\mft)=\mu_0E(u)\mft$ with some $\mu_0\in \mfS^{\times}$.
Then we have $\mft/\mft'\in W(R)^{\times}$
since both $\vphi(\mft)$ and  $\vphi(\mft')$ are generators of 
a principal ideal $I^{[1]}W(R)$.
\end{remark}

We start with two estimations of the ideal $\vphi(gu-u)B^+_{\mrm{cris}}\cap W(R)$ of $W(R)$
for $g\in G$ to study its reduction modulo $p$.
The first proposition gives a ``weak'' estimation, however, it does not need any assumption.
The second one gives a ``strong'' estimation although 
we need some technical assumptions.

\begin{proposition}
\label{lem0}
Let $j_0$ be the minimum integer $1\le j\le p$ such that 
$v_p(ja_j)=1$. 
Put $h=0$ $($resp.\ $h=1)$ if $e<j_0-1$ $($resp.\ $e\ge j_0-1$$)$.

\noindent
$(1)$ Let $g\in G\smallsetminus G_{\underline{\pi}}$ and $N\ge 1$ the integer 
such that $g\pi_{N-1}=\pi_{N-1}$ and $g\pi_N\not= \pi_N$.
Then 
\begin{itemize}
\item[${\rm (i)}$] $gu-u=\vphi^N(\mft')v_g$ for some $v_g\in W(R)$.
\item[${\rm (ii)}$] $\vphi(v_g)= v_gw_g$ for some $w_g\in W(R)$.
\item[${\rm (iii)}$] $\vphi(gu-u)B^+_{\mrm{cris}}\cap W(R) 
\subset v_gw^h_gI^{[1]}W(R)$.
\end{itemize}
\noindent
$(2)$ 
The image of 
$\vphi(gu-u)B^+_{\mrm{cris}}\cap W(R)$
under the projection $W(R)\to R$
is contained in 
$\mfm_R^{\ge c}$   for any $g\in G$. Here,
$$
c=\frac{p}{p-1}+\frac{j_0-1}{e(p-1)}p^h.
$$
\end{proposition}

\begin{proposition}
\label{lem1}
Assume the following conditions.
\begin{itemize} 
\item[${\rm (i)}$] $gu\in uW(R)$ for any $g\in G$.
\item[${\rm (ii)}$] $f^{(n)}(\pi)\not=0$ for any $n\ge 1$.
\end{itemize}
Then we have the following.

\noindent
$(1)$ $gu-u\in uI^{[1]}W(R)$ for any $g\in G$.

\noindent
$(2)$ $\vphi(gu-u)B^+_{\mrm{cris}}\cap W(R)\subset u_fI^{[1]}W(R)$
for any $g\in G$.

\noindent
$(3)$
The image of 
$\vphi(gu-u)B^+_{\mrm{cris}}\cap W(R)$
under the projection  $W(R)\to R$
is contained in 
$\mfm_R^{\ge c}$ for any $g\in G$. Here, 
$$
c=\frac{p}{p-1}+\frac{n_f}{e}.
$$
\end{proposition}

For proofs of these propositions, we use  
\begin{lemma}
\label{lemwr}
Let $v\in W(R)$ such that $v_R(\bar{v})\le 1$.
If $x\in B^+_{\mrm{cris}}$ satisfies $vx\in W(R)$, then we have $x\in W(R)$.
\end{lemma}
\begin{proof}
This is a generalization of Lemma 3.2.2 of \cite{Li3} but almost the same proof 
can be applied to our setting.
We only give one  remark that 
$E(u)$ is contained in $vW(R)+pW(R)$ by the condition $v_R(\bar{v})\le 1=v_R(E(\bar{u}))$,
and thus we can write $E(u)^{i+1}=p^{i+1}b_i +vw_i$ by some $b_i,w_i\in W(R)$. 
\end{proof}

\begin{proof}[Proof of Proposition \ref{lem0}]
(1) By definition of $N$,
we have $\vphi^{-(N-1)}(gu-u)\in I^{[1]}W(R)$ and $\vphi^{-N}(gu-u)\notin \mrm{Fil}^1W(R)$
(cf.\ Lemma 2.1.3 of \cite{CL}).
By the condition $\vphi^{-(N-1)}(gu-u)\in I^{[1]}W(R)$ and the fact that $\vphi(\mft')$
is a generator of $I^{[1]}W(R)$,
we have  $gu-u=\vphi^N(\mft')v_g$ for some $v_g\in W(R)$, which shows (1)-(i).
Taking $\vphi$ to both sides of this equality, we have 
\begin{equation}
\label{eq1:gu-u}
\vphi(gu-u)=\vphi^{N+1}(\mft')\vphi(v_g)=\vphi^N(\mft')\vphi^N(E(u))\vphi(v_g).
\end{equation}
On the other hand, the equation $\vphi(u)=f(u)$ implies 
\begin{equation}
\label{eq2:gu-u}
\vphi(gu-u)=(gu-u)\tilde{w}_g=\vphi^N(\mft')v_g\tilde{w}_g
\end{equation}
where $\tilde{w}_g=\sum^p_{i=1}a_i(gu^i-u^i)/(gu-u)\in W(R)$.
By \eqref{eq1:gu-u} and \eqref{eq2:gu-u},
we obtain 
\begin{equation}
\label{eq3:gu-u}
v_g\tilde{w}_g=\vphi^N(E(u)) \vphi(v_g).
\end{equation}
Hence we have $\vphi^{-N}(v_g)\vphi^{-N}(\tilde{w}_g)\in \mrm{Fil}^1W(R)$.
Here we note that $\vphi^{-N}(v_g)$ is not contained in $\mrm{Fil}^1W(R)$ since
$\mft'\vphi^{-N}(v_g) =\vphi^{-N}(gu-u)\notin \mrm{Fil}^1W(R)$.
Thus we obtain $\vphi^{-N}(\tilde{w}_g)\in  \mrm{Fil}^1W(R)$.
Since $E(u)$ is a generator of $\mrm{Fil}^1W(R)$ (cf.\ Lemma 2.1.3 of \cite{CL}),
we obtain $\tilde{w}_g=\vphi^N(E(u))w_g$ for some $w_g\in W(R)$.
By \eqref{eq3:gu-u}, we obtain $\vphi(v_g)= v_gw_g$, which shows (1)-(ii). 

Finally we show (1)-(iii).
Take any $x=\vphi(gu-u)y\in \vphi(gu-u)B^+_{\mrm{cris}}\cap W(R)$.
We have 
\begin{align*}
x & = \vphi(gu-u)y=\vphi^{N+1}(\mft')\vphi(v_g)y=
\vphi^{N+1}(\mft')v_gw_gy\\
&=\vphi^N(E(u))\vphi^{N}(\mft')v_gw_g y
=\vphi^N(E(u))\vphi^{N-1}(E(u))\vphi^{N-1}(\mft')v_gw_g y\\
& = \cdots = \vphi^N(E(u))\cdots \vphi(E(u))\cdot E(u)\mft' v_gw_g y\\
& =E(u)\mft' v_gw_g^h  z
\end{align*} 
where $z:= \vphi^N(E(u))\cdots \vphi(E(u)) w_g^{1-h}y\in B^+_{\mrm{cris}}$.
Note that we have $v_R(\overline{E(u)})=ev_R(\bar{u})=1$.
By the equality $\vphi(\mft')=E(u)\mft'$, we have $v_R(\bar{\mft}')=1/(p-1)\le 1$.
It follows from Proposition \ref{val} and the equality $gu-u=\vphi^N(\mft')v_g$
that we have  $v_R(\bar{v}_g)=(j_0-1)/(e(p-1))\le 1$.
Furthermore, by the equality $\vphi(v_g)=v_gw_g$, we also see
$v_R(\bar{w}_g^h)=h(j_0-1)/e\le 1$.
Hence it follows from Lemma \ref{lemwr} and  $E(u)\mft' v_gw_g^h z=x\in W(R)$
that 
we have $z\in W(R)$.
Therefore, we obtain $x=\vphi(\mft')v_gw_g^hz\in v_gw_g^hI^{[1]}W(R)$
as desired. 

\noindent
(2) Since $v_R(\bar{\mft}')=1/(p-1)$, 
$v_R(\bar{v}_g)=(j_0-1)/(e(p-1))$
and  $v_R(\bar{w}_g^h)=h(j_0-1)/e$,
the result follows from (1)-(iii) immediately. 
\end{proof}

\begin{proof}[Proof of Proposition \ref{lem1}] 
The assertion (3) follows from (2) immediately, and thus it suffices to show (1) and (2). 
By the assumption (i),  we have $gu-u=uv_g$ for some $v_g\in W(R)$.
By Lemma 2.3.2 of \cite{CL} and the assumption (ii), 
we see $v_g\in I^{[1]}W(R)$, which shows (1).
Take any $x=\vphi(gu-u)y\in \vphi(gu-u)B^+_{\mrm{cris}}\cap W(R)$.
Writing $v_g=\vphi(\mft')v'_g$ with some $v'_g\in W(R)$,
we have 
$$
x=\vphi(uv_g)y=\vphi(u)\vphi(\vphi(\mft'))\vphi(v'_g)y=
u_fE(u)\mft'\left(\frac{f(u)}{u_f}\vphi(E(u))\vphi(v'_g)y\right).
$$ 
Put $z=(f(u)u_f^{-1})\vphi(E(u))\vphi(v'_g)y$, which is an element of $B^+_{\mrm{cris}}$.
Note that we have $v_R(\bar{u})=1/e \le 1$,
$v_R(\bar{\mft}')=1/(p-1)\le 1$ and $u_fE(u)\mft' z=x\in W(R)$.
Hence it follows from Lemma \ref{lemwr} that 
we have $z\in W(R)$.
Therefore, we obtain $x=u_f\vphi(\mft')z\in u_fI^{[1]}W(R)$. 

\end{proof}

The above propositions allow us to show the existence of ``good'' $(\vphi,G)$-modules
which correspond to objects of $\mrm{Rep}^{r,\mrm{cris}}_{\mrm{tor}}(G)$.
For the case $r=1$, we have

\begin{corollary}
\label{rep:mod:gen}
Assume  $v_p(a_1)>1$ and the condition $(P)$. 
Let $j_0$ be the minimum integer $1\le j\le p$ such that 
$v_p(ja_j)=1$. 
Put $h=0$ $($resp.\ $h=1)$ if $e<j_0-1$ $($resp.\ $e\ge j_0-1)$.
Let $\alpha\in W(R)\smallsetminus pW(R)$ such that 
$v_R(\bar{\alpha})\le (j_0-1)p^h/(e(p-1))$.
Let $T$ be an object of $\mrm{Rep}^{1,\mrm{cris}}_{\mrm{tor}}(G)$ such that $pT=0$.
Then there exists a $(\vphi,G)$-module 
$\hat{\mfM}\in \Mod^{1,G}_{\mfS_{\infty}}(\alpha)$ killed by $p$ such that 
$T\simeq \hat{T}(\hat{\mfM})$. 
%
\end{corollary}

\begin{proof}
Take an exact sequence 
$0\to L_1\to L_2\to T\to 0$
of representations of $G$, where
$L_1\subset L_2$ are $G$-stable $\mbb{Z}_p$-lattices 
in a crystalline $\mbb{Q}_p$-representation of $G$ with Hodge-Tate weights in $[0,1]$.
Take a morphism $i\colon \hat{\mfL}_2\to \hat{\mfL}_1$
in $\Mod^{1,\hat{G},\mrm{cris}}_{\mfS}$ which corresponds to
the injection $L_1\hookrightarrow L_2$ via Theorem \ref{MT}.
We regard $\hat{\mfL}_1$ and $\hat{\mfL}_2$
as $(\vphi,G)$-modules by a canonical way. 
It is not difficult to check that the map 
$\mfL_2\to \mfL_1$ of underlying Kisin modules of $i$ is injective, 
and thus we may regard $\hat{\mfL}_2$ as a sub $(\vphi,G)$-module
of $\hat{\mfL}_1$.
Put $\mfM=\mfL_1/\mfL_2$. 
It follows from Proposition \ref{BASIC2} that 
$\mfM$ is an object of $\Mod^1_{\mfS_{\infty}}$. 
Furthermore, we can naturally equip $\mfM$ with a 
$(\vphi,G)$-module structure; we denote it by $\hat{\mfM}$.
By construction, 
we have an exact sequence 
$0\to \hat{\mfL}_2\to \hat{\mfL}_1\to \hat{\mfM}\to 0$
of  $(\vphi,G)$-modules.
It follows from (the proof of) Lemma 3.1.4 of \cite{CL2} 
that this exact sequence induces $0\to L_1\to L_2\to T\to 0$.
We note that $\mfM[1/p]$ is an \'etale $\vphi$-module 
corresponding to  $T|_{G_{\underline{\pi}}}$, and thus $M$ is killed by $p$ (see the isomorphism (3.2.1) of \cite{CL}).
In particular, $\mfM$ is killed by $p$.
Combining this with the fact that $\hat{\mfL}_1$ and $\hat{\mfL}_2$
are objects of $\Mod^{1,\hat{G},\mrm{cris}}_{\mfS}$,
 it follows from Proposition \ref{lem0} 
that $\hat{\mfM}$ is an object of $\hat{\mfM}\in \Mod^{1,G}_{\mfS_{\infty}}(\alpha)$.
\end{proof}

Next we consider general $r$.

\begin{corollary}
\label{condition}
Assume the following conditions.
\begin{itemize}
\item[${\rm (i)}$] $gu\in uW(R)$ for any $g\in G$.
\item[${\rm (ii)}$] $f^{(n)}(\pi)\not=0$ for any $n\ge 1$.
\item[${\rm (iii)}$] $v_p(a_1)>\mrm{max}\{r,1\}$.
\end{itemize}
Let $\hat{\mfM}$ be a free $(\vphi,\hat{G})$-module of height $r$. 
Then the following are equivalent.

\noindent
$(1)$ $\hat{\mfM}$ is an object of $\Mod^{r,\hat{G},\mrm{cris}}_{\mfS}$.

\noindent
$(2)$ $g(1\otimes x)-(1\otimes x)\in  u_fI^{[1]}W(R)\otimes_{\vphi,\mfS} \mfM$
for any $g\in G$ and $x\in \mfM$.
\end{corollary}

\begin{proof}
Note that $u_f$ is divided by $u$, and thus $\vphi^n(u_f)/p^{nr}$ converges to zero 
$p$-adically in $B^+_{\mrm{cris}}$ by Lemma 2.2.2
of \cite{CL}.
Thus the result follows from Theorem \ref{MT} (1) and Proposition \ref{lem1}.
\end{proof}

\begin{corollary}
\label{rep:mod}
Assume the following conditions.
\begin{itemize}
\item[${\rm (i)}$] $gu\in uW(R)$ for any $g\in G$.
\item[${\rm (ii)}$] $f^{(n)}(\pi)\not=0$ for any $n\ge 1$.
\item[${\rm (iii)}$]  $v_p(a_1)>\mrm{max}\{r,1\}$.
\end{itemize}
Let $T$ be an object of $\mrm{Rep}^{r,\mrm{cris}}_{\mrm{tor}}(G)$.
Then there exists a $(\vphi,G)$-module 
$\hat{\mfM}\in \Mod^{r,G}_{\mfS_{\infty}}(u_f)$ such that $T\simeq \hat{T}(\hat{\mfM})$. 

Moreover, we have the following:
Suppose that 
we have an exact sequence 
$$
(\#)\quad 0\to L_1\to L_2\to T\to 0
$$
of representations of $G$, where
$L_1\subset L_2$ are $G$-stable $\mbb{Z}_p$-lattices 
in a crystalline $\mbb{Q}_p$-representation of $G$ with Hodge-Tate weights in $[0,r]$.
Then there exist $\hat{\mfL}_1, \hat{\mfL}_2\in \Mod^{r,G}_{\mfS}(u_f)$,
$\hat{\mfM}\in \Mod^{r,G}_{\mfS_{\infty}}(u_f)$ and 
an exact sequence
$$
(\ast)\quad 0\to \hat{\mfL}_2\to \hat{\mfL}_1\to \hat{\mfM}\to 0
$$
of  $(\vphi,G)$-modules which induces $(\#)$.
\end{corollary}

\begin{proof}
The  proof is almost the same as that of Corollary \ref{rep:mod:gen}.
We only give a remark that 
$\hat{\mfL}_1$ and $\hat{\mfL}_2$ in the present situation
are objects of $\Mod^{r,G}_{\mfS}(u_f)$
by Corollary \ref{condition}, and thus 
$\hat{\mfM}$ is an object of $\Mod^{r,G}_{\mfS_{\infty}}(u_f)$.

\end{proof}

Now we are ready to prove Theorems \ref{FFT:r=1} and \ref{FFT}.
We essentially follow the method of \cite{Oz2}.

\begin{proof}[Proof of Theorem \ref{FFT}]
The goal is to show the equality  
\begin{equation}
\label{Meq1}
\mrm{Hom}_{G}(T,T')= \mrm{Hom}_{G_{\underline{\pi}}}(T,T')
\end{equation} 
for any $T,\ T'\in \mrm{Rep}^{r,\mrm{cris}}_{\mrm{tor}}(G)$.\\

\noindent
\underline{{\bf STEP 1.}}\quad We reduce a  proof to the case where $k=\overline{k}$.
Assume that the theorem holds when $k=\overline{k}$ and consider general cases.
We denote by $L$ and $H$ 
the completion of the maximal unramified extension of $K$
and the absolute Galois group of $L$, respectively.
We identify the inertia subgroup $I$ of $G$ with $H$.
We set $L_{\underline{\pi}}:=\bigcup_{n\ge 0} L(\pi_n)$
and denote by $H_{\underline{\pi}}$ the absolute Galois group of  $L_{\underline{\pi}}$.
We remark that
$L_{\underline{\pi}}$ is an $f$-iterate extension of $L$
since $\pi$ is a uniformizer of $L$.

Let $f\colon T\to T'$ be a $G_{\underline{\pi}}$-equivalent homomorphism.
Since $T|_H$ and $T'|_H$ are objects of 
$\mrm{Rep}^{r,\mrm{cris}}_{\mrm{tor}}(H)$ and $f$ 
commutes with $H_{\underline{\pi}}$, the assumption above 
implies that $f$ is $H$-equivalent.
Since the extension $K_{\underline{\pi}}/K$ is a totally ramified pro-$p$-extension,
we know that $H$ and $G_{\underline{\pi}}$ topologically generates $G$.
Hence $f$ commutes with $G$.\\

\noindent
\underline{{\bf STEP 2.}}\quad  We reduce a  proof to the case where 
$T$ is irreducible.
Assume that the equality \eqref{Meq1} holds when $T$ is irreducible and consider general cases.
Since the category $\mrm{Rep}^{r,\mrm{cris}}_{\mrm{tor}}(G)$ is stable under subquotients  and direct sums
in $\mrm{Rep}_{\mrm{tor}}(G)$ (cf.\ Lemma 4.19 of \cite{Oz2}), it is an exact category in the sense of Quillen
\cite[Section 2]{Qu}.
Hence short exact sequences in $\mrm{Rep}^{r,\mrm{cris}}_{\mrm{tor}}(G)$ give rise
to exact sequences of Hom's and $\mrm{Ext}^1$'s in the usual way.
Thus a standard  {\it d\'evissage}  argument
(with respect to a Jordan-H\"older sequence of $T$)  
reduces a  proof to the case where  $T$ is irreducible.\\

\noindent
\underline{{\bf STEP 3.}}\quad By Steps 1 and 2,
it suffices to show the equality \eqref{Meq1}
under the conditions that $k=\overline{k}$ and $T$ is irreducible. 
Now we assume these conditions.

First we claim that $T|_{G_{\underline{\pi}}}$ is irreducible.
Let $W$ be a $G_{\underline{\pi}}$-stable submodule of $T$.
Since $T$ is irreducible, 
the wild inertia subgroup $I^{\mrm{w}}$ of $G$ 
acts on $T$ trivial.
In particular, the $I^{\mrm{w}}$-action on $T$ preserves $W$.  
Since  $G_{\underline{\pi}}$ and $I^{\mrm{w}}$
topologically generates $G$,
the irreducibility of $T$ implies that $W$ is $0$ or $T$. Thus the claim follows.

By Corollary \ref{rep:mod},
there exist $(\vphi,G)$-modules $\hat{\mfM},\hat{\mfM}'\in \Mod^{r,G}_{\mfS_{\infty}}(u_f)$
such that $T\simeq \hat{T}(\hat{\mfM})$ and $T'\simeq \hat{T}(\hat{\mfM}')$.
Then we have $T|_{G_{\underline{\pi}}}\simeq T_{\mfS}(\mfM)\simeq T_{\mfS}(\mrm{Max}^r(\mfM))$.
By Theorem \ref{MAX} (5) and the condition that $T|_{G_{\underline{\pi}}}$ is irreducible,
we know that $\mrm{Max}^r(\mfM)$ is a simple object in the abelian category
$\mrm{Max}^r_{\mfS_{\infty}}$.
By Proposition \ref{simple2} and the assumption $k=\overline{k}$,
there exists an sequence $\mfn\in \mcal{S}^r_{\mrm{max}}$ such that 
$\mfM(\mfn)\simeq \mrm{Max}^r(\mfM)$.
We note  that the ideal 
$u_fI^{[1]}W(R)$ of $W(R)$ is generated by $u_f\vphi(\mft)$ and 
$v_R(u_f\vphi(\mft)\ \mrm{mod}\ p)= n_f/e+p/(p-1)\le p+p/(p-1) = p^2/(p-1)$.
It follows from  Theorem \ref{simple:str} that 
there exists a (unique) $(\vphi,G)$-module $\hat{\mfM}(\mfn)\in \Mod^{r,G}_{\mfS_{\infty}}(u_f)$
with underlying Kisin module $\mfM(\mfn)$.
Then we have an isomorphism
$T|_{G_{\underline{\pi}}}\simeq \hat{T}(\hat{\mfM}(\mfn))|_{G_{\underline{\pi}}}$.
By this isomorphism,
we know that $\hat{T}(\hat{\mfM}(\mfn))|_{G_{\underline{\pi}}}$ is irreducible
since $T|_{G_{\underline{\pi}}}$ is irreducible.
Hence $\hat{T}(\hat{\mfM}(\mfn))$ is irreducible as a representation of $G$.
In particular, $T$ and $\hat{T}(\hat{\mfM}(\mfn))$ are tame. 
Since $G_{\underline{\pi}}$ and $I^{\mrm{w}}$ topologically generates
$G$, 
the isomorphism
$T|_{G_{\underline{\pi}}}\simeq \hat{T}(\hat{\mfM}(\mfn))|_{G_{\underline{\pi}}}$
is in fact $G$-equivalent.
We consider the following commutative diagram.
\begin{center}
$\displaystyle \xymatrix{
\mrm{Hom}_G(T,T')\ar@{^{(}->}[rr] & &   
\mrm{Hom}_{G_{\underline{\pi}}}(T,T') \\
\mrm{Hom}(\hat{\mfM}',\hat{\mfM}(\mfn)) \ar^{\hat{T}}[u] \ar[r] &
\mrm{Hom}_{\mfS,\vphi}(\mfM',\mfM(\mfn)) \ar^{\mrm{Max}^r\quad \ \ }[r] & 
\mrm{Hom}_{\mfS,\vphi}(\mrm{Max}^r(\mfM'),\mfM(\mfn))
\ar^{T_{\mfS}}[u]. 
}$
\end{center}
Here, we recall that we have $v_R(\bar{u}_f)=n_f/e>p(r-1)/(p-1)$.
Hence the first arrow in the bottom line, obtained by forgetting $G$-actions,
is bijective by Theorem \ref{FFT2}.
Since $\mfM(\mfn)$ is maximal, it is not difficult to check that  
the second arrow in the bottom line is also bijective. 
Furthermore, the right vertical arrow is also bijective by Theorem \ref{MAX} (5).
Therefore, the top horizontal arrow must be bijective as desired.
This is the end of a proof of Theorem \ref{FFT}.
\end{proof}

\begin{proof}[Proof of Theorem \ref{FFT:r=1}]
The goal is to show the equality  
\begin{equation}
\label{Meq2}
\mrm{Hom}_{G}(T,T')= \mrm{Hom}_{G_{\underline{\pi}}}(T,T')
\end{equation} 
for any $T,\ T'\in \mrm{Rep}^{1,\mrm{cris}}_{\mrm{tor}}(G)$.
The arguments in Steps 1 and 2 just above proceed  also for the present situation.
Thus it suffices to show the equality \eqref{Meq2} 
under the conditions that $k=\overline{k}$ and $T$ is irreducible. 
Put $T''=\mrm{ker}(T'\to T'; x\mapsto px)$.
This is an object of $\mrm{Rep}^{1,\mrm{cris}}_{\mrm{tor}}(G)$
by Lemma 4.19 of \cite{Oz2}. 
Since $pT=0$, we know that  any homomorphism 
$T\to T'$ of $\mbb{Z}_p$-modules have values in $T''$.
Thus, by replacing $T'$ with $T''$, we may assume $pT'=0$.

Take any $\alpha\in W(R)\smallsetminus pW(R)$ such that 
$0< v_R(\bar{\alpha})\le (j_0-1)/(e(p-1))$.
Since $T$ and $T'$ are killed by $p$,
there exist $(\vphi,G)$-modules 
$\hat{\mfM},\hat{\mfM}'\in \Mod^{r,G}_{\mfS_{\infty}}(\alpha)$ killed by $p$
such that $T\simeq \hat{T}(\hat{\mfM})$ and $T'\simeq \hat{T}(\hat{\mfM}')$
by Corollary \ref{rep:mod:gen}.
Now we can use the same arguments of the third paragraph of  Step 3.
\end{proof}

In the case where $er<p-1$, we can improve the assumption (iii) 
of Theorem \ref{FFT}.
\begin{theorem}
Assume the following conditions.
\begin{itemize}
\item[$(i)$] $gu\in uW(R)$ for any $g\in G$.
\item[$(ii)$] $f^{(n)}(\pi)\not=0$ for any $n\ge 1$.
\item[$(iii)$] $v_p(a_1)>r$.
\end{itemize}
Then 
the restriction functor $\mrm{Rep}^{r,\mrm{cris}}_{\mrm{tor}}(G)\to 
\mrm{Rep}_{\mrm{tor}}(G_{\underline{\pi}})$
is fully faithful if $e(r-1)<n_f(p-1)/p$ and $er<p-1$.
\end{theorem}

\begin{proof}
Essentially the same proof of Therem \ref{FFT}  proceeds
but arguments in Step 3 become easier in the case where $er<p-1$.
In fact, any torsion Kisin module of height $r$ is automatically 
maximal by Corollary \ref{Great} (2), and hence 
we do not need arguments of Section \ref{simple:act}.
This is the reason why we can improve the assumption (iii) of Theorem \ref{FFT}.
\end{proof}


\begin{thebibliography}{1000}


\bibitem[Ab]{Ab}
Victor Abrashkin,
\emph{Group schemes of period $p>2$},
      Proc. Lond. Math. Soc. (3), {\bf 101} (2010),
      207--259.      


\bibitem[Be]{Be}
Laurent Berger,
\emph{Iterated extensions and relative Lubin-Tate groups}, 
Ann.\ Math.\ Qu\'ebec {\bf 40} (2016), no. 1, 17--28.

\bibitem[Br1]{Br1}
Christophe Breuil, 
\emph{Une application du corps des normes},
      Compos. Math. {\bf 117} (1999) 
      189--203.

\bibitem[Br2]{Br2}
Christophe Breuil, 
\emph{Integral $p$-adic Hodge theory},
      in Algebraic geometry 2000, Azumino (Hotaka), 
      Adv. Stud. Pure Math., vol. 36, Math. Soc. Japan, 2002,
      51--80.


\bibitem[CD]{CD}
Bryden Cais and Christopher Davis, 
\emph{Canonical cohen rings for norm fields}, 
Int.\ Math.\ Res.\ Not. IMRN (2014),
5473--5517.

\bibitem[CL]{CL}
Bryden Cais and Tong Liu,
\emph{On $F$-crystalline representation}, 
Doc.\  Math., {\bf 21} (2016), 223--270.

\bibitem[CL1]{CL1}
Xavier Caruso and Tong Liu,
\emph{Quasi-semi-stable representations},
      Bull.\ Soc.\ Math.\ France, {\bf 137} (2009), no. 2, 
      185--223.
      
\bibitem[CL2]{CL2}
Xavier Caruso and Tong Liu,
\emph{Some bounds for ramification of $p^n$-torsion
      semi-stable representations}.      
      J.\ Algebra, {\bf 325}  (2011),
      70--96. 


\bibitem[Fo1]{Fo1}
Jean-Marc Fontaine,
\emph{Repr\'esentations p-adiques des corps locaux. I}, 
The Grothendieck
Festschrift, Vol. II, Progr. Math., {\bf 87},
Birkh\"auser Boston, Boston, MA  (1990),
249--309.


\bibitem[Fo2]{Fo2}
Jean-Marc Fontaine,
\emph{Le corps des p\'eriods $p$-adiques}, 
Ast\'erisque, no. 223 (1994),  
59--111.


\bibitem[Ga]{Ga}
Hui Gao,
\emph{Crystalline liftings and weight part of Serre's conjecture},
appear at Israel Journal of Mathematics.




\bibitem[GLS1]{GLS1}
Toby Gee, Tong Liu and David Savitt, 
\emph{The Buzzard-Diamond-Jarvis conjecture for unitary groups},
      J.\ Amer.\ Math.\ Soc.\ {\bf 27} (2014), 389--435.



\bibitem[GLS2]{GLS2}
Toby Gee, Tong Liu and David Savitt, 
\emph{The weight part of Serre's conjecture for $GL(2)$},
Forum of Mathematics, $\pi$ {\bf 3} (2015), e2 (52 pages).


\bibitem[Kim]{Kim}
Wansu Kim, 
\emph{The classification of $p$-divisible groups over 2-adic discrete valuation rings}, 
      Math.\ Res.\ Lett.\ {\bf 19} (2012), no. 1,
      121--141.

\bibitem[Kis]{Kis}
Mark Kisin, 
\emph{Crystalline representations and {$F$}-crystals},
Algebraic geometry and number theory, 
Progr.\ Math.\ {\bf 253},
Birkh\"auser Boston,
Boston, MA  (2006),
459--496.



\bibitem[La]{La}
Eike Lau,
\emph{A relation between Dieudonn\'e displays and crystalline Dieudonn\'e theory}, 
          Algebra Number Theory 8 (2014), 2201-2262.


\bibitem[Li1]{Li1}
Tong Liu,
\emph{Torsion $p$-adic Galois representations and 
      a conjecture of Fontaine}, 
Ann.\ Sci.\ \'Ecole Norm.\ Sup.\ (4) {\bf 40} (2007), no. 4,
633--674.

\bibitem[Li2]{Li2}
Tong Liu,
\emph{A note on lattices in semi-stable representations}, 
      Math.\ Ann.\ {\bf 346} (2010),
117--138.

\bibitem[Li3]{Li3}
Tong Liu,
\emph{The correspondence between Barsotti-Tate groups and Kisin modules when $p=2$}, 
Journal de Th\'eroie des Nombres de Bordeaux, {\bf 25}, no. 3 (2013),
661--676.

\bibitem[Li4]{Li4}
Tong Liu,
\emph{Compatibility of Kisin modules for different uniformizers},
       appear at Journal f\"ur die reine und angewandte Mathematik.


\bibitem[Oz1]{Oz1}
Yoshiyasu Ozeki,
\emph{Torsion representations arising from $(\vphi,\hat{G})$-modules}, 
J.\ Number Theory, No. 133 (2013), 3810-3861.

\bibitem[Oz2]{Oz2}
Yoshiyasu Ozeki,
\emph{On Galois equivariance of homomorphisms between torsion crystalline representations},
appear at the Nagoya Mathematical Journal.



\bibitem[Qu]{Qu}
Daniel Quillen,
\emph{Higher algebraic $K$-theory: I},
      in Algebraic $K$-theory, I: Higher $K$-theories (Seattle, 1972), 
      Lecture Notes in Math.\ {\bf 341}, Springer-Verlag, New York, 
      1973, 85--147.

\bibitem[Wi]{Wi}
Jean-Pierre Wintenberger, 
\emph{Le corps des normes de certaines extensions infinies de corps
     locaux; applications},
Ann.\ Sci.\ \'Ecole Norm.\ Sup.\ (4) {\bf 16}  (1983),
59--89.


\end{thebibliography}
\end{document}